\documentclass[10pt,journal,compsoc]{IEEEtran}

\ifCLASSOPTIONcompsoc
\usepackage[nocompress]{cite}
\else
\usepackage{cite}
\fi
\ifCLASSINFOpdf
\else
\fi

\usepackage{diagbox}

\usepackage{array}
\usepackage{microtype}
\usepackage{graphicx}
\usepackage{bm}
\usepackage{algorithm}
\usepackage{algorithmic}
\usepackage{multirow}
\usepackage{amsthm}
\usepackage{amsmath}
\usepackage{subfigure}
\usepackage{amssymb}
\usepackage{bbding}
\usepackage{pifont}
\usepackage{wasysym}
\usepackage{booktabs} % for professional tables
\usepackage{makecell}
\usepackage{subfigure}
\usepackage{bbding}
\usepackage{balance}
\usepackage{float}
\usepackage{subeqnarray}
\usepackage{cases}
\usepackage{tabularx}
\usepackage{color}
\usepackage{colortbl,url}
\usepackage{threeparttable}
\usepackage[table,xcdraw,dvipsnames]{xcolor}  % 如果你已有这一行，确保含有 dvipsnames 或 HTML

\newtheorem{theorem}{Theorem}[section]
\newtheorem{asu}{Assumption}[section]

\usepackage{latexsym,amsmath,amssymb,amsthm}
\graphicspath{{figures/}}
\def\D{\mathcal{D}}
\def\x{\mathbf{x}}
\def\y{\mathbf{y}}
\def\z{\mathbf{z}}
\def\u{\mathbf{u}}

\def\v{\mathbf{v}}

\def\c{\mathbf{c}}

\def\Y{\mathcal{Y}}
\def\S{\mathcal{S}}
\def\X{\mathcal{X}}
\def\T{\mathcal{T}}

\def\D{\mathcal{D}}

\definecolor{mygreen}{HTML}{008B47}
\definecolor{myred}{HTML}{c00000}
\definecolor{mypurple}{HTML}{7030A0}
\begin{document}
	%
	% paper title
	% Titles are generally capitalized except for words such as a, an, and, as,
	% at, but, by, for, in, nor, of, on, or, the, to and up, which are usually
	% not capitalized unless they are the first or last word of the title.
	% Linebreaks \\ can be used within to get better formatting as desired.
	% Do not put math or special symbols in the title.
\title{Augmenting Iterative Trajectory for Bilevel Optimization: Methodology, Analysis and Extensions}

\author{Risheng Liu,~\IEEEmembership{Member,~IEEE,}
	Yaohua Liu,
	Shangzhi Zeng, Jin Zhang% <-this % stops a space
		\IEEEcompsocitemizethanks{\IEEEcompsocthanksitem R. Liu is with the School of Software, Dalian University of Technology, Dalian, Liaoning, P.R., China. E-mail: rsliu@dlut.edu.cn. \protect 
		\IEEEcompsocthanksitem Y. Liu is with the School of Computing and Data Science, University of Hong Kong, Hong Kong SAR, China.  E-mail:  liuyaohua\_918@163.com. \protect 
		\IEEEcompsocthanksitem S. Zeng is with the National Center for Applied Mathematics, and Department of Mathematics, Southern University of Science and Technology, Shenzhen, Guangdong, P.R., China. E-mail: zengsz@sustech.edu.cn. \protect
		\IEEEcompsocthanksitem J. Zhang is with the Department of Mathematics, Southern University of Science and Technology, and National Center for Applied Mathematics, Shenzhen, Guangdong, P.R., China. (Corresponding author, E-mail: zhangj9@sustech.edu.cn.) \protect
	}}
%	\thanks{Manuscript received Mar 26, 2023; revised Oct 08, 2025.}}

% The paper headers
\markboth{Journal of \LaTeX\ Class Files,~Vol.~14, No.~8, August~2015}%
{Shell \MakeLowercase{\textit{et al.}}: Bare Demo of IEEEtran.cls for Computer Society Journals}

\IEEEtitleabstractindextext{%
	\begin{abstract}
		
In recent years, there has been a surge of machine learning applications developed with hierarchical structure, which can be approached from Bi-Level Optimization (BLO) perspective. However, most existing gradient-based methods overlook the interdependence between hyper-gradient calculation and Lower-Level (LL) iterative trajectory, focusing solely on the former. Consequently, convergence theory is constructed with restrictive LL assumptions, which are often challenging to satisfy in real-world scenarios. In this work, we thoroughly analyze the constructed iterative trajectory, and highlight two deficiencies, including empirically chosen initialization and default use of entire trajectory for hyper-gradient calculation. To address these issues, we introduce two augmentation techniques including Initialization Auxiliary (IA) and Pessimistic Trajectory Truncation (PTT), and investigate various extension strategies such as prior regularization, different iterative mapping schemes and acceleration dynamics to construct Augmented Iterative Trajectory (AIT) for corresponding BLO scenarios (e.g., LL convexity and LL non-convexity). Theoretically, we provide convergence analysis for AIT and its variations under different LL assumptions, and establish the convergence analysis for BLOs with non-convex LL subproblem. Finally, we demonstrate the effectiveness of AIT through three numerical examples, typical learning and vision applications (e.g., data hyper-cleaning and few-shot learning) and more challenging tasks such as neural architecture search.
		
	\end{abstract}
	
	% Note that keywords are not normally used for peerreview papers.
	}

% make the title area
\maketitle

% To allow for easy dual compilation without having to reenter the
% abstract/keywords data, the \IEEEtitleabstractindextext text will
% not be used in maketitle, but will appear (i.e., to be "transported")
% here as \IEEEdisplaynontitleabstractindextext when the compsoc 
% or transmag modes are not selected <OR> if conference mode is selected 
% - because all conference papers position the abstract like regular
% papers do.
\IEEEdisplaynontitleabstractindextext
% \IEEEdisplaynontitleabstractindextext has no effect when using
% compsoc or transmag under a non-conference mode.

% For peer review papers, you can put extra information on the cover
% page as needed:
% \ifCLASSOPTIONpeerreview
% \begin{center} \bfseries EDICS Category: 3-BBND \end{center}
% \fi
%
% For peerreview papers, this IEEEtran command inserts a page break and
% creates the second title. It will be ignored for other modes.
\IEEEpeerreviewmaketitle

\IEEEraisesectionheading{\section{Introduction}\label{sec:introduction}}
% Computer Society journal (but not conference!) papers do something unusual
% with the very first section heading (almost always called "Introduction").
% They place it ABOVE the main text! IEEEtran.cls does not automatically do
% this for you, but you can achieve this effect with the provided
% \IEEEraisesectionheading{} command. Note the need to keep any \label that
% is to refer to the section immediately after \section in the above as
% \IEEEraisesectionheading puts \section within a raised box.

% The very first letter is a 2 line initial drop letter followed
% by the rest of the first word in caps (small caps for compsoc).
% 
% form to use if the first word consists of a single letter:
% \IEEEPARstart{A}{demo} file is ....
% 
% form to use if you need the single drop letter followed by
% normal text (unknown if ever used by the IEEE):
% \IEEEPARstart{A}{}demo file is ....
% 
% Some journals put the first two words in caps:
% \IEEEPARstart{T}{his demo} file is ....
% 
% Here we have the typical use of a "T" for an initial drop letter
% and "HIS" in caps to complete the first word.

\IEEEPARstart{R}{ecently}, a variety of emerging deep learning applications have adopted hierarchical optimization models, including Hyperparameter Optimization~\cite{liu2021investigating,shaban2019truncated}, Meta Learning~\cite{nichol2018first,finn2017model}, Neural Architecture Search~\cite{wu2019fbnet,hu2020tf}, and Reinforcement Learning~\cite{liu2018darts,zhang2020bi,wang2020global}. These applications have achieved state-of-the-art performance across diverse domains, reflecting the effectiveness of hierarchical modeling. With proper reformulation, they can all be cast as Bi-Level Optimization (BLO) problems, which optimize two nested subproblems in the form:
\begin{equation}\label{blo_problem}
	\min _{\mathbf{x} \in \mathcal{X}} F(\mathbf{x}, \mathbf{y}), \text { s.t. } \mathbf{y} \in \mathcal{S}(\mathbf{x}):=\arg \min _{ \mathbf{y} \in \mathcal{Y}} f(\mathbf{x}, \mathbf{y}),
\end{equation}
where $\x \in\mathbb{R}^{m}$ and $ \y \in \mathbb{R}^n$ correspond to the Upper-Level (UL) and Lower-Level (LL) variables respectively, the UL objective $F$ and LL objective $f$ are continuously differentiable functions and  $\mathcal{S}(\mathbf{x})$ is the LL solution set for all $\x\in\mathcal{X}$. 

In general, the BLO problem can be considered as a leader-follower game process, where the follower ( i.e., the LL subproblem) continuously responds to the decision $\x$ of the Leader (i.e., the UL subproblem). From the optimistic BLO perspective, we assume that the obtained LL solution $\y\in\mathcal{S}(\x)$ always leads to the optimal solution of UL objective for any given $\x$, then the original BLO in Eq.\eqref{blo_problem} can be reduced to a simple single-level problem as 
\begin{equation}\label{single_reform}
	\min_{\x \in \X} \varphi(\x):= \inf_{\y\in\mathcal{S}(\x)} F(\x,\y),
\end{equation} 
where $\varphi$ is the value function of a simple bilevel problem w.r.t. $\y$ for any given $\x$. Then $\x$ and $\y$ have the cooperative relationship towards minimizing the UL subproblem. Based on this form, various branches of methods~\cite{moore2010bilevel,rajeswaran2019meta,liu2020generic,ye2021difference} have explored how to optimize $\varphi(\x)$ with respective to $\x$ and $\y$ simultaneously, thus solve the above applications captured by this hierarchical structure with high accuracy and efficiency.

Typically, BLO problems remain challenging to be solved caused by the nature of sophisticated dependence between UL and LL variables, especially when multiple LL solutions exist in $\mathcal{S}(\x)$. As for classical theory, KKT condition~\cite{zemkoho2020theoretical} has been considered as an efficient tool to solve the above optimization problems. whereas, these types of methods are impractical to be applied in modern machine learning tasks limited by existence of too many multipliers. Meanwhile, Gradient-Based Methods (GBMs)~\cite{franceschi2017forward} have been widely adopted in various learning and vision tasks to handle real-world applications. Practically speaking, the key of solving Eq.~\eqref{single_reform} with GBMs falls to accessing the highly coupled nested gradient caused by iterative optimization to solve the LL subproblem in Eq.~\eqref{blo_problem}.  We denote $\y^{*}(\x)$ as the Best Response mapping of the follower for a given UL variable $\x$. By embedding $\y^{*}(\x)$ into the simple bilevel problem in Eq.~\eqref{single_reform} as $F\left(\x, \y^*(\x)\right)$, the hyper-gradient of $\varphi(\x)$ can be derived by using the chain rule as

\begin{equation}\label{nest_gradient}
	\nabla_{\x}\varphi(\mathbf{x})=\nabla_{\x}F\left(\mathbf{x}, \mathbf{y}^*(\mathbf{x})\right)+G(\mathbf{x}),
\end{equation}
where $\nabla_{\x}F\left(\mathbf{x}, \mathbf{y}^*(\mathbf{x})\right)$ can be regarded as the direct gradient, and $G(\x)=\left(\nabla_{\x} \mathbf{y}^*(\mathbf{x})\right)^{\top} \nabla_{\y} F\left(\mathbf{x}, \mathbf{y}^*(\mathbf{x})\right)$ denotes the indirect gradient. 

Since the calculation of hyper-gradient requires proper approximation of $\y^*(\x)$ and indirect gradient $G(\x)$, classical dynamics-embedded GBMs~\cite{franceschi2017forward}  tend to construct the approximations by introducing dynamical system obtained from iterative optimization steps. Typically speaking, these types of gradient-based BLO methods approximate $\y^*(\x)$ with $\y_{K}(\x)$ constructed by the following optimization dynamics: $
	\y_{k+1}(\x)=\mathcal{T}_{k+1}\left(\mathbf{x}, \mathbf{y}_{k}(\x)\right),\ k=0,\cdots,K-1,$
where $\y_0(\x)=\y_0$ is a fixed given initial value and $\mathcal{T}_{k+1}(\x,\cdot)$ denotes the iterative dynamics mapping at $k$-th step derived with specific updating scheme. Based on the above formulation, we use $\mathcal{T}_{0:K}$ to denote the historical sequence of $\y_{k}(\x)$, $k=0,...,K$, which will be discussed more thoroughly in this work. In essence, $\mathcal{T}_{0:K}$ reflects the iterative trajectory of LL variables to obtain $\y_{K}(\x)$, then the indirect gradient $G(\x)$ can be approximated with this $\y_{K}(\x)$ by 
$\left(\nabla_{\x} \mathbf{y}_{K}(\mathbf{x})\right)^{\top} \nabla_{\y} F\left(\mathbf{x}, \mathbf{y}_{K}(\mathbf{x})\right)$. By this way, these GBMs can be regarded as using gradient descent methods solving following approximation problem of Eq.~\eqref{single_reform}, 
\begin{equation}\label{single_re_reform}
	\min_{\x \in \X} \varphi_{K}(\x):= F(\x,\y_{K}(\x)).
\end{equation} 
It is worth noting that $\mathbf{y}_{K}(\mathbf{x})$ can be regarded as good approximation as long as $\nabla_{\mathbf{y}} f\left(\mathbf{x}, \mathbf{y}_{K}(\mathbf{x})\right)$ uniformly converges to zero w.r.t. $\x\in\mathcal{X}$ and $F(\x,\y_{K}(\x))$ converges to the real BLO objective $\varphi(\x)$. %in which $\nabla_{\x} \mathbf{y}_{K}(\mathbf{x}) $ is usually computed by Automatic Differentiation (AD).  

To ensure the obtained approximation in Eq.~\eqref{single_re_reform}  is good enough to guarantee the convergence of solutions towards the original BLO, restrictive assumptions of LL subproblem, such as LL Singleton (LLS) and LL Convexity (LLC), are usually required by existing GBMs. Whereas, as complex network structure with multiple convolutional blocks and loss functions with specialized regularization terms are widely employed in lots of applications, the LL subproblems always have non-convex property, which implies a gap between the provided theoretical guarantee and complex properties of real world problems.  To deal with these challenges, several works~\cite{liu2020generic,liu2021value} have spared efforts to design new iterative mapping schemes or approximation theory, thus partly relax the theoretical constraints and obtain some new convergence results. Nevertheless, there is still an urgent need for general BLO framework with solid theoretical analysis to cover the non-convex scenarios.    

\subsection{Contributions}

In this work, we propose the Augmented Iterative Trajectory (AIT) framework with two augmentation techniques, including Initialization Auxiliary (IA) and Pessimistic Trajectory Truncation (PTT), and a series of extension strategies to handle the above limitations of existing GBMs. More specifically, we first analyze the vanilla iterative trajectory constructed by the gradient-based BLO scheme,  and highlight two significant deficiencies that have been largely overlooked: (i) The initialization of the iterative trajectory is typically chosen empirically. (ii) The entire trajectory is often used by default for hyper-gradient calculation.

Aiming at the above deficiencies of iterative trajectory, we first introduce IA to optimize the initialization value for LL trajectory simultaneously with the UL variables. We further enhance IA with a proximal prior regularization strategy as a warm start, and explore different iterative dynamic mapping schemes and acceleration dynamics to construct $\textrm{AIT}_{C}$ for BLOs with convex LL subproblems. To handle the challenging non-convex scenario, we further propose PTT operation to dynamically adjust the iterative trajectory for hyper-gradient calculation, and additionally incorporate accelerated gradient schemes to form $\textrm{AIT}_{NC}$ for non-convex BLOs. Theoretically, we conduct detailed convergence analysis of $\textrm{AIT}_{C}$, $\textrm{AIT}_{NC}$ and its variations to support different challenging scenarios, especially for the unexplored non-convex scenario. Finally, we validate AIT and its extension schemes through extensive numerical experiments and evaluations on diverse learning and vision tasks. We summarize our main contributions as follows.
\begin{itemize}
	
	\item We conduct a thorough analysis of the iterative trajectory construction for the gradient-based BLO scheme, which is a significant departure from the existing GBMs. We then identify two limitations of the vanilla iterative trajectory adopted by current GBMs, and introduce our AIT to address BLOs with more relaxed LL assumptions.
	
	\item We first propose IA to guide the optimization of LL variables, and then introduce prior regularization based on IA and explore various iterative dynamic mapping schemes (e.g., Nesterov's acceleration dynamics) to construct $\textrm{AIT}_{C}$ for BLOs with Convex followers. To handle the more challenging non-convex scenarios, we further propose PTT operation to dynamically truncate the trajectory, and investigate acceleration gradient scheme to construct $\textrm{AIT}_{NC}$.
	
	\item We provide detailed theoretical investigation on $\textrm{AIT}_{C}$, $\textrm{AIT}_{NC}$ and corresponding extension strategies (e.g., proximal prior regularization and acceleration gradient scheme) to ensure the convergence property of AIT without LLS even LLC assumption. Notably, we establish the first convergence guarantee for BLO problems with non-convex LL subproblems.
	
\item We conduct extensive numerical experiments under both convex and non-convex LL settings, including ablation studies to validate the proposed extension strategies. In addition, we demonstrate its broad applicability and consistent performance improvements on typical learning and vision applications (e.g., few-shot classification and data hyper-cleaning) and more challenging tasks including neural architecture search and generative adversarial networks.
	
\end{itemize}

\noindent \textbf{Relation to the Preliminary Version.} 
A preliminary version of this work appeared in the NeurIPS conference proceedings~\cite{liu2021towards}. This journal manuscript substantially extends the preliminary version by establishing a unified framework for AIT. \textbf{Methodologically}, we advance the specific techniques into a comprehensive framework and introduce novel extension strategies, including proximal prior regularization and accelerated gradient schemes, to further enhance convergence efficiency. \textbf{Theoretically}, we generalize the analysis via the  “Weak LL property with IA’’ and establish new convergence results characterizing the limiting behavior of stationary points (Theorems~\ref{thm: thm_station} and ~\ref{theorem4.7}), providing rigorous theoretical guarantees absent in the preliminary work. \textbf{Experimentally}, we validate the proposed framework on more challenging tasks, such as neural architecture search and generative adversarial networks.

\section{A Brief Review of Existing Works}\label{related_works}

In this work, we first provide a brief review of previous works to conclude the general gradient-based BLO scheme, and then we discuss the deficiencies of existing LL iterative trajectory, which have limited the application of existing GBMs to more challenging BLO scenarios.  

As mentioned before, existing GBMs have explored different heuristic techniques to calculate the hyper-gradient (i.e., $\nabla_{\x}\varphi(\mathbf{x})$) with designed approximation operations. In one of the most representative class of GBMs~\cite{maclaurin2015gradient,franceschi2018bilevel,liu2022general}, it first constructs dynamical system to track the optimization path of LL variables, and then explicitly calculates the indirect gradient with respect to UL variables based on ITerative Differentiation (ITD). Typically, RHG and FHG~\cite{maclaurin2015gradient,franceschi2018bilevel} are proposed, in which $\y^{*}(\x)$ and thus $\varphi(\x)$ are approximated by finite iterative optimization steps, and then the obtained approximated optimization problems are solved. Whereas, both methods suffer from the huge time and space complexity of calculating indirect gradient along the whole iterative trajectory with AD. Then Shanban et al.~\cite{shaban2019truncated} proposed T-RHG to manually truncate the backpropagation path, thus release the computation burden. For these methods, the iterative dynamics mapping $\mathcal{T}_{k+1}(\x,\cdot)$ used to construct iterative trajectory is specified as the projected gradient descent. Whereas, the convergence analysis of these methods all require the LL subproblem to be strongly convex, which severely restrict the application scenarios. Recently, Liu et al.~\cite{liu2022general} proposed BDA, which aggregates the gradient information of both UL and LL subproblems to construct new $\mathcal{T}_{k+1}(\x,\cdot)$ to approximate $\y^{*}(\x)$ and $\varphi(\x)$, thus successfully overcomes the LLS restriction issue.  When calculating the nested hyper-gradient with ITD based on different forms of $\mathcal{T}_{k+1}(\x,\cdot)$, the gradient-based BLO scheme could be summarized in Alg.~\ref{alg:innerloop}, where $\varphi_{K}(\x)= F(\x,\y_{K}(\x))$, and $\mathtt{Proj}(\cdot)$ denotes the projection operator. It can be observed that the iterative update of UL variables to solve the approximated subproblem (i.e., $\varphi(\x)$) relies on effective iterative trajectory for hyper-gradient calculation (Step~\ref{inner_loop1}-\ref{inner_loop3}).

\begin{algorithm}[h]
	\caption{Gradient-Based BLO Scheme}\label{alg:innerloop}
	\begin{algorithmic}[1]
		\REQUIRE UL iteration $T$ and LL iteration $K$.  
		\STATE Initialize $\x^0$.
		\FOR {$t=0 \rightarrow T-1$}
		\STATE Initialize $\y_0$.
		\FOR {$k=0 \rightarrow K-1$}\label{inner_loop1}
		\STATE \% Update $\y_k$ with $\x^t$.
		\STATE $\y_{k+1}(\x^t)= \mathcal{T}_{k+1}\left(\mathbf{x}^t, \mathbf{y}_{k}(\x^t)\right)$.\label{inner_loop3}
		\ENDFOR
		%\STATE 
		\STATE \% Update $\x^t$ with $\y_{K}(\x^t)$ based on ITD.
		\STATE $\x^{t+1}=\mathtt{Proj}_{\X}(\x^{t}-\alpha_{\x}\nabla_{\x}\varphi_{K}(\x^t) )$. \label{indirect_gradient}
		\ENDFOR
	\end{algorithmic}
\end{algorithm}

Since most ITD-based approaches only established their convergence theory under the LLS restriction, they neither pay attention to the initialization of the iterative trajectory nor carefully consider how to select the trajectory for hyper-gradient calculation. In particular, (i) an improper initialization $\y_{0}(\x)$ for the LL variables may slow down the convergence of both LL and UL subproblems. More importantly, when multiple solutions exist in $\mathcal{S}(\x)$ (i.e., without LLS), the initialization has a significant influence on the UL convergence behavior, and there is no guarantee that the trajectory constructed from the chosen $\y_{0}(\x)$ can converge to a solution in $\mathcal{S}(\x)$ while optimizing the UL objective.  (ii) Moreover, if the LLC assumption does not hold, the subsequences of the iterative trajectory $\mathcal{T}_{0:K}$ no longer enjoy the consistent convergence property as in the convex case. In this situation, using the entire trajectory for hyper-gradient calculation may lead to oscillation or even divergence of the UL optimization. Overall, the vanilla trajectory adopted by ITD-based approaches thus lacks flexibility and becomes the bottleneck of existing gradient-based BLO scheme to cover more challenging scenarios. Though several researches have explored potential techniques, such as aggregating LL and UL gradient information to design new $\mathcal{T}_{k}(\x,\cdot)$ to construct a good approximated BLO problem in Eq.~\eqref{single_re_reform} and successfully relax the theoretical restrictions on LL subproblem to some extent, the situation where the LLC assumption does not hold still needs to be considered.

In contrast to the ITD-based approaches which compute the hyper-gradient based on the constructed LL iterative trajectories, another branch of the methods~\cite{pedregosa2016hyperparameter,rajeswaran2019meta,lorraine2020optimizing} instead adopts the implicit function theorem and derives the hyper-gradient of UL variables by solving a linear system. Whereas, computing Hessian matrix and its inverse is much more expensive and challenging, particularly when ill-conditioned linear system appears. Besides, these methods based on Approximated Implicit Differentiation (AID) require the LL subproblem to be strongly convex, which limits the application scenarios to a great extent.  Another line of work~\cite{LiuLYZZ21} addresses the nonconvex issue by constructing value-function-based smooth approximation problems, while extra assumptions introduced on the constrained LL subproblems still remain cumbersome to handle. As for other recent efforts~\cite{ji2020provably,chen2022single} which have designed specific update formats of UL or LL variables to derive new convergence rate analysis, Hong et al.~\cite{hong2020two} proposed TTSA to introduce cheap estimates of gradients, and update the LL and UL variables simultaneously with SGD and projected SGD. In addition, Khanduri et al.~\cite{khanduri2021momentum} proposed the MSTSA algorithm to estimate the hyper-gradient with assisted momentum.  Quan et.al~\cite{xiao2022alternating} proposed two variants of the alternating implicit gradient SGD approach with improved efficiency to solve equality-constrained BLOs. The above methods also build their theoretical analysis for BLOs with strongly-convex LL subproblems.  

\begin{table}[hbtp]
	\centering
	\caption{We report the LL assumptions required by the mainstream GBMs to derive their convergence results, the technique to compute $\nabla_{\x}\varphi(\mathbf{x})$ and whether they are compatible with the accelerated technique (i.e., A.).  Note that w/ and w/o are abbreviations for with and without, respectively. } 
	\label{tab_conditions}
	\renewcommand\arraystretch{1.4}
	\setlength{\tabcolsep}{1.1mm}{
		\begin{tabular}{|c|c|c|c|c|}
			\hline
			Method & $\nabla_{\x}\varphi(\mathbf{x})$
			%& \textbf{LLS}
			& {w/o LLS} & {w/o LLC}& {w/ A. } \\			
			\hline
			RHG/FHG &\multirow{3}{*}{ITD}& \XSolidBrush& \XSolidBrush & \XSolidBrush\\
			
			T-RHG && \XSolidBrush &\XSolidBrush & \XSolidBrush \\
			BDA	&&\CheckmarkBold &\XSolidBrush & \XSolidBrush\\
			\hline
			CG&\multirow{2}{*}{AID} & \XSolidBrush& \XSolidBrush & \XSolidBrush\\
			Neumann&& \XSolidBrush& \XSolidBrush & \XSolidBrush\\
			\hline
			AIT	&ITD& \CheckmarkBold & \CheckmarkBold & \CheckmarkBold\\
			\hline
			
		\end{tabular}
	}
\end{table}

In Tab.~\ref{tab_conditions}, we compare the required LL assumptions of mainstream GBMs to build their convergence analysis and whether they support the acceleration technique. Based on the above analysis of gradient-based BLO scheme for ITD-based approaches, in the next section, we construct our AIT framework to address the theoretical issues and handle more challenging BLO scenarios. As it is shown, AIT not only derives similar convergence results without the LLS and LLC restriction, but also supports more flexible iterative mapping schemes including the acceleration techniques.

\section{Augmented Iterative Trajectory (AIT)}\label{IAPTT}

Based on the prior analysis of iterative trajectory, we aim to address two key deficiencies of existing gradient-based BLO schemes: their dependence on empirically preselected initializations $\y_0(\x)$ and their default reliance on the last iterate of the LL trajectory for hyper-gradient computation. To mitigate the first issue, we introduce the IA technique, which augments the trajectory by adding an auxiliary variable to the initialization and jointly optimizing it with the UL variables. We further extend IA with strategies such as prior regularization and acceleration dynamics to construct more efficient trajectories under the LLC assumption. To address the second issue, we propose the PTT technique, which dynamically adjusts the trajectory length to stabilize optimization in non-convex settings. Moreover, in the non-convex case, we also incorporate acceleration schemes into AIT to improve convergence efficiency. Together, IA and PTT, along with these extensions, form the foundation of the AIT framework for solving both convex and non-convex BLOs.

\subsection{ BLOs with Convex LL Subproblem}\label{IA}

Typically speaking, the initialization of iterative trajectory, i.e., $\y_{0}(\x)$, is usually chosen based on a fixed manually (randomly or by some specific strategies) set initial value of LL variable, which takes no consideration of the convergence behavior of UL subproblem. As it is mentioned before, when there exist multiple elements in LL subproblem solution set $\mathcal{S}(\mathbf{x})$, performing projected gradient descent steps (i.e., $\mathcal{T}_{k}(\x,\cdot)$ adopted by most GBMs)  to construct iterative trajectory with improper initial value $\y_0(\x)$ usually not necessarily optimizes the UL objective while converging to $\mathcal{S}(\x)$. Besides, an improper initialization position for $\mathcal{T}_{0:K}$ will also slow down the convergence rate of LL subproblem, even when $\mathcal{S}_{\x}$ is a singleton. 

To address the issues mentioned above, we suggest to introduce an Initialization Auxiliary (IA) variable $\z$ as the initialization position of LL variable to  construct $\mathcal{T}_{0:K}$ for approximating LL subproblem as follows,
\begin{equation}\label{yk_IA}
   	\begin{aligned}
   		&\y_0(\x,\z) = \z,   \quad\mathbf{y}_{k+1}(\x,\z)=\mathcal{T}_{k+1}\left(\mathbf{x}, \mathbf{y}_{k}(\x,\z)\right),
   	\end{aligned}
\end{equation} 
which leads to the following approximation problem of Eq.~\eqref{single_reform}: $
	\min_{\x \in \X, \z \in \Y} \varphi_{K}(\x,\z):= F(\x,\y_{K}(\x,\z)).$
It is worth noting that, the iterative trajectory constructed based on this optimization dynamics formulates the approximation $\y_K(\x,\z)$ parameterized by $\x$ and $\z$.  By this form, we avoid fixing initialization position chosen with empirical strategies, and update the initialization auxiliary $\z$ together with $\x$ to continuously search for the “best’’ initial points  of LL variables to calculate $\y_K(\x,\z)$. By taking into consideration of the convergence behavior of UL subproblem, $\mathcal{T}_{k}(\x,\cdot)$ with updated initial value can approach solutions to the BLO problem in Eq.~\eqref{blo_problem}. The idea of IA variable $\z$ is first proposed for the convergence theory~\cite{LiuLYZZ21} of non-convex first-order optimization methods. 

As mentioned above, the initialization auxiliary variable $\z$ is introduced and optimized simultaneously with UL variable $\x$ for helping the LL subproblem dynamics approach the appropriate $\y^*(\x)$ that minimizes UL objective $F$. To further guide the update of the IA variable $\z$ and thus boost the convergence of the LL subproblem dynamics, we propose to introduce a constraint set $\mathcal{Z}$ and an extra prior regularization term $g(\z)$ to the IA variable $\z$, which leads to the following approximation problem to Eq.~\eqref{single_reform}, 

\begin{equation}\label{IA_g_z}
	\min_{\x \in \X, \z \in \mathcal{Z}} \	\psi_{K}(\x,\z)=\varphi_{K}(\x,\z) + g(\z).
\end{equation}
In practice, $g(\z)$ can be specified as penalty on distance between $\z$ and the prior determined value $\z_p$, where $\z_p$ usually serves as some known state of $\z$ that is very close to the optimal LL variable. As one option for machine learning tasks, we could properly define $\z_p$ as the pretrained weights of the LL variables, and verify its effectiveness with BLO applications.

Another feature of this new approximation format is that the iterative dynamics mapping $\mathcal{T}_{k+1}(\x,\cdot)$ for constructing iterative trajectory can be flexibly specified as different forms proposed in existing works. Here, for general BLOs with convex LL subproblem, we suggest applying the aggregation scheme proposed in BDA~\cite{liu2022general} to construct $\mathcal{T}_{k+1}(\x,\cdot)$  as follows
\begin{equation}\label{eq:improved-lower}
	\begin{aligned}
		&\T_{k+1}\left(\x,\y_k(\mathbf{x},\z)\right) \\
		= &\mathtt{Proj}_{\Y} \left( \y_k(\mathbf{x},\z) -\left( \mu \alpha_k\mathbf{d}^{{F}}_k(\mathbf{x},\z)+(1-\mu)\beta_k\mathbf{d}^{{f}}_k(\mathbf{x},\z)\right) \right),
	\end{aligned}
\end{equation}
where $\mathbf{d}^{{F}}_k(\mathbf{x},\z) = s_u\nabla_{\mathbf{y}} F(\mathbf{x}, \mathbf{y}_{k}(\mathbf{x},\z))$ and $\mathbf{d}^{{f}}_k(\mathbf{x},\z) = s_l\nabla_{\mathbf{y}} f(\mathbf{x},\mathbf{y}_{k}(\mathbf{x},\z))$ 
denote the descent directions of $F$ and $f$ with the introduced IA, respectively. 
Here, $s_u, s_l > 0$ are the step size parameters, $\mu \in (0,1)$ and $\alpha_k, \beta_k \in (0,1]$ are the aggregation weights at $k$-th iteration.

Specially, for the case where the LLS condition is satisfied, following the vanilla update format in RHG, $\mathcal{T}_{k+1}\left(\mathbf{x}, \cdot \right)$ could be simply chosen as
\begin{equation}\label{yk_def}
	\begin{aligned}
		\mathcal{T}_{k+1}\left(\mathbf{x}, \mathbf{y}_{k}(\x,\z)\right)= \mathtt{Proj}_{\Y} ( \y_k(\x,\z)- \alpha_{\y}^k\nabla_{\y}f(\x,\y_k(\x,\z)) ),
\end{aligned}
\end{equation}
where $\alpha_{\y}^k$ denotes the step size parameter. By this means, embedding the IA technique still helps optimize the initialization to construct the iterative trajectory, thus further improves the convergence behavior of BLOs under LLS assumption. 

Besides, our new scheme with IA technique also supports to consistently improve the existing methods with embedded acceleration dynamics. As one of the practical implementation, we could embed the Nesterov's acceleration methods~\cite{beck2009fast} with our IA technique, which has been widely used for convex optimization problem to accelerate the convergence of gradient descent, then we obtain the accelerated version as 
\begin{equation}\label{yk_def2}
	\begin{cases}
	\begin{aligned}
		&\y_0(\x,\z) = \z, \quad t_0 =1, \quad  t_{k+1} = {\left(1+\sqrt{1+4t_{k}^2} \right)}/{2},  \\
		& \u^{k+1}(\x,\z) = \y^{k+1}(\x,\z) + \left(\frac{t_{k}-1}{{t_{k+1}}}\right)(\y^{k+1}(\x,\z) - \y^{k}(\x,\z)),\\
		& \y_{k+1}(\x,\z) = \mathtt{Proj}_\Y \left( \u^{k}(\x,\z) - \alpha \nabla_{\y} f(\x,\u^{k}(\x,\z)) \right),\\
		& \T_{k+1}\left(\x,\y_k(\mathbf{x},\z)\right)  =  \y_{k+1}(\x,\z), \\
	\end{aligned}
\end{cases}
\end{equation}
where $\alpha$ denotes the step size parameter.  

\begin{algorithm}[h]
	\caption{The Proposed $\textrm{AIT}_{C}$ Algorithm}\label{alg:AIT-C}
	\begin{algorithmic}[1]
		\REQUIRE UL iteration $T$, LL iteration $K$.  
		\STATE Initialize $\x^0$ and $\z^0$.
		\FOR {$t=0 \rightarrow T-1$}
		\STATE $\y_0=\z^t$.
		\FOR {$k=0 \rightarrow K-1$}
		\STATE \% Iterative dynamics mapping.
		\STATE $\mathbf{y}_{k+1}(\x^t,\z^t)=\mathcal{T}_{k+1}\left(\mathbf{x}^t, \mathbf{y}_{k}(\x^t,\z^t)\right)$.
		\ENDFOR
		\STATE \% Update $\x$ with $\y_{K}(\x^t,\z^t)$
		\STATE $\x^{t+1}=\mathtt{Proj}_{\X}(\x^{t}-\alpha_{\x}\nabla_{\x}\psi_{K}(\x^t,\z^t))$.
		\STATE \% Update $\z$ with $\y_{K}(\x^t,\z^t)$.
		\STATE $\z^{t+1}=\mathtt{Proj}_{\mathcal{Z}}(\z^t-\alpha_{\z}\nabla_{\z}\psi_{K}(\x^t,\z^t))$.
		\ENDFOR
	\end{algorithmic}
\end{algorithm}

In Alg.~\ref{alg:AIT-C}, we illustrate the algorithm of our Augmented Iterative Trajectory for BLOs with Convex LL subproblem ($\textrm{AIT}_{C}$ for short), where $\psi_{K}(\x,\z)=F(\x,\y_{K}(\x,\z)) + g(\z)$. Whereas, to make the gradient-based BLO scheme be efficient and convergent for the BLOs with nonconvex LL subproblem, the proposed IA technique is not enough. In the next part, we will propose another augmentation technique named Pessimistic Trajectory Truncation, and discuss how to use this new technique together with IA to derive a new gradient-based BLO algorithm that is convergent and efficient for the BLOs with nonconvex LL subproblem. The detailed convergence analysis could be found in the next section.

\subsection{BLOs with Non-convex LL Subproblem}\label{PTT}

 A further limitation of existing gradient-based BLO schemes is their reliance on the last iterate of the LL optimization trajectory $\mathcal{T}_{0:K}$ for hyper-gradient computation. In contrast to the convex setting, when the LL subproblem is nonconvex, the trajectory generated by gradient-based methods may exhibit oscillatory behavior. Consequently, the last iterate $\y{K}(\x,\z)$, for example, the one generated by the proximal gradient method as in \eqref{yk_def}, does not necessarily converge uniformly to the LL solution set. Hence, using $F(\x,\y_{K}(\x,\z))$ as a surrogate for the true BLO objective $\varphi(\x)$ may lead to inaccurate approximations. This oscillatory behavior has been observed in classical convergence results for the proximal gradient method in the nonconvex setting~\cite[Theorem 10.15]{beck2017first}. To quantify optimality, one typically defines the proximal gradient residual mapping $\mathcal{R}_{\alpha}(\x,\y): = \y - \mathtt{Proj}_{\Y} \left(\y - \alpha \nabla_{\y}f(\x,\y) \right)$. For nonconvex LL problems, the residual at the last iterate $\y_{K}(\x,\z)$ may not converge uniformly to zero w.r.t. $(\x,\z)$. Instead, only the best iterate along the trajectory satisfies uniform convergence, i.e., $\min_{0 \le k \le K} \|  \mathcal{R}_{\alpha} (\x,\y_{k}(\x,\z)) \| $ uniformly converges to zero w.r.t. $(\x,\z)$. Therefore, directly substituting the last iterate $\y_{K}(\x,\z)$ into the UL objective  does not necessarily provide a reliable approximation strategy.
 	
 To address the challenges, we introduce a new strategy termed Pessimistic Trajectory Truncation (PTT). Building on the IA-based trajectory in Eq.~\eqref{yk_IA}, PTT does not rely solely on the last iterate. Instead, it adopts a pessimistic principle: rather than approximating the BLO objective using a single trajectory point, it reformulates the approximation problem by minimizing, with respect to $(\x,\z)$, the largest UL objective value attained along the trajectory. Specifically, we define $\phi_{{K}}(\x,\z):= \max_{1\le k \le K} \left\lbrace F(\x,\y_{k}(\x,\z)) \right\rbrace = F(\x,\y_{\bar{K}}(\x,\z))$, where $\bar{K}=\arg\max_{1\le k \le K} \{ F(\x,\y_{k}(\x^t,\z^t))  \}$ and approximate the original BLO problem by $\min_{\x \in \X, \z \in \mathcal{Y} }  \phi_{{K}}(\x,\z) $. In practice, the index $\bar{K}$ adaptively truncates the trajectory according to UL objective values. Unlike the naive last-iterate strategy, PTT, especially when combined with IA, yields an approximation problem that more reliably captures the behavior of the original BLO.

Theoretical justification for the surrogate problem $\min_{\x \in \X, \z \in \mathcal{Y} }~\phi_{{K}}(\x,\z) $ is provided in Section~\ref{sec42}. Our analysis leverages a key property of many gradient-type methods for nonconvex optimization, as $K \to \infty$, $\min_{0 \le k \le K} \|  \mathcal{R}_{\alpha} (\x,\y_{k}(\x,\z)) \| $ uniformly converges to zero w.r.t. $(\x,\z)$. This property, satisfied for instance by the proximal gradient method, motivates the PTT design. Although explicitly selecting $i(K) := \arg\min_{0 \le k \le K} \|  \mathcal{R}_{\alpha} (\x,\y_{k}(\x,\z)) \| $ would be impractical, PTT effectively incorporates this “best iterate” principle by minimizing over the worst-case UL objective along the trajectory. As a result, the surrogate problem inherits the desired approximation guarantees, as established in Section~\ref{sec42}.

By combining IA and PTT, our algorithm avoids the pitfalls of using the oscillatory last iterate of the LL optimization in nonconvex LL settings, thereby removing the LLC assumption and broadening applicability to more realistic scenarios. Furthermore, the PTT operation typically identifies a relatively small index $\bar{K}$, which shortens the effective trajectory length used in hyper-gradient backpropagation. This leads to reduced computational cost in practice, as corroborated by our running-time analysis on non-convex numerical examples.

Finally, our proposed algorithm is flexible with respect to the choice of iterative mapping $\mathcal{T}_{k+1}(\x,\y_{k}(\x,\z))$. For instance, one may adopt the projected gradient method in Eq.~\eqref{yk_def}.or employ accelerated gradient schemes to improve efficiency in the nonconvex case~\cite{ghadimi2016accelerated}. For example, the accelerated scheme is given by
\begin{equation}\label{yk_def3}
	\begin{cases}
	\begin{aligned}
		& \u_{0}(\x,\z) = \v_{0}(\x,\z) = \z,   \\
		& \y_{k+1}(\x,\z) = (1-\alpha_k)\u_{k}(\x,\z) + \alpha_k \v_{k}(\x,\z), \\
		& \u^{k+1}(\x,\z) = \mathtt{Proj}_\Y \left( \y_{k+1}(\x,\z) - \beta_k \nabla_{\y} f(\x,\y_{k+1}(\x,\z)) \right),\\
		& \v_{k+1}(\x,\z) = \mathtt{Proj}_\Y \left( \v_{k}(\x,\z) - \lambda_k \nabla_{\y} f(\x,\y_{k+1}(\x,\z)) \right),\\
		& \T_{k+1}\left(\x,\y_k(\mathbf{x},\z)\right)  =  \y_{k+1}(\x,\z),
	\end{aligned}
\end{cases}
\end{equation}
where $\alpha_k = \frac{2}{k+1}$, $\beta_k = \frac{1}{2L_f}$ and $\lambda_k = \frac{k\beta_k}{2}$, and $k = 0,\ldots, K-1$. 

\begin{algorithm}[h]
	\caption{The Proposed $\textrm{AIT}_{NC}$ Algorithm}\label{alg:AIT-N}
	\begin{algorithmic}[1]
		\REQUIRE UL iteration $T$, LL iteration $K$.  
		\STATE Initialize $\x^0$ and $\z^0$.
		\FOR {$t=0 \rightarrow T-1$}
		\STATE $\y_0=\z^t$.
		\FOR {$k=0 \rightarrow K-1$}
		\STATE \%  Update $y_k$ with $\x^{t}$ and $\z^{t}$.
		\STATE $\mathbf{y}_{k+1}(\x^t,\z^t)=\mathcal{T}_{k+1}\left(\mathbf{x}^t, \mathbf{y}_{k}(\x^t,\z^t)\right)$.
		\ENDFOR
		\STATE \% Pessimistic Trajectory Truncation.%\phi_{k}(\x^t,\z^t)
		\STATE $\bar{K}=\arg\max_{1\le k \le K} \{ F(\x,\y_{k}(\x^t,\z^t))  \}$. \label{outer_loop_1}
		\STATE \% Update $\x$ with $\y_{\bar{K}}(\x^t,\z^t)$.
		\STATE  Denote $\phi_{\bar{K}}(\x,\z) = F(\x,\y_{\bar{K}}(\x,\z))$.
		\STATE $\x^{t+1}=\mathtt{Proj}_{\X}(\x^{t}-\alpha_{\x}\nabla_{\x}\phi_{\bar{K}}(\x^t,\z^t))$.\label{outer_loop_2}
		\STATE \% Update $\z$ with $\y_{\bar{K}}(\x^t,\z^t)$.
		\STATE $\z^{t+1}=\mathtt{Proj}_{\mathcal{Z}}(\z^t-\alpha_{\z}\nabla_{\z}\phi_{\bar{K}}(\x^t,\z^t))$.\label{outer_loop_3}
		\ENDFOR
	\end{algorithmic}
\end{algorithm}

With the two proposed augmentation techniques and the supported iterative acceleration scheme, the complete algorithm of our AIT for BLOs with non-convex LL subproblem ($\textrm{AIT}_{NC}$ for short) is illustrated in Alg.~\ref{alg:AIT-N}.  Beyond these formulations, we emphasize that the augmentation strategies in AIT provide a new way on trajectory construction. Specifically, IA and PTT serve as augmentation mechanisms for the two endpoints of the trajectory—initialization and truncation—thereby improving both the starting position and the termination behavior of the iterative trajectory. In parallel, the extension strategies, such as gradient aggregation, prior regularization, and acceleration, enrich the trajectory evolution itself by enabling more effective iterative mappings (i.e., $\mathcal{T}_{k}$). Collectively, these augmentations establish a unified and principled view of trajectory enhancement, which is crucial for achieving both theoretical generality and practical performance across diverse BLO tasks. In the following sections, we provide the convergence analysis of both $\textrm{AIT}_{C}$ and $\textrm{AIT}_{NC}$ and further discuss the applicable scenarios of different variations.

\section{Theoretical Analysis}\label{Theoretical_Analysis}
In this section, we present the asymptotic convergence analysis of $\textrm{AIT}_{C}$ in Alg. \ref{alg:AIT-C} and $\textrm{AIT}_{NC}$ in Alg. \ref{alg:AIT-N} under mild conditions on the general schematic iterative module $\mathcal{T}_k(\x,\cdot)$. And
we also show that these conditions on $\mathcal{T}_k(\x,\cdot)$ are satisfied for some specific choices of $\mathcal{T}_k(\x,\cdot)$. Specifically, for convex BLOs, we show the asymptotic convergence of Alg.~\ref{alg:AIT-C} when $\mathcal{T}_k(\x,\cdot)$ is chosen as BDA~\cite{liu2022general} in Eq.\eqref{eq:improved-lower}. And for convex BLOs whose LL solution set is singleton, we show the asymptotic convergence of Algorithm \ref{alg:AIT-C} when $\mathcal{T}_k(\x,\cdot)$ is chosen as projected gradient in Eq.~\eqref{yk_def} and Nesterov’s acceleration scheme in Eq.~\eqref{yk_def2}. For nonconvex BLOs, we show the asymptotic convergence of Algorithm \ref{alg:AIT-N} when $\mathcal{T}_k(\x,\cdot)$ is specifically chosen as projected gradient method and accelerated gradient scheme proposed by \cite{ghadimi2016accelerated} in Eq.~\eqref{yk_def3}.

\subsection{Convergence Analysis for $\textrm{AIT}_{C}$}

We first derive the asymptotic convergence analysis of Alg.~\ref{alg:AIT-C} for convex BLOs with general schematic iterative module $\mathcal{T}_k(\x,\cdot)$.  Alg.~\ref{alg:AIT-C} can be regarded as the application of projected gradient method on 
the following single-level approximation problem of Eq.~\eqref{blo_problem},

\begin{equation}\label{eq:upper_varphiK}
	\begin{aligned}
		&\min_{\x \in \X, \z \in \mathcal{Z} } ~~  F(\x,\y_K(\x,\z)) + g(\z) \\
		&s.t. ~ \begin{cases}
			 \y_0(\x,\z) = \z, &\\
		                  \y_{k+1}(\mathbf{x},\z) = \mathcal{T}_k(\x,\y_k(\mathbf{x},\z) ),&\end{cases} 
	\end{aligned}
\end{equation}
where $k = 0,\ldots,K-1$, the function $g : \mathbb{R}^m \rightarrow \mathbb{R}$ and the closed set $\mathcal{Z} \subset \mathbb{R}^m$ are the function and compact set that are with prior information on lower level variable $\y$. To show the asymptotic convergence of Alg.~\ref{alg:AIT-C}, i.e., the approximation problem in Eq.~\eqref{eq:upper_varphiK} converges to the original bilevel problem in Eq.~\eqref{blo_problem}, i.e., $\min_{\x \in \X} \varphi(\x)$, we make following mild assumptions on BLOs throughout this subsection.

\begin{asu}\label{assum:convex}We make following standing assumptions throughout this subsection:
	\begin{itemize}
		\item[(1)] $F, \nabla_\y F, f, \nabla_\y f$ are continuous functions on $\X\times\mathbb{R}^m$.
		\item[(2)] $F(\x,\cdot) : \mathbb{R}^m \rightarrow \mathbb{R}$ is $L_F$-smooth, convex and bounded below by $M_0$ for any $\x\in \X$.
		\item[(3)] $f(\x,\cdot) : \mathbb{R}^m \rightarrow \mathbb{R}$ is $L_f$-smooth and convex for any $\x\in \X$. $f(\x,\y)$ is level-bounded in $\y$ uniformly in $\x\in\X$
		\item[(4)] $\X$, $\Y$ and $ \mathcal{Z}$ are compact sets, and $\Y$ is a convex set.
		\item[(5)] $\hat{\S}(\x) := \mathrm{argmin}_{\y } \{F(\x,\y), \ s.t. \ \y \in \Y \cap \S(\x) \} $ is nonempty for all $\x \in \X$.
		\item[(6)] $g(\z)$ is continuous on $\mathcal{Z}$.
	\end{itemize}
\end{asu}

Assumptions~\ref{assum:convex}(1–3), concerning the boundedness and smoothness of the problem functions, are standard in the literature on convex BLOs; see, e.g., \cite{liu2020generic,grazzi2020iteration,ji2021bilevel}.  
Assumption~\ref{assum:convex}(5) ensures the well-definedness of the problem. Moreover, all conditions in Assumption~\ref{assum:convex} are satisfied for the convex BLO numerical example presented in Section~\ref{numerical_results}.
Before presenting the convergence result, we recall the auxiliary function $ \varphi_K(\x, \z) = F(\x,\y_K(\x,\z))$, and rewrite the approximation problem in Eq.~\eqref{eq:upper_varphiK} into the following compact form, 
\begin{equation}\label{eq:proximal_prior}
	\min_{\mathbf{x}\in\X, \z \in \mathcal{Z}} \psi_K(\x,\z) = \varphi_K(\x, \z) + g(\z).
\end{equation}

We propose following two essential properties for general schematic iterative module $\mathcal{T}_k(\x,\cdot)$ for the convergence analysis:
\begin{enumerate}
	\item[(1)]\textbf{UL property with IA:} For each $\x \in \mathcal{X}$, $\z \in \mathcal{Z}$,
	\begin{equation*}
		\lim\limits_{K \rightarrow \infty}\varphi_K(\x,\z) \rightarrow \varphi(\x).\label{eq:dist-varphi}
	\end{equation*}
	\item[(2)]\textbf{LL property with IA:} $\{\y_{K}(\x,\z)\}$ is uniformly bounded on $\X \times \mathcal{Z}$, and for any $\epsilon>0$, there exists $k(\epsilon)>0$ such that whenever $K>k(\epsilon)$, 	
	\begin{equation*}
		\sup_{\x \in \X, \z \in \mathcal{Z}} \left\{ f(\x,\y_K(\x,\z)) - f^{\ast}(\x,\z) \right\} \le \epsilon.
	\end{equation*}
\end{enumerate}

Now, we are ready to present the asymptotic convergence result of Alg.~\ref{alg:AIT-C}.
\begin{theorem}\label{thm:general}
	Suppose both the above UL and LL properties with IA hold.
	Let $(\x_K, \z_K)$ be a minimum of the approximation problem in Eq.~\eqref{eq:upper_varphiK}, i.e.,
	\begin{equation*}
		\psi_K(\x_K, \z_K) \le \psi_K(\x, \z), \quad \forall \x \in \X, \z \in \mathcal{Z}.
	\end{equation*}
	Then, any limit point $\bar{x}$ of the sequence $\{x_K\}$ satisfies that $\bar{x}\in\mathrm{argmin}_{x\in X}\varphi(x)$, i.e., $\bar{x}$ is a solution to BLO in Eq.~\eqref{blo_problem}.
\end{theorem}
\begin{proof}
	For any limit point $\bar{\x}$ of the sequence $\{\x_K\}$, let $\{\x_{l}\}$ be a subsequence of $\{\x_K\}$ such that $\x_{l} \rightarrow \bar{\x} \in \X$. As $\{\y_K(\x,\z)\}$ is uniformly bounded on $\X \times \mathcal{Z}$ and $\{\z_K\} \subset \mathcal{Z}$ is bounded, we can further find a subsequence $\{(\x_{m},\z_m)\}$ of $\{(\x_{l},\z_l)\}$ satisfying $\y_m(\x_m,\z_m) \rightarrow \bar{\y}$ and $\z_m \rightarrow \bar{\z}$ for some $\bar{\y}$ and $\bar{\z}$. It follows from the {LL property with IA} that for any $\epsilon > 0$, there exists $M(\epsilon) > 0$ such that for any $m > M(\epsilon)$, we have
	\begin{equation*}
		f(\x_m,\y_m(\x_m,\z_m)) - f^{\ast}(\x_m) \le \epsilon.
	\end{equation*}
	Thanks to the continuity of $f(\x,\y)$, we have that $f^{\ast}(\x):=\min_{\y}f(\x,\y)$ is Upper Semi-Continuous (USC for short) on $\X$ (see, e.g., \cite[Lemma 2]{liu2020generic}). 
	By letting $m \rightarrow \infty$, and since $f$ is continuous and $f^{\ast}(\x)$ is USC on $\X$, we have $f(\bar{\x},\bar{\y}) - f^{\ast}(\bar{\x}) \le \epsilon$. As $\epsilon$ is arbitrarily chosen, we have $f(\bar{\x},\bar{\y}) - f^{\ast}(\bar{\x}) \le 0$ and thus $\bar{\y} \in \S(\bar{\x})$.
	Next, as $F$ is continuous at $(\bar{\x},\bar{\y})$ and g is continuous at $\bar{\z}$, for any $\epsilon > 0$, there exists $M(\epsilon) > 0$ such that for any $m > M(\epsilon)$, it holds
	\begin{equation*}
		F(\bar{\x},\bar{\y}) + g(\bar{\z}) \le F(\x_m,\y_m(\x_m,\z_m)) + g(\z_m) + \epsilon.
	\end{equation*}
	Then, we have, for any $m > M(\epsilon)$ and $\x \in \X$, $\z \in \mathcal{Z}$,
	\begin{equation}\label{eq1}
		\begin{aligned}
			\ \varphi(\bar{\x}) + g(\bar{\z}) &= \inf_{\y \in \S(\bar{\x}) } F(\bar{\x}, \y) + g(\bar{\z})\\
			&\le F(\bar{\x},\bar{\y}) + g(\bar{\z}) \\
			&\le F(\x_m,\y_m(\x_m,\z_m)) + g(\z_m) + \epsilon \\
			&\le \varphi_m(\x,\z) + g(\z) + \epsilon.\quad\quad
		\end{aligned}
	\end{equation}
	Taking $\z = \bar{\z}$, $m \rightarrow \infty$ and by the {UL property with IA}, we have
	\begin{equation*}
		\begin{aligned}
			\varphi(\bar{\x}) \le \lim_{m \rightarrow \infty}\varphi_m(\x,\bar{\z}) + \epsilon = \varphi(\x) + \epsilon, \ \forall \x \in \X.
		\end{aligned}
	\end{equation*}
	By taking $\epsilon \rightarrow 0$, we have $\varphi(\bar{\x}) \le \varphi(\x), \ \forall \x \in \X$ which implies $\bar{\x} \in \arg\min_{\x \in \X} \varphi(\x)$.
\end{proof}

In the above convergence result, we observe that the convergence of Eq.~\eqref{eq:upper_varphiK} to the original bilevel problem in Eq.~\eqref{blo_problem} does not rely on any restrictive assumptions on $g$ or $\mathcal{Z}$. This implies that prior knowledge can be incorporated into $g$ or $\mathcal{Z}$ to potentially accelerate convergence without compromising the overall correctness of the method.

The above convergence result characterizes the convergence of the approximation problem defined in Eq.~\eqref{eq:upper_varphiK} to the original bilevel problem in Eq.~\eqref{blo_problem} in terms of the minimum. However, in practice the approximation problem in Eq.~\eqref{eq:upper_varphiK} is usually nonconvex, and thus stationary points are typically obtained by applying the proximal gradient method. It is therefore natural to characterize the limiting behavior of the stationary points of Eq.~\eqref{eq:upper_varphiK}. To exclude pathological scenarios in which $\varphi_K(\x,\z)$ exhibits infinitely oscillatory behavior as $K \to \infty$, we impose a regularity condition on the subdifferential of $\varphi_K(\x,\z)$. 
Recall that for a proper l.s.c. function $h:\mathbb{R}^d \to \mathbb{R}\cup\{+\infty\}$, the Fréchet subdifferential at $u$ is 
$\hat{\partial} h(u) := \{ \xi \mid \liminf_{v \to u,\, v \neq u} (h(v)-h(u)-\langle \xi, v-u\rangle)/\|v-u\| \ge 0 \}$, 
and the limiting subdifferential is 
$\partial h(u) := \{ \xi \mid \exists\, u^k \to u,\; h(u^k)\to h(u),\; \xi^k \in \hat{\partial} h(u^k),\; \xi^k \to \xi \}$.
Building on the notion of prox-regularity in \cite[Definition 13.27]{rockafellar2009variational}, we make the following assumption. 
We say that $\psi_K(\x,\z)$ is uniformly prox-regular at $(\bar{\x},\bar{\z})$ if there exist $\epsilon>0$ and $\rho\ge 0$ such that, for all $K$, 
$\psi_K(\x',\z') \ge \psi_K(\x,\z) + \langle \xi, (\x',\z')-(\x,\z)\rangle - \tfrac{\rho}{2}\|(\x',\z')-(\x,\z)\|^2$ 
whenever $\|(\x,\z)-(\bar{\x},\bar{\z})\|\le \epsilon$, $\|(\x',\z')-(\bar{\x},\bar{\z})\|\le \epsilon$, and $\xi \in \partial \psi_K(\x,\z)$.

\begin{theorem}\label{thm: thm_station}
	Suppose both the above UL and LL properties with IA hold.
Let $(\x_K,\z_K)$ be a $\epsilon_K$-stationary point of the approximation problem in Eq.~\eqref{eq:upper_varphiK} with $\epsilon_K \rightarrow 0$, i.e.,
	\begin{equation*}
		\xi_K \in \partial \psi_K(\x_K, \z_K) + \mathcal{N}_{\X \times \mathcal{Z}}(\x_K, \z_K) \quad \text{with} \,\,\, \|\xi_K\| \le \epsilon_K.
	\end{equation*}
	Then, any limit point $(\bar{\x},\bar{\z})$ of the sequence ${(\x_K,\z_K)}$ is such that, if $\psi_K$ is uniformly prox-regular at $(\bar{\x},\bar{\z})$, we have $$0 \in \partial \varphi(\bar{\x}) + \mathcal{N}_{\X}(\bar{\x}).$$ That is, $\bar{\x}$ is a stationary point to BLO in Eq.~\eqref{single_reform}.
\end{theorem}
\begin{proof}
Let $(\bar{\x},\bar{\z})$ be a limit point of ${(\x_K,\z_K)}$, and let ${(\x_{l},\z_{l})}$ be a subsequence such that $(\x_{l},\z_{l}) \rightarrow (\bar{\x}, \bar{\z}) \in \X \times \mathcal{Z}$. Since ${\y_K(\x,\z)}$ is uniformly bounded on $\X \times \mathcal{Z}$, there exists a further subsequence ${(\x_{m},\z_{m})}$ with $\y_m(\x_m,\z_m) \rightarrow \bar{\y}$ for some $\bar{\y}$.
By the same arguments as in the proof of Theorem~\ref{thm:general}, we obtain
	\begin{equation}\label{eq11}
		\varphi(\bar{\x}) + g(\bar{\z}) \le \liminf_{m  \rightarrow \infty}\, \psi_m(\x_m,\z_m).
	\end{equation}
	Next, since $(\x_K,\z_K)$ is an $\epsilon_K$-stationary point of $\min_{\x\in\X, \z \in \mathcal{Z}} \psi_K(\x,\z)$ and $\psi_K$ is uniformly prox-regular at $(\bar{\x},\bar{\z})$, there exist $\epsilon>0$ and $\rho \ge 0$ such that, for sufficiently large $m$,
	\[
	\begin{aligned}
			\psi_m(\x,\z) \ge \,&\psi_m(\x_m, \z_m) + \langle \xi_m, (\x,\z)-(\x_m, \z_m) \rangle \\ &- \tfrac{\rho}{2}\|(\x,\z)- (\x_m, \z_m) \|^2,
	\end{aligned}
	\] 
	for all $(\x,\z) \in \X \times \mathcal{Z}$ with $\|(\x,\z)-(\bar{\x},\bar{\z})\|\le \epsilon$. Fix any $\x \in \X$ with $\|\x - \bar{\x}\|\le \epsilon$, and take $\z = \bar{\z}$. Passing to the limit $m \to \infty$ in the inequality above, and combining with the UL property with IA and \eqref{eq11}, and noting that $\xi_m \to 0$, we deduce
	\[
	\varphi(\x) \ge \varphi(\bar{\x}) - \tfrac{\rho}{2}\|(\x,\z)- (\bar{\x},\bar{\z})\|^2.
	\]
	This implies that $0 \in \partial \varphi(\bar{\x}) + \mathcal{N}_{\X}(\x)$ and therefore $\bar{\x}$ is a stationary point of $\min_{\x \in \X} \varphi(\x)$.
\end{proof}

Next, we show that both the UL and LL properties with IA hold when $\mathcal{T}_k(\x,\cdot)$ is chosen as some specific schemes. Here we provide the asymptotic convergence analysis of Alg.~\ref{alg:AIT-C}, when $\mathcal{T}_k(\x,\cdot)$ is defined as BDA\cite{liu2022general} in Eq.~\eqref{eq:improved-lower}.

\begin{theorem}\label{thm_convergence}
	%Suppose Assumptions~\ref{assum:F}, \ref{convex_assump} are satisfied, $\X$, $\Y$ and $\mathcal{Z}$ are compact, and .
	Let $\{\y_k(\mathbf{x},\z)\}$ be the output generated by \eqref{eq:improved-lower} with
	$s_l \in (0,1/L_f)$, $s_u \in (0,1/L_F)$, $\mu \in (0,1)$, $\alpha_k = \frac{1}{k+1}$, $\beta_k \in [\underline{\beta}, 1]$ with some $\underline{\beta} > 0$, $|\beta_k - \beta_{k-1}| \le \frac{c_\beta}{(k+1)^2}$ with some $c_\beta > 0$,
	then we have that both the LL and UL properties with IA hold. 
\end{theorem}
\begin{proof}
	We can easily obtain the UL property with IA from \cite[Theorem 3]{liu2022general}. Next, since Assumption~\ref{assum:convex} holds, $\X$, $\Y$ and $\mathcal{Z}$ are compact, by using the convergence rate result established in  \cite[Theorem 4]{liu2022general} and the similar arguments in the proof of \cite[Theorem 5]{liu2022general}, it can be shown that there exists $C > 0$ such that for any $\x \in \X$, $\z \in \mathcal{Z}$, we have 
	\[
	f(\x,\y_K(\mathbf{x},\z)) - f^*(\x) \le C\sqrt{\frac{1+\ln K}{K^{\frac{1}{4}}}}.
	\]
	Because $\sqrt{\frac{1+\ln K}{K^{\frac{1}{4}}}} \rightarrow 0$ as $K \rightarrow \infty$, $\{\y_K(\mathbf{x},\z)\} \subset \Y$, and $\Y$ is compact, LL property with IA holds. 
\end{proof}

\subsubsection{Lower-Level Singleton Case}
In this part, we focus on the convex BLOs whose lower-level problem solution set is singleton for any $\x \in \X$. We first consider constructing the iterative module $\mathcal{T}_k(\x,\cdot)$ by projected gradient method as described in Eq.~\eqref{yk_def}.	

\begin{theorem}\label{pg_convergence}
	Suppose $\S(\x)$ is singleton for all $\x \in \X$. Let $\{\y_k(\mathbf{x},\z)\}$ be the output generated by \eqref{yk_def} with
	$\alpha \in (0,1/L_f)$, then we have that both the LL and UL convergence properties with IA hold. 
\end{theorem}
\begin{proof}
	According to \cite[Theorem 10.21]{LiuLYZZ21}, when $f(\x,\cdot)$ is convex and $L_f$-smooth for any $\x\in \X$, and $\alpha \in (0,1/L_f)$, $\{\y_{k}(\x,\z)\}$ satisfies
	\begin{equation}\label{PG_rate}
		\begin{aligned}
			f(\x,\y_K(\x,\z)) - f^{\ast}(\x) &\le  \frac{\mathrm{dist}(\y_0(\x,\z),\S(\x))}{2\alpha K} \\ &= \frac{ \mathrm{dist}(\z, \S(\x))}{2\alpha K},
		\end{aligned}
	\end{equation}
	where $\mathrm{dist}(\z,S(\x))$ denotes the distance from $\z$ to the set $S(\x)$.
	Since $\X$, $\Y$ and $\mathcal{Z}$ are all compact sets, then there exists $M > 0$ such that $\mathrm{dist}(\z,\S(\x)) \le M$ for $(\x,\z) \in \X \times \mathcal{Z}$. Then we can easily obtained that {LL property with IA} holds from  Eq.~\eqref{PG_rate}, $\{\y_K(\mathbf{x},\z)\} \subset \Y$ and $\Y$ is compact. Next, given any $(\x,\z) \in \X \times \mathcal{Z}$, for any limit point $\bar{\y}(\x,\z)$ of sequence $\{\y_k(\x,\z) \}$, we have $f(\x,\bar{\y}(\x,\z)) = f^*(\x)$ and thus $\bar{\y}(\x,\z) \in \S(\x)$ from \eqref{PG_rate}. Therefore, since $\{\y_k(\x,\z) \}$ is bounded and $\S(\x)$ is singleton, we have $\y_k(\x,\z) \rightarrow \S(\x) = \hat{\S}(\x)$ and thus {UL property with IA} holds.
\end{proof}

Next, we consider the iterative module $\mathcal{T}_k(\x,\cdot)$ as the Nesterov’s acceleration schemes described in Eq.~\eqref{yk_def2}.

\begin{theorem}
	Suppose $\S(\x)$ is singleton for all $\x \in \X$. Let $\{\y_k(\mathbf{x},\z)\}$ be the output generated by Eq.~\eqref{yk_def2} with
	$\alpha \in (0,1/L_f)$, then we have that both the LL and UL convergence properties with IA hold. 
\end{theorem}
\begin{proof}
	According to \cite[Theorem 10.34]{LiuLYZZ21}, when $f(\x,\cdot)$ is convex and $L_f$-smooth for any $\x\in \X$, and $\alpha  \in (0,1/L_f)$, $\{\y_{k}(\x,\z)\}$ admits the following property,
	\begin{equation}
		\begin{aligned}
			f(\x,\y_K(\x,\z)) - f^{\ast}(\x) &\le  \frac{2\mathrm{dist}(\y_0(\x,\z),\S(\x))}{\alpha(K+1)^2} \\
			&= \frac{2\mathrm{dist}(\z, \S(\x))}{\alpha(K+1)^2}.
		\end{aligned}
	\end{equation}
	Since $\X$, $\Y$ and $\mathcal{Z}$ are compact sets, by using the above inequality and the similar arguments as in the proof of Theorem \ref{pg_convergence}, we can show that both the LL and UL  convergence properties with IA hold.
\end{proof}

\subsection{Convergence Analysis for $\textrm{AIT}_{NC}$}\label{sec42}

In this part,  we first conduct the asymptotic convergence analysis of $\textrm{AIT}_{NC}$ in Alg.~\ref{alg:AIT-N} with general schematic iterative module $\mathcal{T}_k$.  Indeed, Alg.~\ref{alg:AIT-N} can be regarded as the application of projected gradient method on the following single-level approximation problem of Eq.~\eqref{blo_problem},
\begin{equation}\label{eq:upper_phiK}
	\begin{aligned}
		&\min_{\x \in \X, \z \in \Y } ~~  \phi_{{K}}(\x,\z) := \max_k \left\lbrace F(\x,\y_{k}(\x,\z))\right\rbrace \\
		&s.t. \begin{cases} \y_0(\x,\z) = \z, &\\
		\y_{k+1}(\mathbf{x},\z) = \mathcal{T}_k(\x,\y_k(\mathbf{x},\z) ), ~~ k = 0,\ldots,K-1.&		\end{cases}
	\end{aligned}
\end{equation}
Then we investigate the convergence of solutions of this approximation problem towards those of the original BLO under following mild assumptions.

\begin{asu}\label{assum1}We make following standing assumptions throughout this subsection:
	\begin{itemize}
		\item[(1)] $F, f, \nabla_\y f$ are continuous functions on $\X\times\mathbb{R}^m$.
		\item[(2)] $f(\x,\cdot) : \mathbb{R}^m \rightarrow \mathbb{R}$ is $L_f$-smooth for any $\x\in \X$.
		\item[(3)] $\X$ and $\Y$ are convex compact sets.
		\item[(4)] $\S(\x)$ is nonempty for any $\x \in \X$.
		\item[(5)] For any $(\bar{\x},\bar{\y})$ minimizing $F(\x,\y)$ over constraints $\x \in \X, \y \in \Y$ and $\y \in \tilde{\S}(\x)$, it holds that $\bar{\y} \in \S(\x)$.
	\end{itemize}
\end{asu}

Note that  $\tilde{\S}(\x)$ denotes the set of LL stationary points, i.e., $\tilde{\S}(\x) = \{\y \in \Y | 0 = \nabla_{\y}f(\x,\y) + \mathcal{N}_{\Y}(\y) \}.$ %It should be noticed that $\y \in \tilde{\S}(\x)$ if and only if $\mathcal{R}_{\alpha}(\x,\y) = 0$. 
Assumptions~\ref{assum1}(1)–(2), which concern the boundedness and smoothness of the problem functions, are standard in the bilevel optimization literature; see, e.g., \cite{liu2022bome,shen2023penalty,huang2024optimal}.  
Assumption~\ref{assum1}(4) ensures that the problem is well-defined.  
Assumption~\ref{assum1}(5) guarantees that the set of LL problem stationary points does not relax the bilevel problem, a condition that can be ensured by the Polyak–Łojasiewicz (PL) property of LL problem, as assumed in other bilevel optimization studies; see, e.g., \cite{liu2022bome,shen2023penalty,huang2024optimal}.  
Finally, all conditions in Assumption~\ref{assum1} are satisfied by the numerical example presented in Section~\ref{numerical_results}.

We propose the \textbf{weak LL property with IA} for the asymptotic convergence analysis of Algorithm \ref{alg:AIT-N} with general schematic iterative module $\mathcal{T}_k(\x,\cdot)$. We say that \textbf{weak LL property with IA} is satisfies if $\mathcal{T}_k(\x,\z) = \z$ for any $k$ and $\z \in \tilde{\S}(\x)$; $\{\y_{K}(\x,\z)\}$ is uniformly bounded on $\X \times \mathcal{Z}$; and there exists $\alpha > 0$ such that for any $\epsilon>0$, there exists $k(\epsilon)>0$ such that whenever $K>k(\epsilon)$, 	
\begin{equation*}
	\sup_{\x \in \X, \z \in \mathcal{Z}}\left\{ \min_{0 \le k \le K} \|  \mathcal{R}_{\alpha} (\x,\y_{k}(\x,\z)) \| \right\} \le \epsilon,
\end{equation*}
where 
\[
\mathcal{R}_{\alpha}(\x,\y): = \y - \mathtt{Proj}_{\Y} \left(\y - \alpha \nabla_{\y}f(\x,\y) \right).
\]
It should be noticed that $\mathcal{R}_{\alpha}(\x,\y) = 0$ if and only if $\y \in \tilde{\S}(\x)$. 
In the following theorem, we establish the asymptotic convergence of Alg.~\ref{alg:AIT-N} when the weak LL property with IA is satisfied.

\begin{theorem}\label{Thm1}
	Suppose the weak LL property with IA holds.
	Let $(\x_K, \z_K)$ be a minimum of approximation problem in Eq.~\eqref{eq:upper_phiK}, i.e.,
	\begin{equation*}
		\phi_K(\x_K, \z_K) \le \phi_K(\x, \z), \quad \forall \x \in \X, \z \in \mathcal{Z}.
	\end{equation*}
	Then, any limit point $\bar{x}$ of the sequence $\{x_K\}$ satisfies that $\bar{x}\in\mathrm{argmin}_{x\in X}\varphi(x)$, i.e., $\bar{x}$ is a solution to BLO in Eq.~\eqref{blo_problem}.
\end{theorem}
\begin{proof}%[Proof of Theorem 3.1]
	For any $K > 0$, we define $$i(K) := \mathrm{argmin}_{0 \le k \le K} \| \mathcal{R}_{\alpha}(\x,\y_{k}(\x,\z))\|. $$
	For any limit point $\bar{\x}$ of the sequence $\{\x_K\}$, let $\{\x_{l}\}$ be a subsequence of $\{\x_K\}$ such that $\x_{l} \rightarrow \bar{\x} \in \X$.
	As $\{\y_{i(K)}(\x_K, \z_K)\} \subset \Y$ and $\Y$ is compact, we can find a subsequence $\{\x_{j}\}$ of $\{\x_{l}\}$ satisfying $\y_{i(j)}(\x_j,\z_j) \rightarrow \bar{\y}$ for some $\bar{\y} \in \Y$. It follows from the weak LL convergence property that for any $\epsilon > 0$, there exists $J(\epsilon) > 0$ such that for any $j > J(\epsilon)$, we have
	\[
	\|\mathcal{R}_{{\alpha}}(\x_j,\y_{i(j)}(\x_j,\z_j))\| \le \epsilon.
	\]
	As $\Y$ is a convex compact set, it follows from \cite[Theorem 6.42]{beck2017first} that $\mathtt{Proj}_{\Y}$ is continuous. Combined with the assumed continuity of $\nabla_{\y} f(\x,\y)$, we get that $\mathcal{R}_{\alpha}(\x,\y)$ is continuous.
	Then, by letting $j \rightarrow \infty$, we have
	\[
	\|\mathcal{R}_{{\alpha}}(\bar{\x}, \bar{\y})\| \le \epsilon.
	\]
	As $\epsilon$ is arbitrarily chosen, we have $\|\mathcal{R}_{{\alpha}}(\bar{\x}, \bar{\y})\| \le 0$ and thus $\bar{\y} \in \tilde{S}(\bar{\x})$. Then by using the same arguments as in the proof of \cite[Theorem 3.1]{liu2021towards}, we can get the conclusion.
\end{proof}

Theorem~\ref{Thm1} establishes convergence of the approximation problem in Eq.~\eqref{eq:upper_phiK} to the original BLO in Eq.~\eqref{blo_problem} in terms of optimal solutions.  
In analogy to Theorem~\ref{thm: thm_station}, we next characterize the limiting behavior of stationary points of Eq.~\eqref{eq:upper_phiK}.

\begin{theorem}\label{theorem4.7}
	Suppose the weak LL property with IA holds.  
	Let $(\x_K,\z_K)$ be an $\epsilon_K$-stationary point of the approximation problem in Eq.~\eqref{eq:upper_phiK}, with $\epsilon_K \to 0$, i.e.,
	\begin{equation*}
		\xi_K \in \partial \phi_K(\x_K, \z_K) + \mathcal{N}_{\X \times \mathcal{Z}}(\x_K, \z_K), \qquad \|\xi_K\| \le \epsilon_K.
	\end{equation*}
	Then, for any limit point $(\bar{\x},\bar{\z})$ of the sequence $\{(\x_K,\z_K)\}$, if $\phi_K$ is uniformly prox-regular at $(\bar{\x},\bar{\z})$ and there exists $\delta > 0$ such that $	\inf_{\y\in \tilde{\S}(\bar{\x}) \cap \mathbb{B}(\bar{\z};\delta)} F(\bar{\x},\y) = \inf_{\y\in \tilde{\S}(\bar{\x})} F(\bar{\x},\y)$,
	where $\mathbb{B}(\bar{\z};\delta):= \{\z : \|\z - \bar{\z}\| \le \delta\}$,  
	then
	\[
	0 \in \partial \tilde{\varphi}(\bar{\x}) + \mathcal{N}_{\X}(\bar{\x}),
	\]
	where $\tilde{\varphi}(\x):= \inf_{\y\in \tilde{\S}(\x)\cap \mathbb{B}(\bar{\z};\epsilon)} F(\x,\y)$ for some $\epsilon > 0$.  
	That is, $\bar{\x}$ is a stationary point of the problem $\min_{\x \in \X} \tilde{\varphi}(\x)$.
\end{theorem}
\begin{proof}
	For any $K > 0$, define
	\[
	i(K) := \mathrm{argmin}_{0 \le k \le K} \| \mathcal{R}_{\alpha}(\x,\y_{k}(\x,\z))\|.
	\]
	Let $\{(\x_{l}, \z_{l})\}$ be a subsequence converging to $(\bar{\x}, \bar{\z})$. Following the same arguments as in the proof of Theorem \ref{Thm1}, we can extract a further subsequence, re-indexed by $j$, such that $(\x_j, \z_j) \to (\bar{\x}, \bar{\z})$ and $\y_{i(j)}(\x_j,\z_j)\to \bar{\y}$, where $\bar{\y} \in \tilde{S}(\bar{\x})$. 	Following the same reasoning as in \cite[Theorem 3.1]{liu2021towards}, we deduce
	\begin{equation}\label{eq12}
	\inf_{\y \in \tilde{\S}(\bar{\x})} F(\bar{\x},\y) 
	\le \liminf_{j \to \infty} \phi_j(\x_j,\z_j).
	\end{equation}
	Next, since $(\x_K,\z_K)$ is an $\epsilon_K$-stationary point of $\min_{\x\in\X, \z\in \mathcal{Z}} \phi_K(\x,\z)$ and $\phi_K$ is uniformly prox-regular at $(\bar{\x},\bar{\z})$, there exist $\epsilon>0$ and $\rho \ge 0$ such that for sufficiently large $j$,
	\[
	\begin{aligned}
		\phi_j(\x,\z) \ge &\, \phi_j(\x_j, \z_j) + \langle \xi_j, (\x,\z)-(\x_j, \z_j) \rangle \\
		&- \tfrac{\rho}{2}\|(\x,\z)-(\x_j, \z_j)\|^2,
	\end{aligned}
	\]
	for all $(\x,\z) \in \X \times \mathcal{Z}$ with $\|(\x,\z)-(\bar{\x},\bar{\z})\| \le \epsilon$.  
	Fix $\x \in \X$ with $\|\x - \bar{\x}\| \le \epsilon/2$, and take $\y \in \tilde{\S}(\x)$ with $\|\y - \bar{\z}\| \le \epsilon/2$.  
	Since $\y_k(\x,\y) = \y$ for all $k$, we have $\phi_j(\x,\y) = F(\x,\y)$. Substituting $\z = \y$ and passing to the limit $j \to \infty$, combining with \eqref{eq12} and using $\xi_j \to 0$, we obtain
	\[
	F(\x,\y) \ge \inf_{\y\in \tilde{\S}(\bar{\x})} F(\bar{\x},\y) - \tfrac{\rho}{2}\|(\x,\z)-(\bar{\x},\bar{\z})\|^2.
	\]
	Taking the infimum over $\y \in \tilde{\S}(\x) \cap \mathbb{B}(\bar{\z};\epsilon/2)$, and letting
$
	\tilde{\varphi}(\x) := \inf_{\y \in \tilde{\S}(\x) \cap \mathbb{B}(\bar{\z};\epsilon/2)} F(\x,\y),
$
	together with the assumption $\tilde{\varphi}(\bar{\x}) = \inf_{\y \in \tilde{\S}(\bar{\x})} F(\bar{\x},\y)$ (by potentially reducing $\epsilon$), we deduce
	\[
	\tilde{\varphi}(\x) \ge \tilde{\varphi}(\bar{\x}) - \tfrac{\rho}{2}\|(\x,\z)-(\bar{\x},\bar{\z})\|^2.
	\]
	This implies$
	0 \in \partial \tilde{\varphi}(\bar{\x}) + \mathcal{N}_{\X}(\bar{\x})$,
	and hence $\bar{\x}$ is a stationary point of $\min_{\x \in \X} \tilde{\varphi}(\x)$.
\end{proof}

Next, we will show that when $\mathcal{T}_k$ is chosen as the classical projected gradient method and the accelerated gradient scheme proposed by \cite{ghadimi2016accelerated} in Eq.~\eqref{yk_def3}, the \textbf{weak LL property with IA} is satisfied.
We first consider constructing the iterative module $\mathcal{T}_k(\x,\cdot)$ by the classical projected gradient method as,
\begin{equation}\label{PG2}
	\T_{k+1}\left(\x,\y_k(\mathbf{x},\z)\right) = \mathtt{Proj}_{\Y}\left( \y_k(\mathbf{x},\z) - \alpha\nabla_{\mathbf{y}} f(\mathbf{x},\mathbf{y}_{k}(\mathbf{x},\z)) \right),
\end{equation}
where $\alpha $ denotes the step size parameter, and $k = 0,\ldots, K-1$. 
\begin{theorem}\label{pg_convergence_nonconvex}
	Let $\{\y_k(\mathbf{x},\z)\}$ be the output generated by \eqref{PG2} with
	$\alpha^k_{\y} \in [\underline{\alpha}_{\y},\overline{\alpha}_{\y}] \subset (0,\frac{2}{L_f})$, then we have that the weak LL property with IA holds. 
\end{theorem}
\begin{proof}
	It follows from \cite[Lemma 3.1]{liu2021towards} that there exists $C_f > 0$ such that
	\begin{equation*}
		\min_{0 \le k \le K} \|  \mathcal{R}_{\underline{\alpha}_{\y}} (\x,\y_{k}(\x,\z)) \| \le \frac{C_f}{\sqrt{K+1}}, \quad \forall \x \in \X, \z \in \Y.
	\end{equation*}
	Then the conclusion follows immediately.
\end{proof}

The iterative module $\mathcal{T}_k(\x,\cdot)$ can also be the accelerated gradient scheme proposed in \cite{ghadimi2016accelerated}, which has been described in Eq.~\eqref{yk_def3}.

\begin{theorem}\label{apg_convergence_nonconvex}
	Let $\{\y_k(\mathbf{x},\z)\}$ be the output generated by Eq.~\eqref{yk_def3}, then we have that the weak LL property with IA	 holds. 
\end{theorem}
\begin{proof}
	Since when $\Y$ is a compact set, $\v_0(\x,\z)$ and $\S(x)$ are both uniformly bounded for any $\x \in \X$. Then,
	it follows from \cite[Corollary 2]{ghadimi2016accelerated} that there exists $C > 0$ such that
	\begin{equation*}
		\min_{0 \le k \le K} \|  \mathcal{R}_{\frac{1}{2L_f}} (\x,\y_{k}(\x,\z)) \| \le \frac{C}{\sqrt{K}}, \quad \forall \x \in \X, \z \in \Y.
	\end{equation*}
	Then the conclusion follows immediately.
\end{proof}

\section{Applications}
As mentioned above, although many popular hierarchical learning and vision tasks have been well understood as the BLO problem in Eq.~\eqref{blo_problem}, almost no solution strategies with high efficiency and accuracy are available for some challenging tasks due to high-dimensional search space and complex theoretical properties. Moreover, many complex deep learning models endowed with nested constrained objectives are still lack of proper reformulation and reliable methodology with grounded theoretical guarantee.

\textbf{Neural Architecture Search}. We first consider one of the most representative  large-scale high-dimensional BLO application, i.e., Neural Architecture Search (NAS). Typically, NAS aims to find high-performance neural network structures with automated process. Here, we focus on the gradient-based differentiable NAS methods, which have been widely investigated recently. This type of methods select certain search strategy to find the optimal network architecture based on the well defined architecture search space. DARTS~\cite {liu2018darts}, ones of the representative differentiable NAS methods, relaxes the discrete search space to be continuous, and jointly optimize the architecture parameter (i.e., the UL variable $\x$) and network weights (i.e., the LL variable $\y$) with gradient descent step, which implies a typical BLO problem. Following the standard NAS setting for classification tasks, the UL and LL objectives are defined as
$F(\mathbf{x}, \mathbf{y}) = \ell_{\mathtt{val}}(\x,\y)$ and $f(\mathbf{x}, \mathbf{y}) = \ell_{\mathtt{tr}}(\x,\y)$, where $\ell_{\mathtt{val}}$ and $\ell_{\mathtt{tr}}$ represent the cross-entropy losses calculated on the validation and training datasets, respectively.

 In details, the searched architecture consists of multiple computation cells, and each cell is a directed acyclic graph including $N$ nodes. The nodes in each cell, denoted as $d$, represent the latent feature map, while the directed edges $o_{(i,j)}, i,j\in[1,N]$ between nodes $d^{i}$ and $d^{j}$ indicate the candidate operations based on defined search space denoted as $\mathcal{O}$. Then $\x$ is supposed to be defined as the whole continuous variables $\x=\{\x_{(i,j)}\}$, where $\x_{(i,j)}$ denotes the vector of the operation mixing weights for nodes pair $(i,j)$ of dimension $\mathcal{O}$. Hence $o_{(i,j)}$ is supposed to be the ideal operation chosen by $\arg\max$ operation, i.e., $o_{(i,j)}=\arg\max_{o\in\mathcal{O}}\x_{{(i,j)}}$. For most learning and vision tasks which employ NAS to optimize the network structure, the LL subproblem usually has high-dimensional search space and complex task-specific constraints. 

As one of the mostly widely used solution strategies, DARTS introduces finite difference approximation with a single-step gradient update of $\y$ to avoid the complex matrix-vector production. On top of that, a series of variations ~\cite{xu2019pc,chen2019progressive,dong2019searching} based on DARTS have been developed and applied in numerous applications. In addition to this approximation operation adopted by DARTS, we could also solve this BLO application with other GBMs such as RHG and CG. Whereas, since the LL variables contain a cumbersome number of parameters from different searching cells, all the above methods suffer from poor theoretical investigation on this non-convex high-dimensional BLO problem. In comparison, by effective augmentation of the iterative trajectory, AIT updates the UL variables with more accurate hyper-gradient, which is more qualified for this BLO problem with non-convex follower.

\begin{figure*}[h!]
	\begin{center}
		\begin{tabular}{c@{\extracolsep{0.1em}}c@{\extracolsep{0.1em}}c@{\extracolsep{0.1em}}c@{\extracolsep{0.1em}}c@{\extracolsep{0.1em}}}
			\includegraphics[height=3.2cm,width=4.2cm,trim=0 0 0 0,clip]{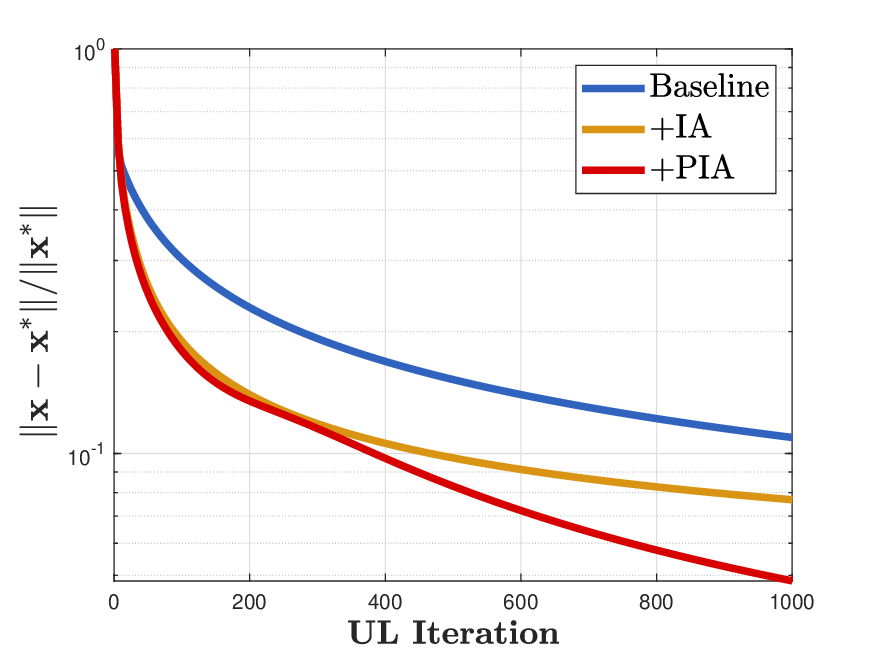} &\includegraphics[height=3.2cm,width=4.2cm,trim=0 0 0 0,clip]{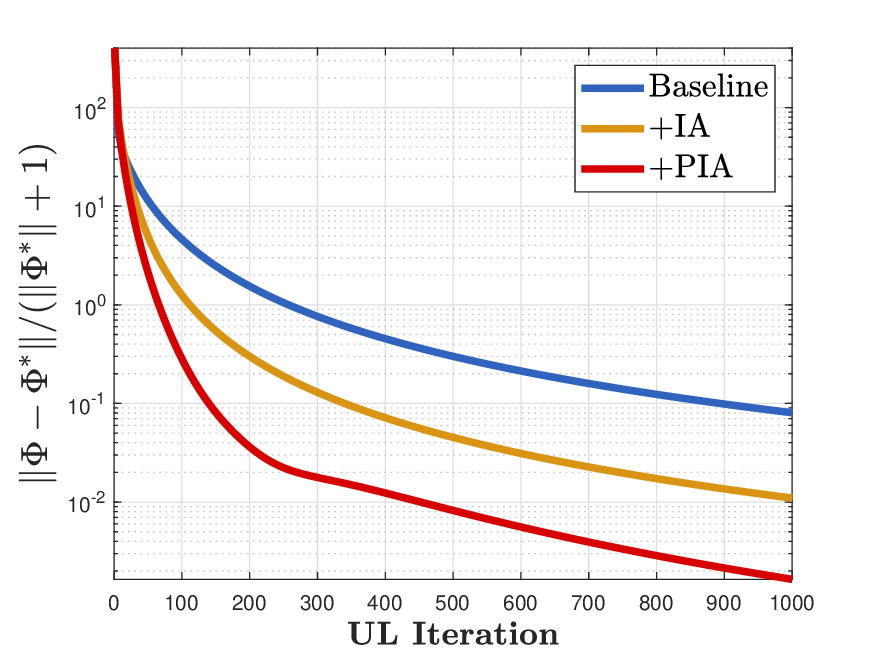} &                 
			
			\includegraphics[height=3.2cm,width=4.2cm,trim=0 0 0 0,clip]{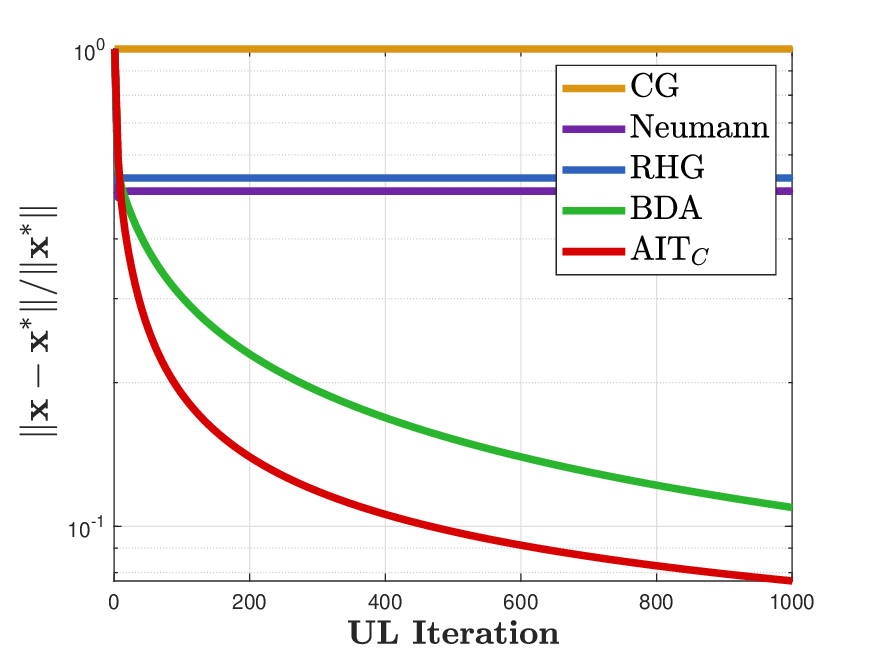} &\includegraphics[height=3.2cm,width=4.2cm,trim=0 0 0 0,clip]{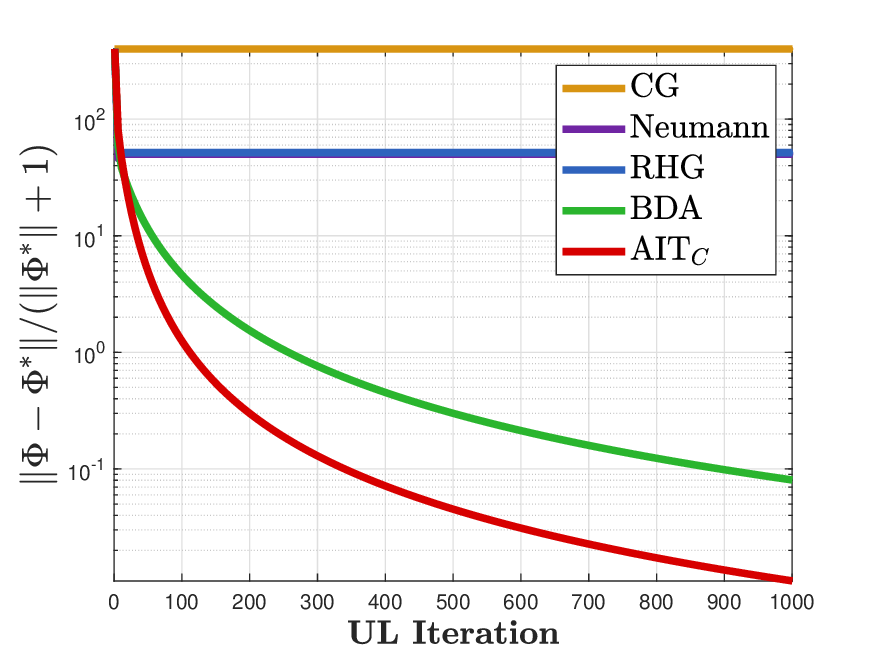}\\
			\footnotesize (a) &\footnotesize (b)	&\footnotesize (c) &\footnotesize (d) \\ 

		\end{tabular}
	\end{center}

	\caption{The first two subfigures illustrate the convergence behavior of $\left\|\Phi-\Phi^{*}\right\|/(\left\|\Phi^{*}\right\|+1)$ and $\Vert \x - \x^*\Vert/ \Vert \x^*\Vert$ for the baseline (i.e., BDA) and variations with techniques from $\textrm{AIT}_{C}$ including IA (i.e., $+\textrm{IA}$) and prior regularization (i.e., $+\textrm{P}$). As for the last two subfigures, we compare the convergence behavior of $\textrm{AIT}_{C}$ and mainstream GBMs including CG, Neumann, RHG, BDA. We choose the same representative initialization point for all the methods as $(\x_{0},\y_{0})=(2\mathbf{e},(2\mathbf{e},2\mathbf{e}))$. }\label{toy_convergence_LLC}
\end{figure*}

\textbf{Generative Adversarial Networks}.  We move one more step to apply our AIT to other challenging tasks with nested constrained learning objectives, i.e., Generative Adversarial Network (GAN), which has been one of the most popular deep learning applications and achieved impressive progress in various computer vision tasks. In general, GAN based model is widely recognized as a minimax game between two adversarial players: the generator network $G(\x;\cdot)$ (parameterized by $\x$) and the discriminator network $D(\y;\cdot)$ (parameterized by $\y$). The generative model $G$ depicts the distribution $P_{G}$, and is supposed to fool $D$ with generated data sampled from $P_{G}$ by optimizing $F$ to minimize the distance between $P_{G}$ and the real data distribution $P_{\mathtt{data}}$. Thus the generic GAN can be defined as  $\min_{\x}\max_{\y}V(\x,\y) =-\log(D(\y;\u))
- \log(1-D(\y;G(\x;\v)))$, where $V(\x,\y)$ denotes a joint loss function to optimize both objectives, $\u$ and $\v$ are sampled data from some real data distribution $P_{\text{data}}$ and generated fake distribution $P_{G}$, respectively.

Typically speaking, $G(\x;\cdot)$ first generates $\v$, and then $D(\y;\cdot)$ corresponds with predicted probability that $\v$ is sampled from $P_{data}$. Based on this, most related methods, e.g.,Vanilla GAN (VGAN)~\cite{goodfellow2014generative}  and Wasserstein GAN (WGAN)~\cite{ArjovskyWGAN}, alternatively train $G$ and $D$, and simply treat them on an equal footing, which ignores the intrinsic leader-follower relationship between $G(\x;\cdot)$ and $D(\y;\cdot)$. In this work, we first refer to the reformulation in~\cite{liu2021investigating} to depict the original GAN as below:
\begin{equation}
	\begin{aligned}
		F(\mathbf{x}, \mathbf{y}) &=\log (D(\y;G(\x;\mathbf{v}))),\\
		f(\mathbf{x}, \mathbf{y}) &=\log D(\y;\mathbf{u}) +\log (1-D(\y;G(\x;\mathbf{v}))).\\ 
	\end{aligned}\nonumber
\end{equation}
Thus the LL objective $f$ targets on enhancing the ability of $D(\y;\cdot)$ to distinguish the fake data generate by $G(\x;\cdot)$, while the UL objective $F$ optimizes $\x$ to generate more ‘‘real’’ data distribution to confuse the discriminator. Under this BLO formulation, most existing methods with alternating training strategies essentially directly compute the hyper-gradient with the final state of trajectory instead of the sequence, which contains more high-order gradient information.  In comparison, our methods admit multiple solutions and non-convexity of the LL subproblem, thus could solve the above problem with self contained theoretical guarantee.

\section{Experiments}
 In this section, we report the quantitative and visualization results of the numerical examples, typical learning and vision applications and more challenging tasks. We run the experiments of toy examples based on a PC with Intel Core i7-8700 CPU, 32GB RAM and implement the real-world BLO applications on a Linux server with NVIDIA GeForce RTX 2080Ti GPU (12GB VRAM). 
\subsection{Numerical Evaluations}\label{numerical_results}

At first, we validate the provided theoretical results of our AIT based on numerical examples with different LL assumptions. More specifically, we focus on the UL variable (i.e., $\x $) and approximative UL subproblem ($\Phi$ is used as unified denotation for $\varphi_{K}(\x)$ in Eq.~\eqref{single_re_reform}, $ \varphi_{K}(\x,\z)$ in Eq.~\eqref{eq:proximal_prior} and $\phi_{\bar{K}}(\x,\z)$ in Eq.~\eqref{eq:upper_phiK} under convex and non-convex scenarios, respectively), and analyze the convergence behavior of $\left\|\Phi-\Phi^{*}\right\|/(\left\|\Phi^{*}\right\|+1)$ ($\left\|\Phi^{*}\right\|$ is equal to $0$ in some cases) and $\Vert \x - \x^*\Vert/ \Vert \x^*\Vert$  during the UL optimization. We use the numerical example under LLC assumption to show the effectiveness of  $\textrm{AIT}_{C}$ with IA and the prior regularization term, and compare the convergence behavior of $\textrm{AIT}_{C}$ with a series of mainstream GBMs. Next, we use another example with the well defined LLS assumption to implement our $\textrm{AIT}_{C}$ with IA and the Nesterov's acceleration dynamics. 

Moving forward to the toy example with LL non-convexity, we first verify the effectiveness of $\textrm{AIT}_{NC}$ with IA, PTT and the acceleration gradient scheme, and compare $\textrm{AIT}_{NC}$ with the above GBMs to show its great improvement. Furthermore, we depict the loss surfaces and the derived solutions to show the significant improvement of $\textrm{AIT}_{NC}$  under the non-convex scenario, and analyze the running time influenced by different factors, including initialization points, dimension of LL variables (i.e., $n$) and LL iterations (i.e, $K$). 

\subsubsection{Convex Scenario}

%Since the numerical performance on nonconvex cases has been verifed
We first consider the LLC scenario and conduct the numerical experiments with the following example from~\cite{liu2020generic}. With $\mathbf{x} \in \mathbb{R}^{n}, [\mathbf{y}]_{1} \in \mathbb{R}^{n} $ and $ [\mathbf{y}]_{2} \in \mathbb{R}^{n}$:
\begin{equation}
	\begin{aligned}
		&\min _{\mathbf{x} \in \mathcal{X}}\|\mathbf{x}-[\mathbf{y}]_{2}\|^{4}+\|[\mathbf{y}]_{1}-\mathbf{e}\|^{4}, \\
		&\text { s.t. }([\mathbf{y}]_{1}, [\mathbf{y}]_{2}) \in \underset{[\mathbf{y}]_{1} \in \mathbb{R}^{n}, [\mathbf{y}]_{2} \in \mathbb{R}^{n}}{\operatorname{argmin}}\frac{1}{2}\|[\mathbf{y}]_{1}\|^{2}-\mathbf{x}^{\top} [\mathbf{y}]_{1},
	\end{aligned}
\end{equation}
where $\mathcal{X}=[-10,10] \times \cdots[-10,10] \subset \mathbb{R}^{n}$, and $\mathbf{e}$ represents the vector of which the elements are all equal to 1.  Note that we use bold text to represent scalars for all the following description. We set $n=20$, $T=1000$ and choose $(\x_{0},\y_{0})=(2\mathbf{e},(2\mathbf{e},2\mathbf{e}))$ as the initialization point. By simple calculation, we can obtain the unique optimal solution as $\mathbf{x}^{*}=\mathbf{e}$, $\mathbf{y}^{*}=(\mathbf{e},\mathbf{e})$. This numerical example satisfies the LLC assumption for BDA while violating the LLS assumption admitted by most GBMs such as RHG, CG~\cite{pedregosa2016hyperparameter} and Neumann~\cite{lorraine2020optimizing}.  

To investigate the effectiveness of Alg.~\ref{alg:AIT-C}, we take BDA as the baseline, and implement our $\textrm{AIT}_{C}$ by incrementally adding IA and the prior regularization. In the first two subfigures of Fig.~\ref{toy_convergence_LLC}, we analyze the convergence behavior of  $\left\|\Phi-\Phi^{*}\right\|/(\left\|\Phi^{*}\right\|+1)$ and $\Vert \x - \x^*\Vert/ \Vert \x^*\Vert$. As it has been proved in~\cite{liu2021investigating}, the convergence property of BDA can be consistently verified and obtain expected convergence results for this BLO problem. As for the augmentation technique, the embedded IA of $\textrm{AIT}_{C}$ helps dynamically adjust the initialization of $[\y]_1$ and $[\y]_2$, and the augmented optimization trajectory further improves the convergence speed of UL variable (i.e., $\x$) and approximative  UL subproblem (i.e., $\Phi$). Besides, with prior regularization  to guide the optimization process of $\z$, our $\textrm{AIT}_{C}$ can obtain higher convergence precision and speed.

In the last two subfigures of Fig.~\ref{toy_convergence_LLC}, we compare our $\textrm{AIT}_{C}$ with series of mainstream GBMs including CG, Neumann, RHG, and BDA. Since the LLS assumptions of CG, Neumann and RHG can not be satisfied, no convergence property for these methods could be guaranteed in this case. In comparison with BDA, $\textrm{AIT}_{C}$ also gains advantage over convergence speed and precision. 
%\subsubsection{Convex Scenario}

Furthermore, to demonstrate the performance of other iterative formats and acceleration dynamics, we further introduce another toy example with more strict LL property, i.e, the well-defined LLS assumption:
\begin{equation}
	\begin{aligned}
		&\min _{\mathbf{x} \in \mathcal{X}}-\mathbf{e}^{T} \mathbf{x}+\|\mathbf{x}-[\mathbf{y}]_{2}\|^{4}+\|[\mathbf{y}]_{1}-\mathbf{e}\|^{4}, \\
		&\text { s.t. }([\mathbf{y}]_{1}, [\mathbf{y}]_{2}) \in \underset{[\mathbf{y}]_{1} \in \mathbb{R}^{n}, [\mathbf{y}]_{2} \in \mathbb{R}^{n}}{\operatorname{argmin}}\|[\mathbf{y}]_{1}+[\mathbf{y}]_{2}-\mathbf{x}\|^{4},
	\end{aligned}
\end{equation}
where $\mathbf{x} \in \mathbb{R}^{n}, [\mathbf{y}]_{1} \in \mathbb{R}^{n} $, $ [\mathbf{y}]_{2} \in \mathbb{R}^{n}$, $\mathcal{X}=[-1,1] \times \cdots \times[-1,1] \subset \mathbb{R}^{n}$. Given any $\x\in \X$, it satisfies $\min _{[\mathbf{y}]_{1} \in \mathbb{R}^{n}, [\mathbf{y}]_{2} \in \mathbb{R}^{n}}\|[\mathbf{y}]_{1}+[\mathbf{y}]_{2}-\mathbf{x}\|^{4}=0$, then the unique optimal solution of this BLO problem is $\x^{*}=\mathbf{e}$,  $\mathbf{y}^{*}=(\frac{1}{2}\mathbf{e},\frac{1}{2}\mathbf{e})$. Since the solution set of this LL subproblem (i.e., $\mathcal{S}(\x)$) is a singleton, we can verify that the above example can be properly solved with GBMs under LLS assumption.

\begin{figure}[htbp]
	\begin{center}
		\begin{tabular}{c@{\extracolsep{0.1em}}c@{\extracolsep{0.1em}}}
			\includegraphics[height=3.2cm,width=4.2cm,trim= 0 0 0 0,clip]{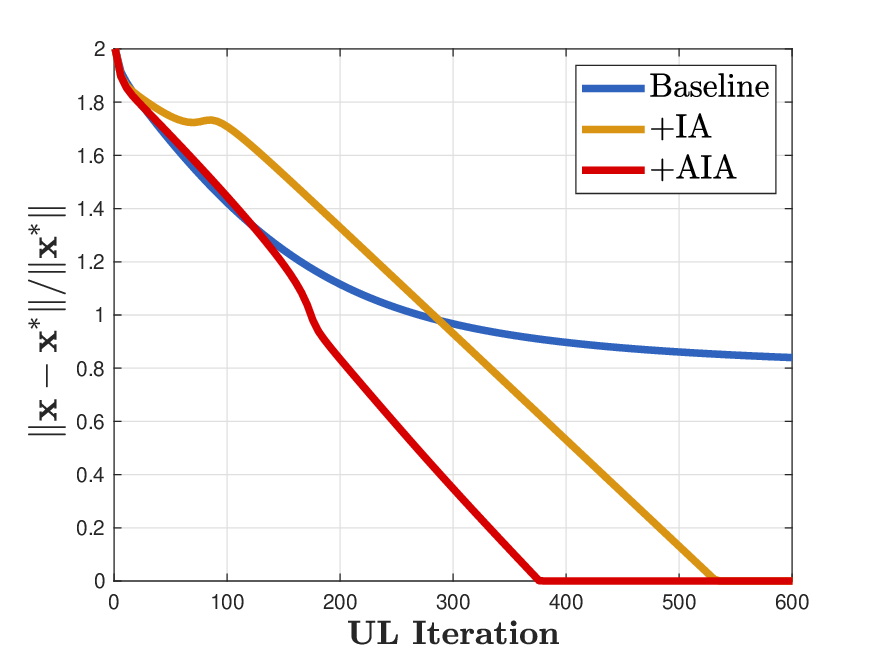}&\includegraphics[height=3.2cm,width=4.2cm,trim= 0 0 0 0,clip]{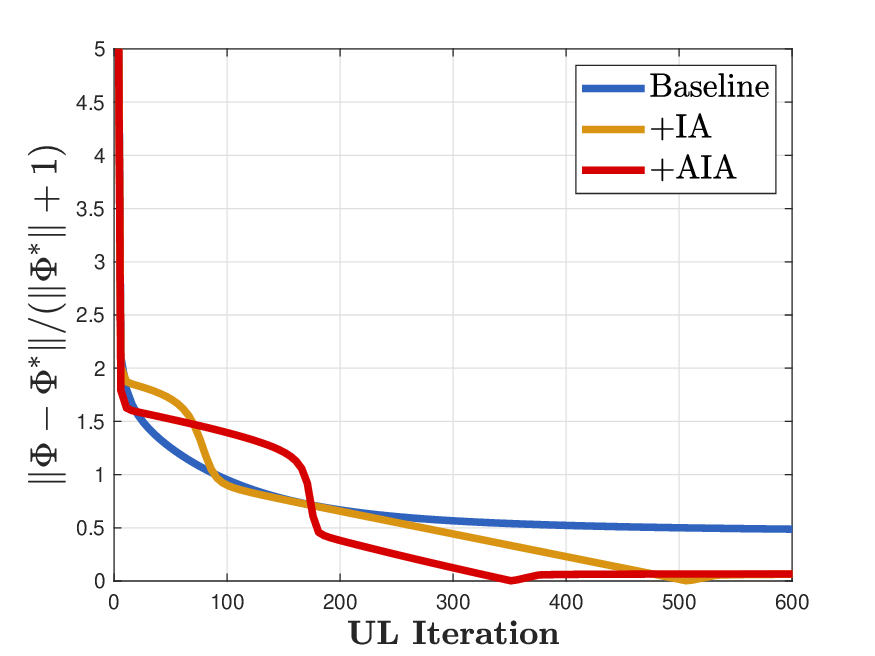}\\
			\footnotesize (a) & \footnotesize (b) \\
		\end{tabular}
	\end{center}
	\caption{Illustrating the convergence curve of $\left\|\Phi-\Phi^{*}\right\|/(\left\|\Phi^{*}\right\|+1)$ and $\Vert \x - \x^*\Vert/ \Vert \x^*\Vert$ for the baseline (i.e., RHG) and variations with techniques of $\textrm{AIT}_{C}$, including IA (i.e., $+\textrm{IA}$) and Nesterov's acceleration dynamics (i.e., $+\textrm{A}$) during the LL optimization. The initialization point is chosen as $(\x_{0},\y_{0})=(-\mathbf{e},(-\mathbf{e},-2\mathbf{e}))$.}\label{toy_convergence_LLS}
\end{figure}

In this case, we implement our $\textrm{AIT}_{C}$ with RHG as the baseline, and further incorporate the Nesterov's acceleration dynamics to facilitate the convergence behavior. When an appropriate initialization position is chosen, commonly used GBMs are also supposed to converge with acceptable speed. To demonstrate the effectiveness of IA, we set the initialization point as $(\x_{0},\y_{0})=(-\mathbf{e},(-\mathbf{e},-2\mathbf{e}))$, where $\z$ is further away from the optimal solution $\z^{*}$. In Fig~\ref{toy_convergence_LLS}, it can be seen that the IA technique could effectively accelerate the convergence of UL variables towards the global optimal solution even when RHG converges slower due to improper initialization of LL variable. In addition, the Nesterov's strategy further helps the LL variable (i.e., $\y$) converge faster with embedded acceleration dynamics.

\subsubsection{Non-Convex Scenario}

\begin{figure*}[h!]
	\begin{center}
		\begin{tabular}{c@{\extracolsep{0.1em}}c@{\extracolsep{0.1em}}c@{\extracolsep{0.1em}}c@{\extracolsep{0.1em}}c@{\extracolsep{0.1em}}}
			\includegraphics[height=3.2cm,width=4.2cm,trim=0 0 0 0,clip]{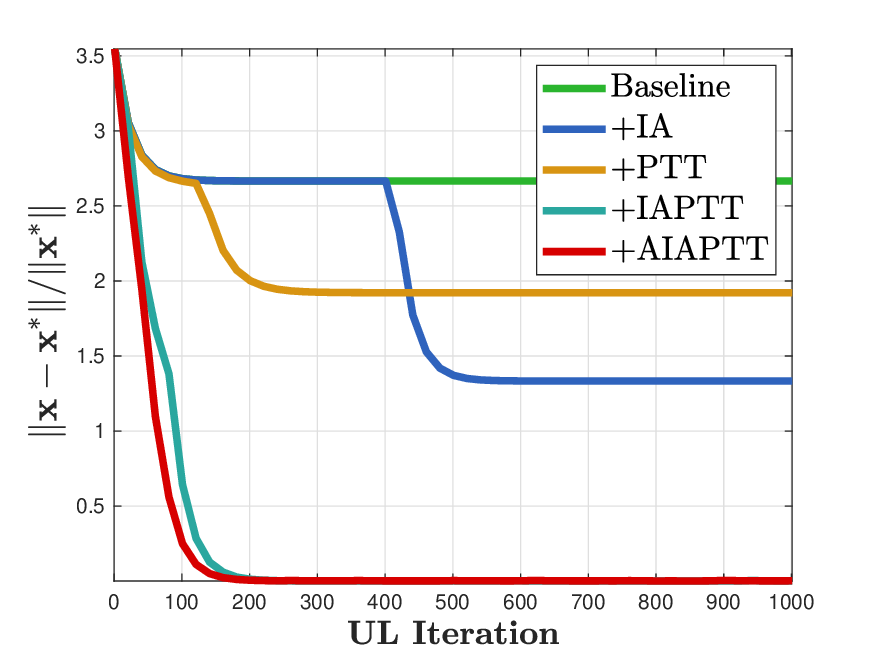} &\includegraphics[height=3.2cm,width=4.2cm,trim=0 0 0 0,clip]{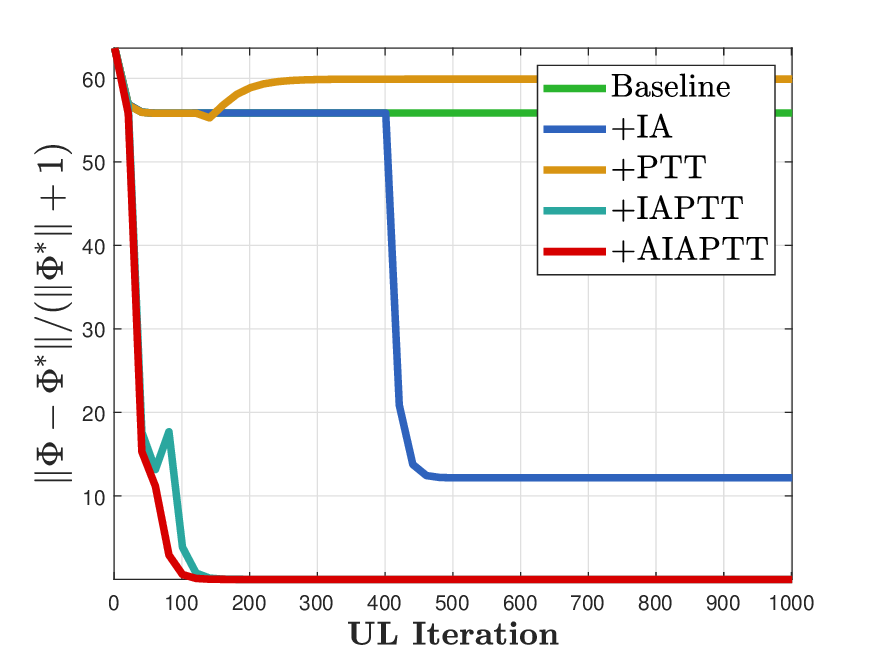} &                 
			
			\includegraphics[height=3.2cm,width=4.2cm,trim=0 0 0 0,clip]{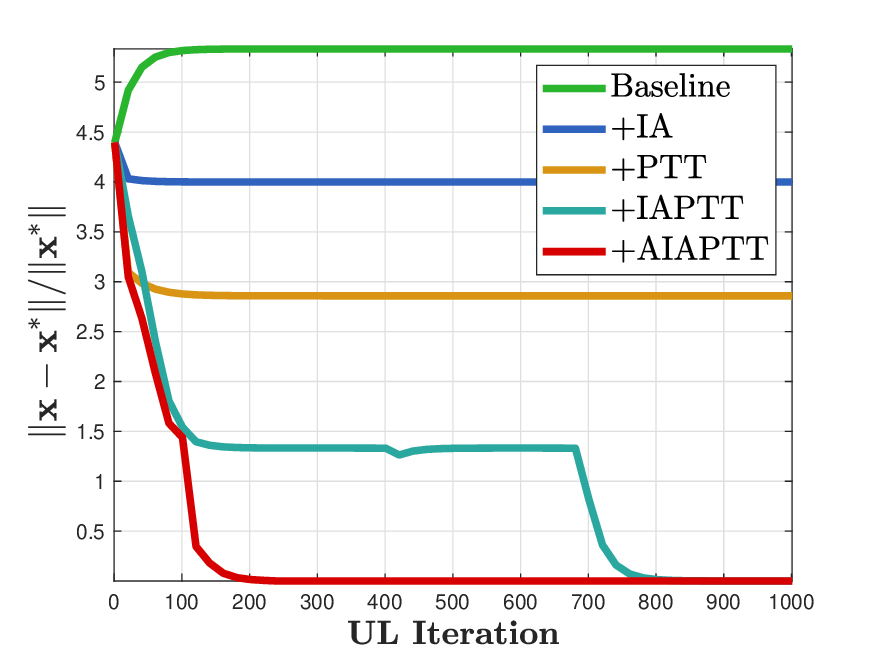} &\includegraphics[height=3.2cm,width=4.2cm,trim=0 0 0 0,clip]{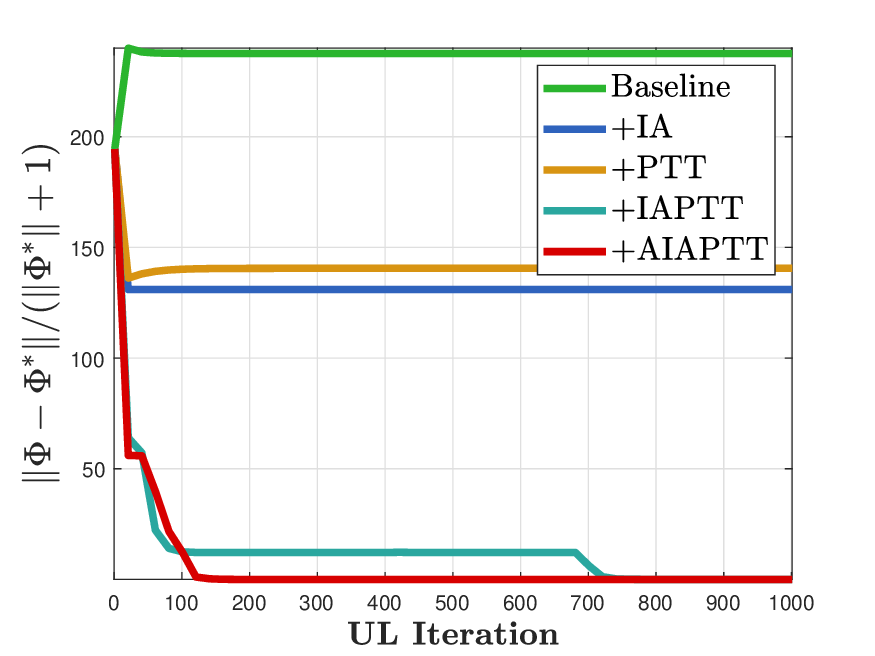}\\
			\multicolumn{2}{c}{\footnotesize(a)~$(\x_{0},\y_{0})=(-6,0)$} & 
			\multicolumn{2}{c}{\footnotesize(b)~$(\x_{0},\y_{0})=(-8,-8)$}\\
		\end{tabular}
	\end{center}
	%\vspace{-0.4cm}
	\caption{We plot the convergence behavior of $\left\|\Phi-\Phi^{*}\right\|/(\left\|\Phi^{*}\right\|+1)$ and $\Vert \x - \x^*\Vert/ \Vert \x^*\Vert$ for the baseline (i.e., RHG) and the variation with IA (i.e., $+\textrm{IA}$), PTT (i.e., $+\textrm{PTT}$) and the acceleration gradient scheme (i.e., $+\textrm{A}$).}\label{toy_convergence_nonconvex_1}
\end{figure*}

\begin{figure*}[h!]
	\begin{center}
		\begin{tabular}{c@{\extracolsep{0.1em}}c@{\extracolsep{0.1em}}c@{\extracolsep{0.1em}}c@{\extracolsep{0.1em}}c@{\extracolsep{0.1em}}}
			\includegraphics[height=3.2cm,width=4.2cm,trim=0 0 0 0,clip]{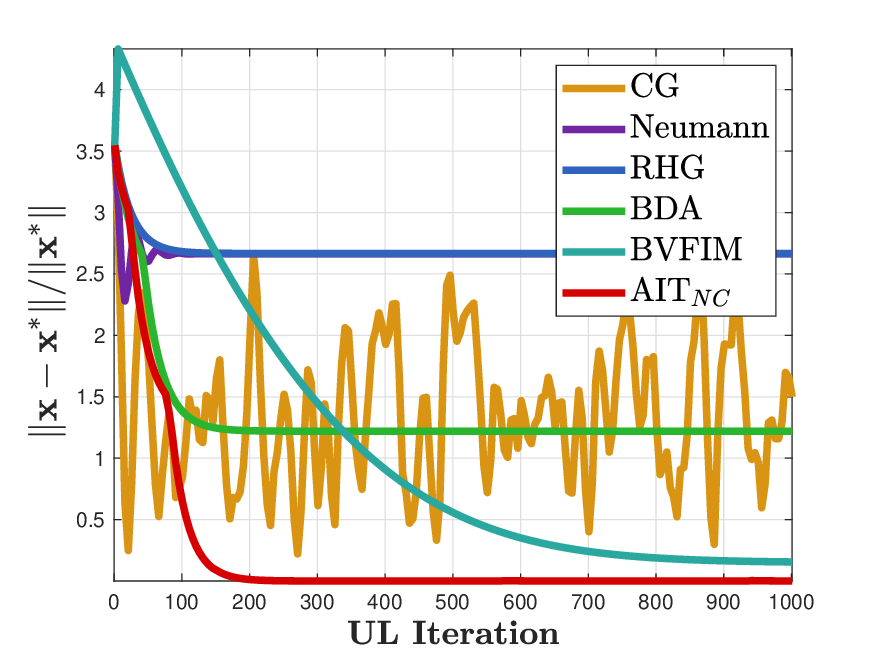} &\includegraphics[height=3.2cm,width=4.2cm,trim=0 0 0 0,clip]{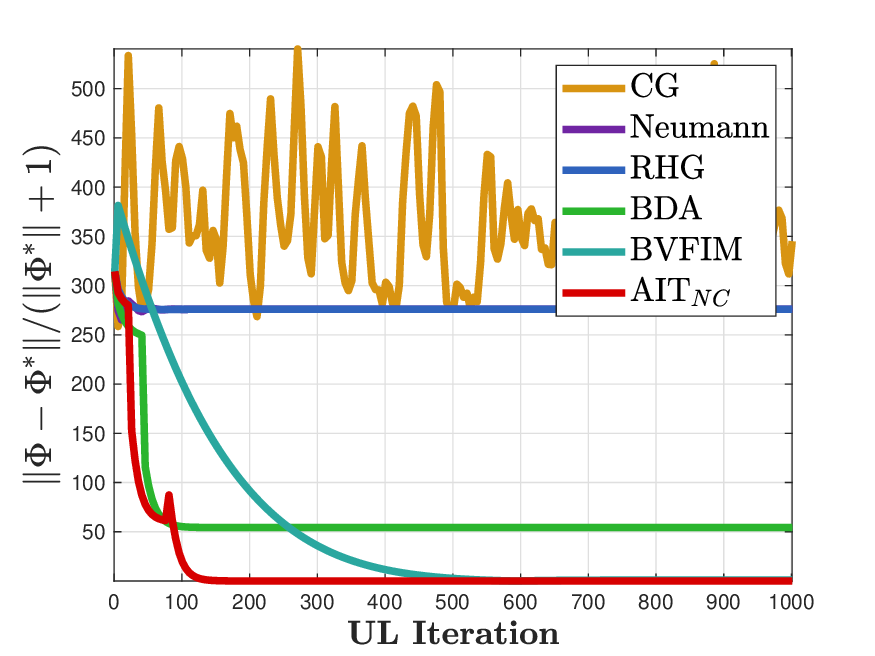} && \includegraphics[height=3.2cm,width=4.2cm,trim=0 0 0 0,clip]{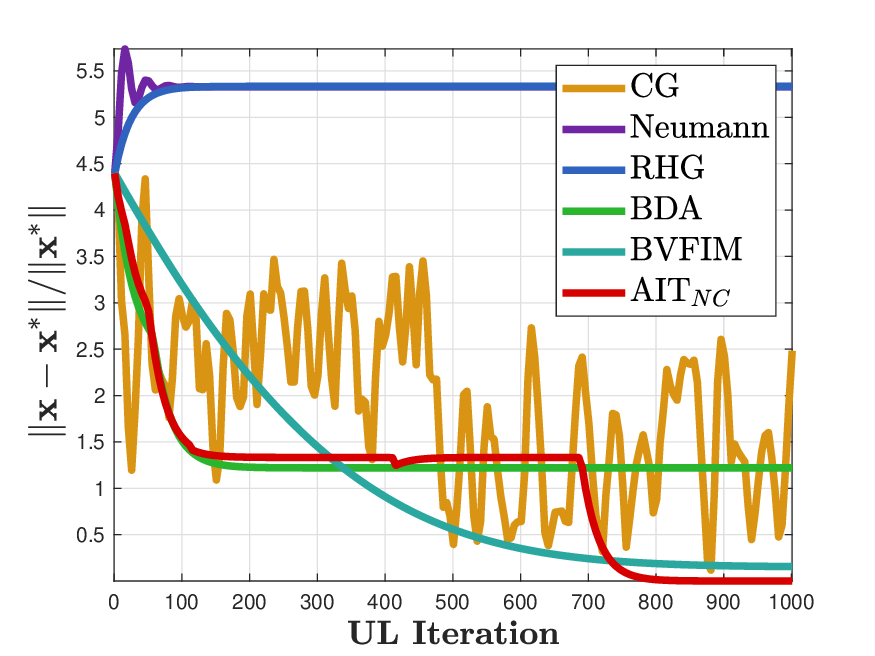} &\includegraphics[height=3.2cm,width=4.2cm,trim=0 0 0 0,clip]{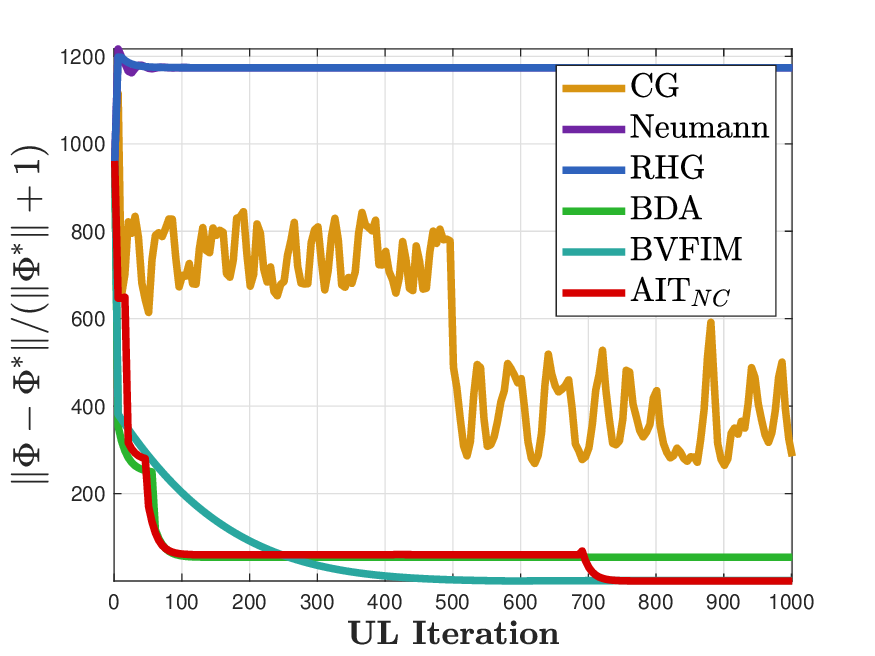}\\
			\multicolumn{2}{c}{\footnotesize(a)~$(\x_{0},\y_{0})=(-6,0)$} && 
			\multicolumn{2}{c}{\footnotesize(b)~$(\x_{0},\y_{0})=(-8,-8)$}\\
		\end{tabular}
	\end{center}
	%\vspace{-0.4cm}
	\caption{In this figure, we compare the convergence behavior of $\left\|\Phi-\Phi^{*}\right\|/(\left\|\Phi^{*}\right\|+1)$ and $\Vert \x - \x^*\Vert/ \Vert \x^*\Vert$ for our $\textrm{AIT}_{NC}$ and series of mainstream GBMs including CG, Neumann, RHG, BDA and BVFIM.}\label{toy_convergence_nonconvex_3}
\end{figure*}

To validate the superiority of AIT, we move on to more challenging scenario and introduce another toy example~\cite{liu2021value} with non-convex LL subproblem, which takes the form of:
\begin{equation}
	\begin{aligned}
		&\min _{\mathbf{x} \in \mathbb{R}, \mathbf{y} \in \mathbb{R}^{n}}\|\mathbf{x}-a\|^{2}+\|\mathbf{y}-a-\mathbf{c}\|^{2}, \\
		&\text { s.t. }[\mathbf{y}]_{i} \in \underset{[\mathbf{y}]_{i} \in \mathbb{R}}{\operatorname{argmin}} \sin \left(\mathbf{x}+[\mathbf{y}]_{i}-[\mathbf{c}]_{i}\right), \forall i,
	\end{aligned}
\end{equation}
where both $a\in\mathbb{R}$ and $\c\in\mathbb{R}^{n}$ are constants, and $[\cdot]_{i}$ denotes the i-th element of the vector. Then the optimal solution of this problem can be calculated in the following form:
$$
\mathbf{x}=\frac{(1-n) a+n C}{1+n}, \text { and }[\mathbf{y}]_{i}=C+[\mathbf{c}]_{i}-\mathbf{x}, \forall i, 
$$
where $$
C=\underset{C}{\operatorname{argmin}}\left\{\|C-2 a\|: C=-\frac{\pi}{2}+2 k \pi, k \in \mathbb{Z}\right\},
$$ and the corresponding value of UL objective is $
F^{*}=\frac{n(C-2 a)^{2}}{1+n}
$. Here we set $a=2$, $[\c]_{i}=2, i=1,...,n$ and $[\y]_{i}\in[-10, 10]$. It can be easily verified that this example satisfies the theoretical assumption of $\textrm{AIT}_{NC}$ while violating the LLS assumption for RHG, CG, Neumann and the LLC assumption for BDA due to existence of multiple local optimal solutions of the LL subproblem.  At first, we set $n=1$ to investigate the convergence property based on the simplified form and draw the loss surfaces.  Then we use larger $n$ when analyzing the influence of dimension of LL variables on the running time. We can calculate the corresponding optimal solution as $(\x^{*},\y^{*})=(3\pi/4, 3\pi/4 + 2)$. In consideration of distance between the initialization position and the global optimal solution, we choose two initialization points for $\x$ and $\y$ including $(\x_{0},\y_{0})=(-6, 0)$ and $(\x_{0},\y_{0})=(-8,-8)$. 

\begin{figure}[!h]
	\begin{center}
		\renewcommand\arraystretch{0.1}
		\begin{tabular}{c@{\extracolsep{0.1em}}c@{\extracolsep{0.1em}}c@{\extracolsep{0.1em}}}
			\multicolumn{3}{c}{\includegraphics[height=0.75cm,width=9cm,trim=20 250 30. 0,clip]{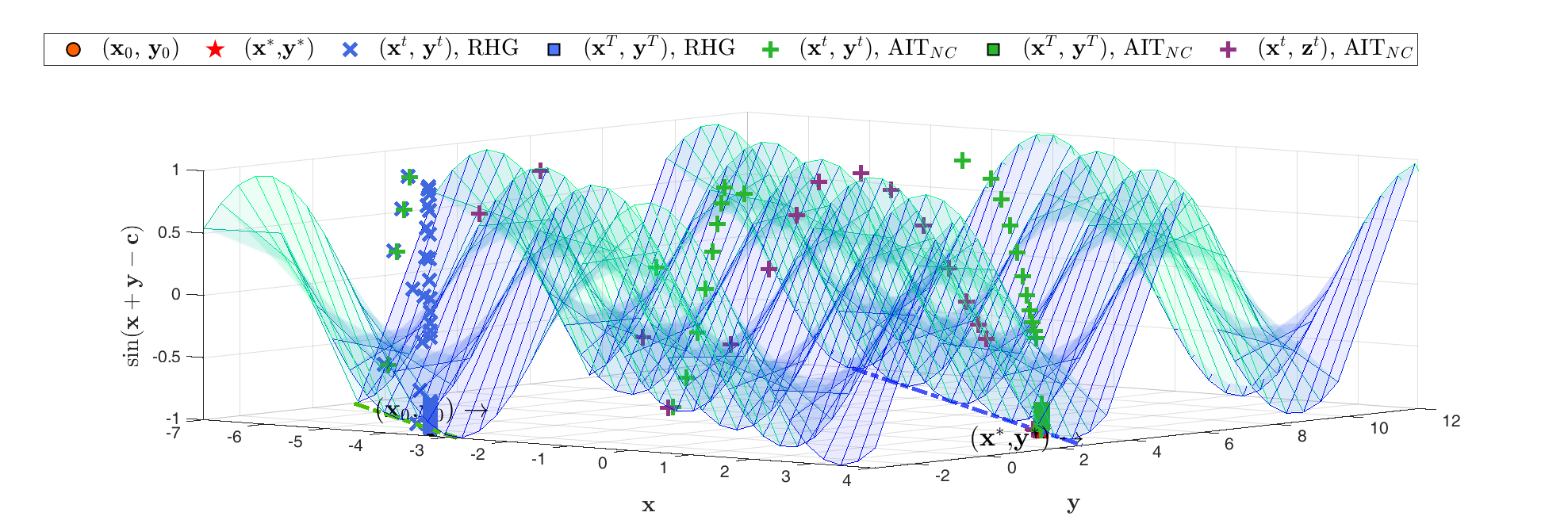}}\\
			\specialrule{0em}{-4pt}{-4pt}
			\multirow{3}{*}{\includegraphics[height=6.3cm,width=4.2cm,trim=0 28 20 10,clip]{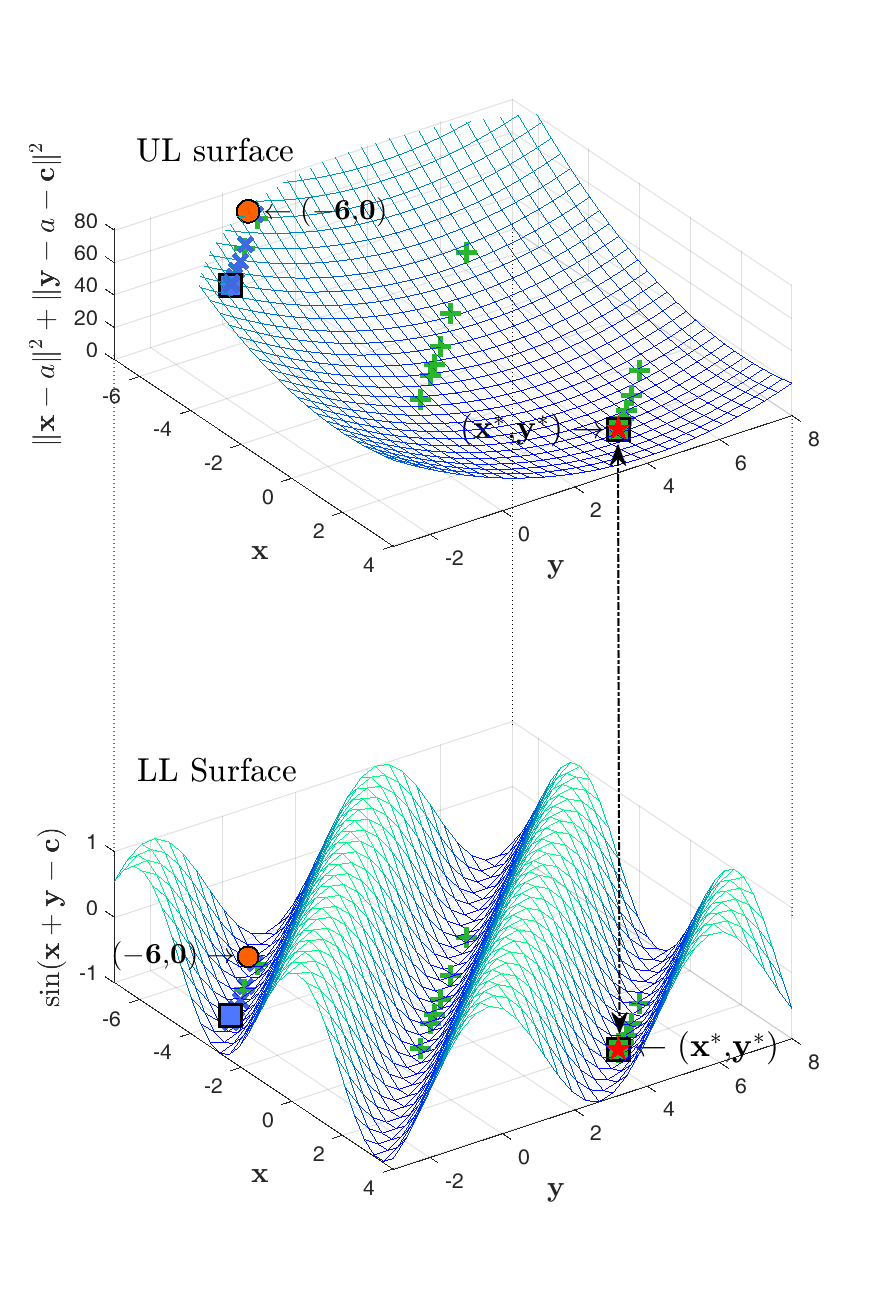}}&& \\
			&&	\includegraphics[height=2.9cm,width=4.2cm,trim= 0 0 0. 0,clip]{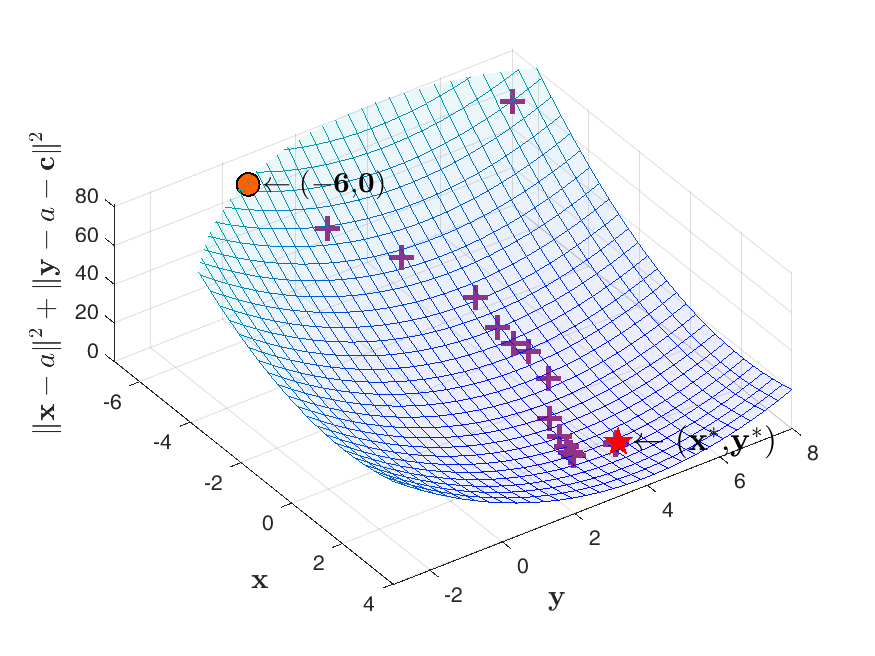}\\
			\specialrule{0em}{0pt}{0pt}
			&&\footnotesize (b) UL Trajectory\\
			&&\includegraphics[height=3cm,width=4.2cm,trim= 0 0 0 0,clip]{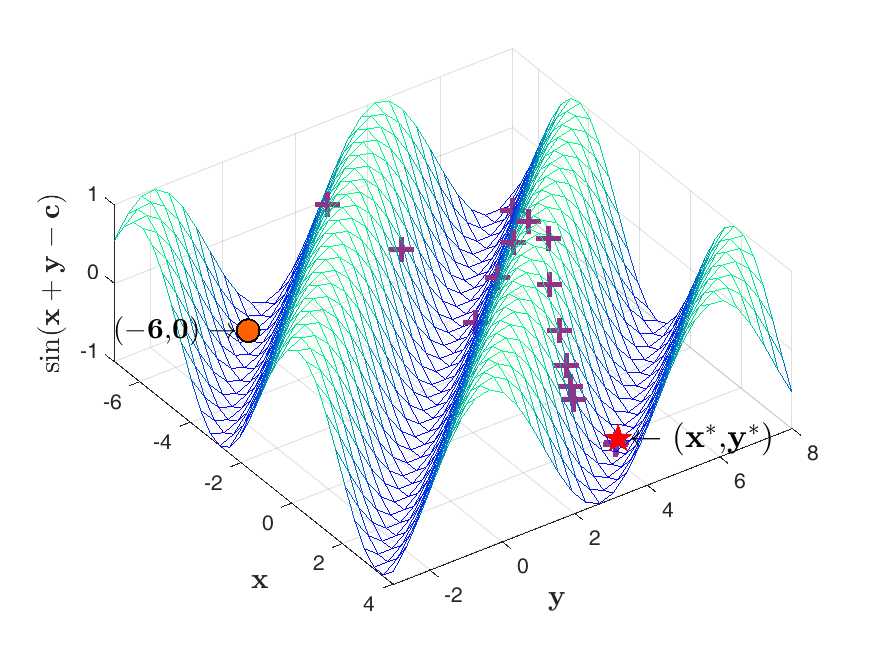}\\
			\specialrule{0em}{0pt}{0pt}
			\footnotesize (a) Loss Surface &&
			\footnotesize (c) LL Trajectory\\
		\end{tabular}
	\end{center}
	\caption{In the left subfigure, we visualize the loss surfaces of the UL and LL subproblems and compare the optimization trajectories of RHG and $\textrm{AIT}_{NC}$, both initialized at the same point $(\x_{0},\y_{0}) = (-6, 0)$. In the two subfigures on the right, we further illustrate the optimization process of $\textrm{AIT}_{NC}$ with IA and PTT by plotting the trajectory of $(\x^{t}, \z^{t})$ on both UL and LL surfaces during the UL update. }\label{toy_convergence_nonconvex_2}
\end{figure}

In Fig.~\ref{toy_convergence_nonconvex_1}, we take RHG as the baseline, and start by incrementally embedding two augmentation techniques and the acceleration gradient scheme of $\textrm{AIT}_{NC}$ into the baseline. It can be seen that adding PTT technique shows consistently better convergence results with respective of $\x$, while the convergence of $\phi_{\bar{K}}(\x,\z)$ may be influenced by different initialization points without dynamically optimized $\z$. When we implement our $\textrm{AIT}_{NC}$ with both techniques, since the theoretical convergence property is guaranteed, our algorithm can uniformly improve the convergence behavior of $\phi_{\bar{K}}(\mathbf{x}, \mathbf{z})$ and $\x$ on this non-convex problem. Furthermore, by embedding the acceleration gradient scheme in Eq.~\eqref{yk_def3} under our $\textrm{AIT}_{NC}$, the convergence rate of gradient descent for this non-convex BLO problem could be further improved.  

Furthermore, In Fig.~\ref{toy_convergence_nonconvex_3}, we compare $\textrm{AIT}_{NC}$ with mainstream GBMs, including CG, Neumann, RHG, BDA and BVFIM with the same initialization points. As it is expected, all the above methods except BVFIM can be easily stuck at the local optimal solution. Besides, BVFIM may not converge exactly to the true optimal solution influenced by its sensitivity to the added penalty terms. In comparison, under our $\textrm{AIT}_{NC}$ algorithm, by dynamically adjusting the initial value of LL variables and truncation of $\mathcal{T}$ for computation of hyper-gradients, the convergence results without LLC assumption can be consistently verified. 
\begin{figure}[h!]
	\begin{center}
		\begin{tabular}{c@{\extracolsep{0.1em}}c@{\extracolsep{0.1em}}c@{\extracolsep{0.1em}}}
			& \includegraphics[height=3.2cm,width=4.2cm,trim=0 0 0 0,clip]{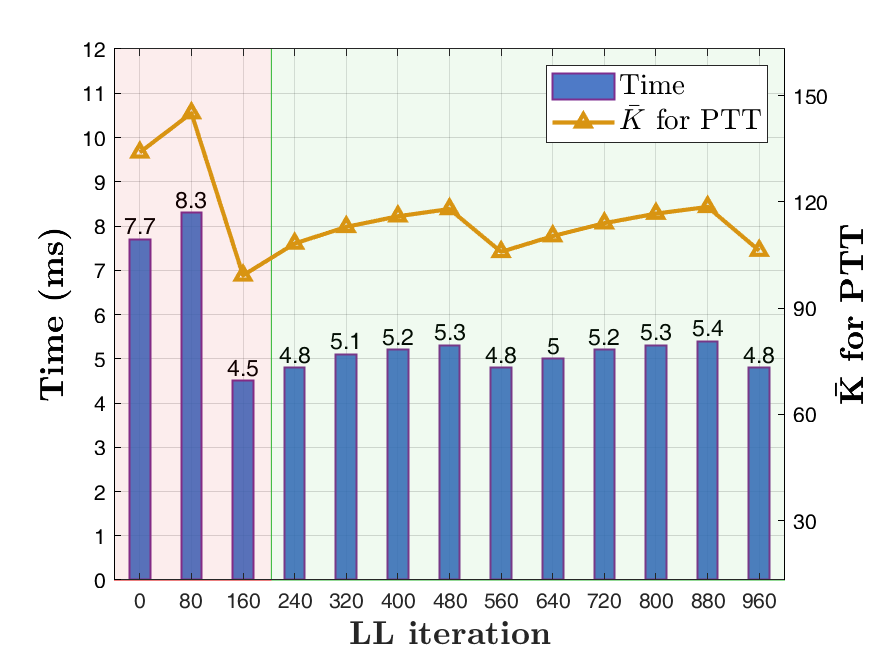} &\includegraphics[height=3.2cm,width=4.2cm,trim=0 0 0 0,clip]{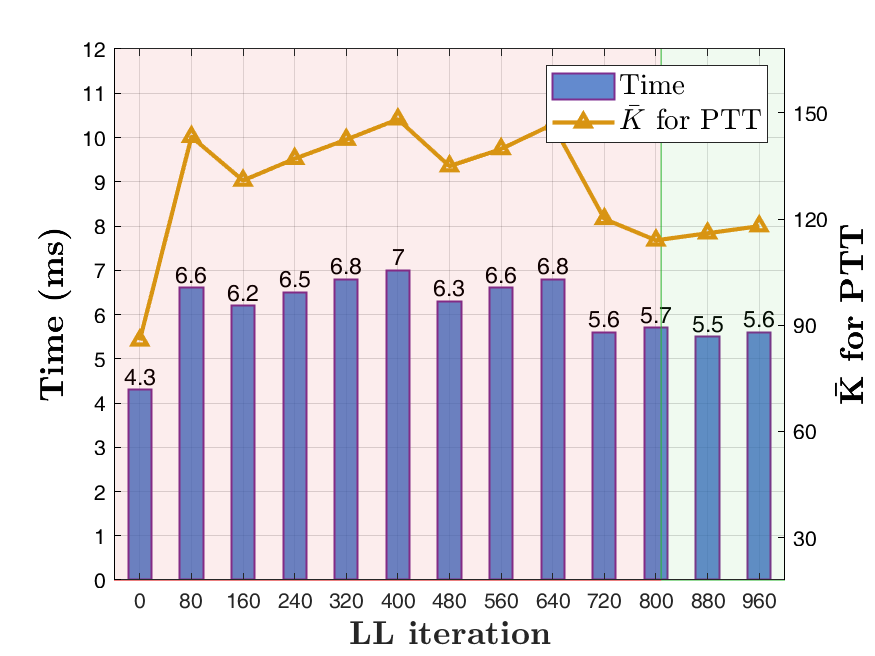}\\
			& \footnotesize (a)~$(\x_{0},\y_{0})=(-6,0)$ &\footnotesize (b)~$(\x_{0},\y_{0})=(-8,-8)$ \\
			
		\end{tabular}
	\end{center}
	%\vspace{-0.4cm}
	\caption{In the two subfigures, we report average $\bar{K}$ and running time per iteration of our $\textrm{AIT}_{NC}$ during the LL optimization with different initialization points. Note that the part with green background represents that the convergence curve comes to stabilize ($\Vert \x - \x^*\Vert <1\times e^{-3}$). }\label{toy_convergence_nonconvex_4}
\end{figure}

More vividly, in Fig.~\ref{toy_convergence_nonconvex_2}, we mark the initialization point (i.e.,~$(\x_{0},\y_{0})$), optimal solution (i.e.,~$(\x^{*},\y^{*})$), iterative solution (i.e.,~$(\x^{t},\y^{t})$), and convergence solution (i.e.,~$(\x^{T},\y^{T})$) for both RHG and $\textrm{AIT}_{NC}$ on the UL and LL loss surfaces. As confirmed by the convergence analysis, due to the existence of many local optima along the path from $(\x_{0},\y_{0})$ to $(\x^{*},\y^{*})$, RHG's iterative trajectory tends to get trapped in suboptimal regions, ultimately leading to a solution far from the global optimum.  For $\textrm{AIT}_{NC}$, the cross-shaped markers indicate the initialization points $(\x^t, \z^t)$ produced by IA, which guide the trajectory across local valleys and ridges toward the global optimum. With the help of $\z^t$ and the truncated trajectory length $\bar{K}$, $\textrm{AIT}_{NC}$ yields a better estimation of $\y_{\bar{K}}(\x^t,\z^t)$ and thus a more accurate hyper-gradient. Consequently, the resulting $(\x^t, \y^t)$ avoids poor local minima and successfully converges under the non-convex landscape, demonstrating the effectiveness of $\textrm{AIT}_{NC}$.

\begin{figure}[h!]
	\begin{center}
		\begin{tabular}{c@{\extracolsep{0.1em}}c@{\extracolsep{0.1em}}c@{\extracolsep{0.1em}}}
			&\includegraphics[height=3.2cm,width=4.2cm,trim=0 0 0 0,clip]{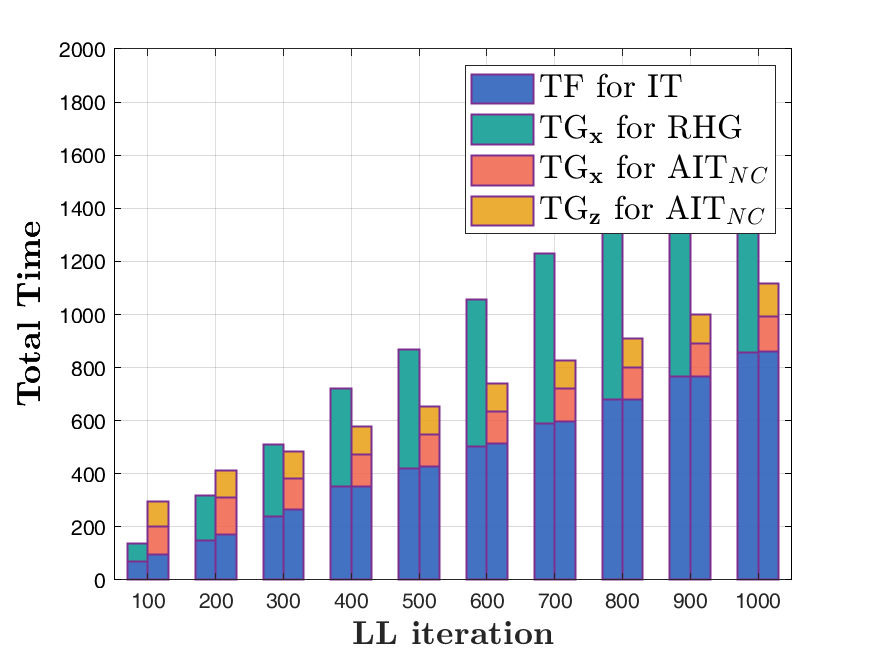} &\includegraphics[height=3.2cm,width=4.2cm,trim=0 0 0 0,clip]{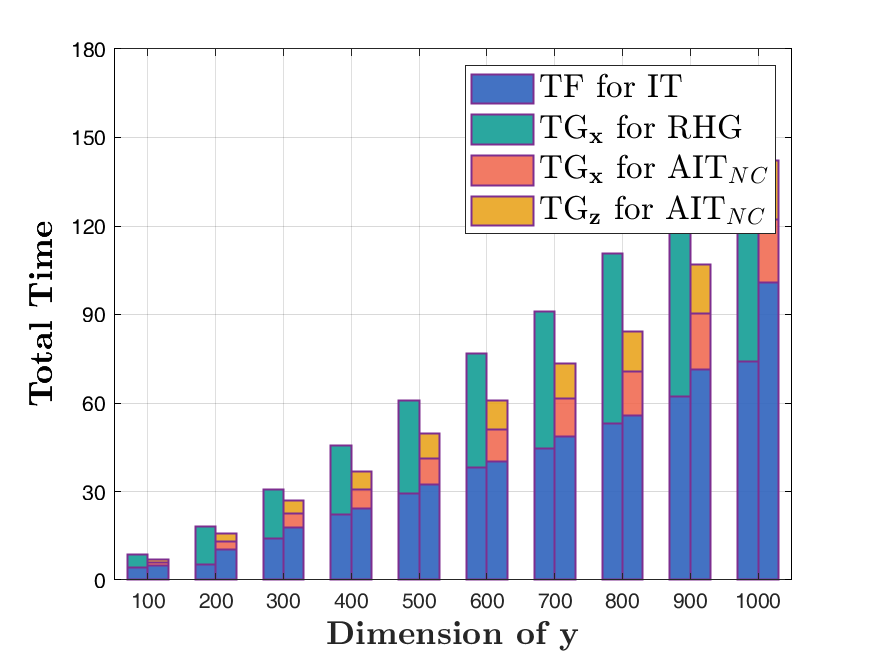}\\
			&\footnotesize (a)~Influence of $K$ &\footnotesize (b) Influence of $n$\\
			
		\end{tabular}
	\end{center}
	%\vspace{-0.4cm}
	\caption{The two subfigures report the total time of RHG and $\textrm{AIT}_{NC}$ for a single UL iteration as the LL iteration $K$ or dimension of $\y$ increases. “IT” refers to the iterative trajectory; $\textrm{TG}_{(\cdot)}$ denotes the time required for computing the gradients of the corresponding variables, and $\textrm{TF}$ represents the forward propagation time used to construct $\mathcal{T}$. Note that we use the finally stabilized $\bar{K}$ for $\textrm{AIT}_{NC}$ and run 5 steps to calculate the average time for comparison.}\label{toy_convergence_nonconvex_5}
\end{figure}

In Fig.~\ref{toy_convergence_nonconvex_4}, we report the average running time of hyper-gradient computation with respective to different initialization points in this numerical case. It can be observed that as the PTT operation of $\textrm{AIT}_{NC}$  dynamically truncates the optimization trajectory and selects $\bar{K}\in[1,K]$ after $K$ LL iterations of LL optimization, the computational burden of hyper-gradient w.r.t. $\x$ and $\z$ is continuously eased. In addition, we found that as the curve converges with favorable convergence speed in Fig.~\ref{toy_convergence_nonconvex_3}, $\bar{K}$ tends to stabilize around a fixed number independent of the chosen initialization points. 

To evaluate the performance for high-dimensional BLO problems, we further investigate how dimension of LL subproblems (i.e., $n$) and LL iteration (i.e., $K$) influence the computation efficiency. Known that the computation burden has been evaluated to be less dependent on the dimension of UL subproblem when calculating the hyper-gradient with AD~\cite{franceschi2017forward} (adopted by classical GBMs, e.g., RHG), so in Fig.~\ref{toy_convergence_nonconvex_5}, we continuously increase the dimension of $\y$ and record the average time for gradient computation of RHG and $\textrm{AIT}_{NC}$. We can observe that $\bar{K}\in[1,K]$ always decreases the time cost for the UL optimization process compared with RHG, even extra computation for $\bar{K}$ and hyper-gradient of $\z$ is introduced in the $\textrm{AIT}_{NC}$. Besides, we also evaluate the runtime of RHG and $\textrm{AIT}_{NC}$ with increasing $K$. Compared with $n$, the increasing $K$ has more influence on the computation for iterative optimization of $\y_{k}$ and hyper-gradient of $\x$. Accordingly, the PTT technique chooses smaller $\bar{K}$ for backpropagation ( Step \ref{outer_loop_2}-\ref{outer_loop_3} in Alg.~\ref{alg:AIT-N}), thus reduces the computation burden of backpropagation (i.e., $\textrm{PT}_{\x}$ and $\textrm{PT}_{\z}$).

\subsection{Typical Learning and Vision Applications}

In the following, we further demonstrate the performance of AIT with typical BLO applications influenced by the nonconvex LL objectives and nonconvex LL network structure, including few-shot learning and data hyper-cleaning.

\subsubsection{Data Hyper-Cleaning}

The goal of data hyper-cleaning is to cleanup the dataset in which the labels have been corrupted. This BLO problem defines the UL variables $\x$ as a vector, and the dimension of $\x$ equals to the number of corrupted samples, while the LL variable $\y$ corresponds to parameters of the classification model. Following the common problem setting of existing works~\cite{shaban2019truncated}, we first split the dataset into three disjoint subsets: $\mathcal{D}_{\mathtt{tr}}$, $\mathcal{D}_{\mathtt{val}}$ and $\mathcal{D}_{\mathtt{test}}$, and then randomly pollute the labels in $\mathcal{D}_{\mathtt{tr}}$ with fixed ratio. Then the UL subproblem can be well defined as
$
F(\mathbf{x}, \mathbf{y})=\sum_{\left(\mathbf{u}_{i}, \mathbf{v}_{i}\right) \in \mathcal{D}_{\mathtt{val}}} \ell\left(\mathbf{y}(\mathbf{x}); \mathbf{u}_{i}, \mathbf{v}_{i}\right), 
$
where $\left(\mathbf{u}_{i}, \mathbf{v}_{i}\right)$ represents the data pairs, $\ell\left(\mathbf{y}(\mathbf{x}); \mathbf{u}_{i}, \mathbf{v}_{i}\right)$ denotes the cross-entropy loss function with classifier parameter $\y$ and data pairs from different subsets. By introducing $\sigma(\mathbf{x})$ as the element-wise sigmoid function to constrain the element in the range of $\left[0,1\right]$, the LL subproblem can be well defined as: $ f(\mathbf{x}, \mathbf{y})=\sum_{\left(\mathbf{u}_{i}, \mathbf{v}_{i}\right) \in \mathcal{D}_{\mathtt{tr}}}[\sigma(\mathbf{x})]_{i} \ell\left(\mathbf{y} ; \mathbf{u}_{i}, \mathbf{v}_{i}\right)$.

\begin{table}[h!]
	\centering
	\caption{With convex network structure, we report results of existing GBMs including CG, Neumann, RHG and BDA for solving data hyper-cleaning tasks. Acc and F1 score denote the test accuracy and the harmonic mean of the precision and recall, respectively. $\textrm{AIT}_{P}$ denotes our AIT with embedded prior regularization.}\label{tab:hyper cleaning}
	\renewcommand\arraystretch{1.2}
	\setlength{\tabcolsep}{1.1mm}{
		\begin{tabular}{|c|c| c|c| c| }
			\hline
			\multirow{2}{*} { Method } & \multicolumn{2}{c|} { MNIST } & \multicolumn{2}{c|}{ CIFAR10} \\
			\cline { 2 - 5}
			& Acc. & F1 score  & Acc. & F1 score \\
			\hline
			CG & $89.12$ & $87.51$ &$36.80$ & $68.56$ \\
			Neumann & $88.90$ & $89.90$ &$35.15$ & $67.79$\\
			RHG & $88.96$ & $90.49$& $34.80$ &$69.75$  \\
			BDA & $88.55$ & $90.99$ &$35.56$ &$69.58$\\
			\hline
			$\textrm{AIT}$ &{$89.58$} &{$91.35$}& $36.95$&{\color{black}$\mathbf{72.19}$} \\
			%	\hline
			$\textrm{AIT}_{P}$&{\color{black}$\mathbf{89.63}$} & {\color{black}$\mathbf{91.98}$}& {\color{black}$\mathbf{38.08}$}&$70.75$\\
			\hline
		\end{tabular}
	}
\end{table}

Typically, to satisfy the LLC assumption admitted by existing GBMs, the LL classification model is usually defined as a single fully-connected layer, and multiple-layer classifier with more parameters are not accessible.  Based on the AIT, the constraints could be continuously relaxed thus we have more choices when designing the LL model.  Besides, the prior regularization also serves as effective tool to help improve the performance of AIT. To find prior determined value $\z_{p}$ for the IA, (i.e., $\z$), we pretrain the LL classification model on $\mathcal{D}_{\textrm{tr}}$ and restore the pretrained model weights as the proximal prior to adjust the updates and facilitate the convergence behavior of $\z$. Then we could specify the form of prior regularization  in Eq.~\eqref{eq:proximal_prior}. Practically, we implement the prior regularization  $g(\z)$ as $\ell_2$ regularization, and add this term only for the first iteration to investigate how this regularization term influence the latter optimization process. Three well known datasets including MNIST, FashionMNIST and CIFAR10 are used to conduct the experiments. Specifically, 5000, 5000, 10000 samples are randomly selected to construct $\mathcal{D}_{\mathtt{tr}}$, $\mathcal{D}_{\mathtt{val}}$ and $\mathcal{D}_{\mathtt{test}}$, then half of the labels in $\mathcal{D}_{\mathtt{tr}}$ are tampered.

We first implement this application with convex LL network structure including a single-layer fully-connected classifier. Then the shape of $\y$ is $\y\in\mathbb{R}^{10\times d}$, where $d$ denotes the dimension of the flattened input data. In Tab.~\ref{tab:hyper cleaning}, we compare our AIT with CG, Neumann, RHG and BDA. As it has been verified in the numerical experiments before, we can observe that AIT helps improve the performance on both Accuracy and F1 score based on MNIST and CIFAR10 dataset. Besides, though the  prior regularization term only exists in the first iteration, it helps correct the initial state of constructed optimization dynamics, and further improves the performance of AIT on both datasets. 

\begin{table}[h!]
	\centering
	\caption{Reporting results of existing methods for solving data hyper-cleaning tasks with non-convex LL network structure.}\label{tab:hyper_cleaning_all}
	\renewcommand\arraystretch{1.2}
	\setlength{\tabcolsep}{1.3mm}{
		\begin{tabular}{|c|c| c|c|c|c|c| }
			\hline
			\multirow{2}{*} { Method } & \multicolumn{2}{c|} { MNIST } & \multicolumn{2}{c|} { FashionMNIST }  & \multicolumn{2}{c|} { CIFAR10 }\\
			\cline { 2 - 7 }
			& Acc. & F1 score & Acc. & F1 score & Acc. & F1 score\\
			\hline
			CG & $89.19$ & $85.96$ & $83.15$ & $85.13$ &$34.16$&$69.10$\\
			Neumann & $87.54$ & $89.58$ & $81.37$ & $87.28$  &$33.45$&$68.87$\\
			RHG & $87.90$ & $89.36$& $81.91$ & $87.12$  &$34.95$&$68.27$\\
			BDA & $87.15$ & $90.38$&  $79.97$ & $88.24$ &$36.41$&$67.33$\\
			\hline
			$\textrm{AIT}$ & {\color{black}$\mathbf{90.88}$} & {$91.57$}& {$83.67$} & {$90.37$} &{$37.16$}&{$71.57$}\\
			$\textrm{AIT}_{P}$&{$90.41$}&{\color{black}$\mathbf{91.95}$}&{\color{black}$\mathbf{83.80}$}&{\color{black}$\mathbf{90.40}$}&{\color{black}$\mathbf{37.70}$}&{\color{black}$\mathbf{72.74}$}\\
			\hline
		\end{tabular}
	}
\end{table}

\begin{table*}[!htbp]
	\centering
	\caption{Mean test accuracy of 5-way classification on MiniImageNet and TieredImageNet with non-convex LL objective. We use $\pm$ to report the accuracy with $95\%$ confidence intervals over tasks. }
	\label{tab:few_shot}
	\renewcommand\arraystretch{1.2}
	\setlength{\tabcolsep}{1.3mm}{
		\begin{tabular}{  |c | c | c | c | c | c| }
			%\toprule
			\hline
			\multirow{2}{*}{Network} &   \multirow{2}{*}{Methods} & \multicolumn{2}{c|}{MiniImagenet} & \multicolumn{2}{c|}{TieredImagenet} \\
			\cline{3-6}
			& & 5-way 1-shot & 5-way 5-shot & 5-way 1-shot & 5-way 5-shot \\
			\cline{1-6}
			\multirow{5}{*}{\rotatebox{0}{ConvNet-4}} & MAML & $48.70 \pm 0.75\%$ &  $63.11\pm 0.11\%$ & $49.06 \pm 0.50\%$ &  $67.48\pm 0.47\%$\\
			&RHG & $48.89 \pm 0.81\%$ & $63.02 \pm 0.70\%$ & $49.63 \pm 0.67\%$ & $66.14 \pm 0.57\%$\\
			&T-RHG & $47.67 \pm 0.82\%$ & $63.70 \pm 0.76\%$ & $50.79 \pm 0.69\%$ & $67.39 \pm 0.60\%$\\
			&BDA& $49.08 \pm 0.82\%$ & $62.17 \pm 0.70\%$ & $51.56 \pm 0.68\%$ & $\mathbf{68.21 \pm 0.58}\%$\\
			&$\textrm{AIT}$  &$\mathbf{49.80 \pm 0.61\%} $ & $\mathbf{64.76 \pm 0.54\%}$ & $\mathbf{51.86 \pm 0.68\%}$ &$68.01 \pm 0.60\%$\\
			\hline
			\multirow{3}{*}{\rotatebox{0}{ResNet-12}}&MAML & $51.03 \pm 0.50\%$ &  {\color{black}${68.26\pm 0.47}\%$ }& {\color{black}${58.58\pm 0.49}\%$} &  $71.24\pm 0.43\%$  \\
			&RHG &{\color{black}${ 50.54\pm 0.85}\%$} & $64.53 \pm 0.68\%$ & $58.19 \pm 0.76\%$ &{\color{black}${75.20 \pm 0.60}\%$}\\
			&$\textrm{AIT}$ &{\color{black}$\mathbf{56.69 \pm 0.66}\%$} & {\color{black}$\mathbf{70.21 \pm 0.55}\%$} &{\color{black}$\mathbf{60.71 \pm 0.77}\%$}& {\color{black} $\mathbf{75.85 \pm 0.59}\%$} \\
			\hline
		\end{tabular}
	}
\end{table*}

In addition, we introduce the non-convex network structure with two layers of fully-connected network to expand the search space of LL subproblem, where $\y\in\mathbb{R}^{10\time 301}\times\mathbb{R}^{301\times d}$. In Tab.~\ref{tab:hyper_cleaning_all}, we compare our $\textrm{AIT}$ and the variation with prior regularization, (i.e., $\textrm{AIT}_{P}$) with RHG, BDA, CG and Neumann on validation accuracy, F1 score and running time. As it is shown, our method performs the best compared with other GBMs on both accuracy and F1 score. Besides, although only used as a practical technique for solving non-convex BLOs, the embedded proximal prior also consistently improve the performance of $\textrm{AIT}_{P}$ on three datasets.

\subsubsection{Few-Shot Learning}

Few-shot learning applications, more specifically, the few-shot classification tasks~\cite{lecun1998gradient}, train the model to classify unseen instances with only a few samples based on prior training data of similar tasks. As for the N-way K-shot classification, we first define the training dataset as $\mathcal{D}=\{\mathcal{D}^{j}\}$, where $\mathcal{D}^{j}=\mathcal{D}_\mathtt{tr}^{j}\bigcup\mathcal{D}_{\mathtt{val}}^{j}$ corresponds to the meta-train and meta-validation set for the $j\textrm{-th}$ task. In the training phase, we select $K$ samples from each of the $N$ classes in $\mathcal{D}_{\mathtt{tr}}$ for meta training, and use more data from these classes in $\mathcal{D}_{\mathtt{val}}$ for meta validation. Then in the test phase, we construct new tasks on test dataset under the same data distribution to test the performance. 

We follow the commonly used parameter setting for GBMs, where the UL variables $\x$ correspond to the shared hyper representation module model, while the LL variables $\y$ correspond to the task-specific classifier. For this classification problem, the cross-entropy function (denoted as $\ell$) is commonly used for the LL objective $f$ and UL objective $F$.  Besides,~$\ell_{q} $ regularization with $ 0<q<1 $ can be applied to stabilize the training and effectively avoid over-fitting, while existing methods only guarantee the convergence for convex scenario where $q \geq 1$, our AIT supports more flexible non-convex LL objective without the LLC or LLS constraints. Then we add $\|{\y^{j}}\|_{q}$ as the non-convex regularization term to define the LL objectives as 
$F\left(\mathbf{x},\left\{\mathbf{y}^{j}\right\}\right)=\sum_{j} \ell\left(\mathbf{x}, \mathbf{y}^{j}; \mathcal{\D}_{\mathtt{val}}^{j}\right), f\left(\mathbf{x},\left\{\mathbf{y}^{j}\right\}\right)=\sum_{j} \ell\left(\mathbf{x}, \mathbf{y}^{j} ; \mathcal{\D}_{\mathtt{tr}}^{j}\right)+ \|{\y^{j}}\|_{q}.\nonumber
$
The above implementation of non-convex LL objective implies another class of complex LL subproblem for BLOs, which could be used to verify the significant improvement of AIT.

We implement $5$-way $1$-shot and $5$-way $5$-shot classification tasks to validate the performance based on MiniImagenet~\cite{vinyals2016matching} and TieredImagenet~\cite{ren2018meta} datasets. Both datasets are subsets constructed from the larger ILSVRC-12 dataset, while TieredImagenet includes more classes than MiniImagenet. Besides, TieredImagenet categorizes the source data with hierarchical structure to split the training and test datasets and also simulates more real scenarios. 

In Tab.~\ref{tab:few_shot}, we compare our AIT with representative methods including MAML~\cite{finn2017model}, RHG, T-RHG~\cite{shaban2019truncated} and BDA. We also consider two structures as the backbones of hyper representation module, including ConvNet-4 with $4$ layers of convolutional blocks and the larger ResNet-12 with Residual blocks. We use a single fully-connected layer as the task-specific classifier for both backbones. For MAML, the hyper representation module and the classifier are treated as a whole and update as the initialization parameters. We first use ConvNet-4 as the backbone and compare AIT with RHG, T-RHG, BDA and MAML, then we also implement the ResNet-12 backbone and compare AIT with initialization based (i.e., MAML) and recurrence based methods (i.e., RHG). As it is shown in Tab.~\ref{tab:few_shot}, for the 5-way 1-shot and 5-way 5-shot classification tasks on MiniImagenet and TieredImagenet, AIT shows significant performance improvement on the accuracy with different backbones as the hyper representation module.

\begin{figure}[h!]
	\begin{center}
		\renewcommand\arraystretch{0.8}
		\begin{tabular}{c@{\extracolsep{0.1em}}c@{\extracolsep{0.1em}}}
			\includegraphics[height=3.5cm,width=4.5cm,trim=0 0 0. 0,clip]{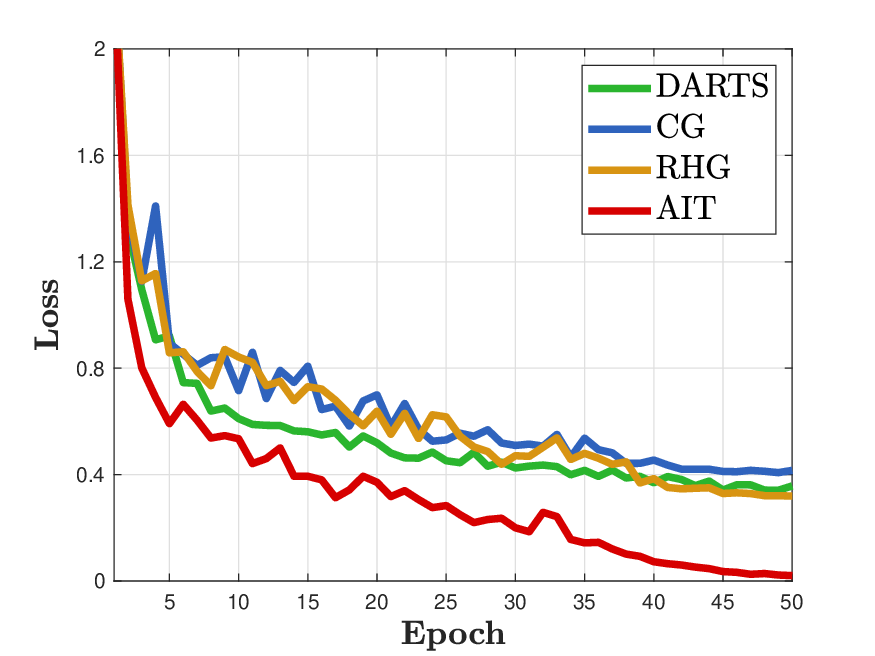}&\includegraphics[height=3.5cm,width=4.5cm,trim=0 0 0. 0,clip]{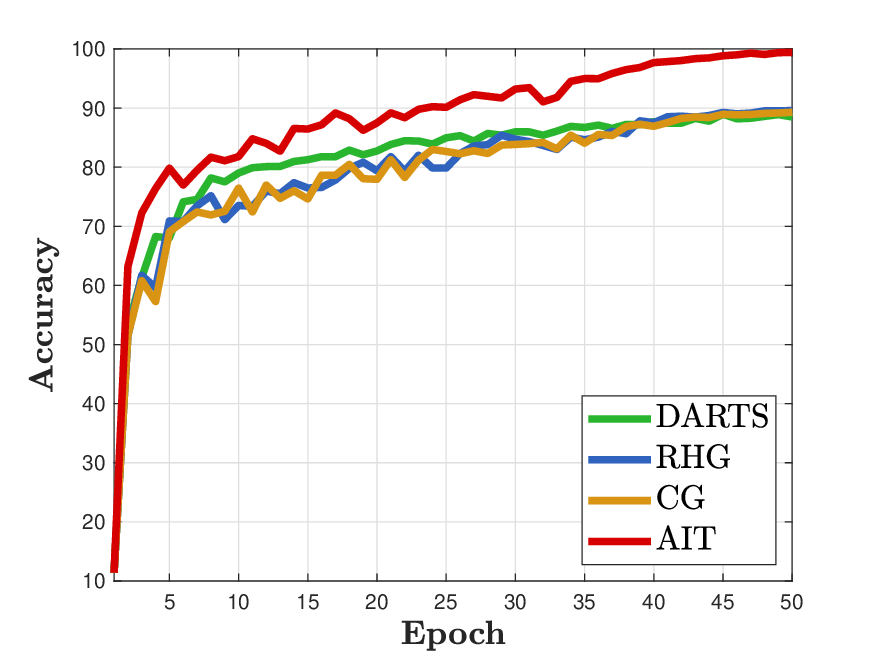}\\
			\specialrule{0em}{0pt}{0pt}
			\footnotesize (a) Validation Loss &\footnotesize (b) Validation Accuracy\\
		\end{tabular}
	\end{center}
	\caption{Comparison of the validation loss and accuracy of DARTS, CG, RHG and AIT in the searching process.}\label{nas_acc}
\end{figure} 

\subsection{Extensions for More Challenging Tasks} 

In this subsection, we first implement AIT to solve large-scale and high-dimensional real-world BLO applications, e.g., neural architecture search. Then we move on to other complex deep learning models, e.g., generative adversarial networks, to demonstrate its generalization performance.   
\subsubsection{Neural Architecture Search}

We denote the training set, validation set, and test set as $(\mathcal{D}_{\mathtt{tr}}, \mathcal{D}_{\mathtt{val}}, \mathcal{D}_{\mathtt{test}})$, then the UL objective and LL objective can be written as $F\left(\x, \y; \mathcal{\D}_{\mathtt{val}}\right)$ and $f\left(\x, \y; \mathcal{\D}_{\mathtt{tr}}\right)$. Then we follow the standard training strategies in DARTS, which consists of the searching and the inference stages. As for the searching stage, DARTS executes a single-step LL optimization (Step \ref{inner_loop1} in Alg.\ref{alg:innerloop}) with $f\left(\mathbf{x}, \mathbf{y}; \mathcal{\D}_{\mathtt{tr}}\right)$, and continuously optimizes the architecture parameter $\x$ (Step \ref{indirect_gradient} in Alg.\ref{alg:innerloop}) to find better cell structures according to the UL objectives  $F\left(\mathbf{x}, \mathbf{y}; \mathcal{\D}_{\mathtt{val}}\right)$. When it comes to the inference stage, the architecture parameter $\x$ is fixed, and we reuse the searched cell structure to construct larger architecture and train it from scratch based on $\mathcal{D}_{\mathtt{tr}}$. Finally, we test the performance of trained model at the inference stage on $\mathcal{D}_{\mathtt{test}}$. In this paper, we conduct the experiments on CIFAR-10 dataset, and use the same search space and similar experimental settings of DARTS. The stacked cell for final architecture has two types, including reduction cells and normal cells. In the reduction cell, operations adjacent to the input nodes use stride 2 to reduce the input shape, while normal cells maintain the original shape with stride 1 for the first two nodes. We use $3$ layers of cells for searching with $50$ epochs and construct larger structure with $8$ layers of the searched cells, and train the model from scratch with 600 epochs in total. 

\begin{figure}[htbp]
	\begin{center}
		\renewcommand\arraystretch{0.1}
		\begin{tabular}{c@{\extracolsep{0.1em}}c@{\extracolsep{0.1em}}c@{\extracolsep{0.1em}}}
			\specialrule{0.01em}{3pt}{3pt}
			\multirow{2}{*}{\rotatebox{90}{Normal Cell~~~~~~~~~}}&&\\
			&\includegraphics[height=1.6cm,width=4cm,trim=0 0 0 0,clip]{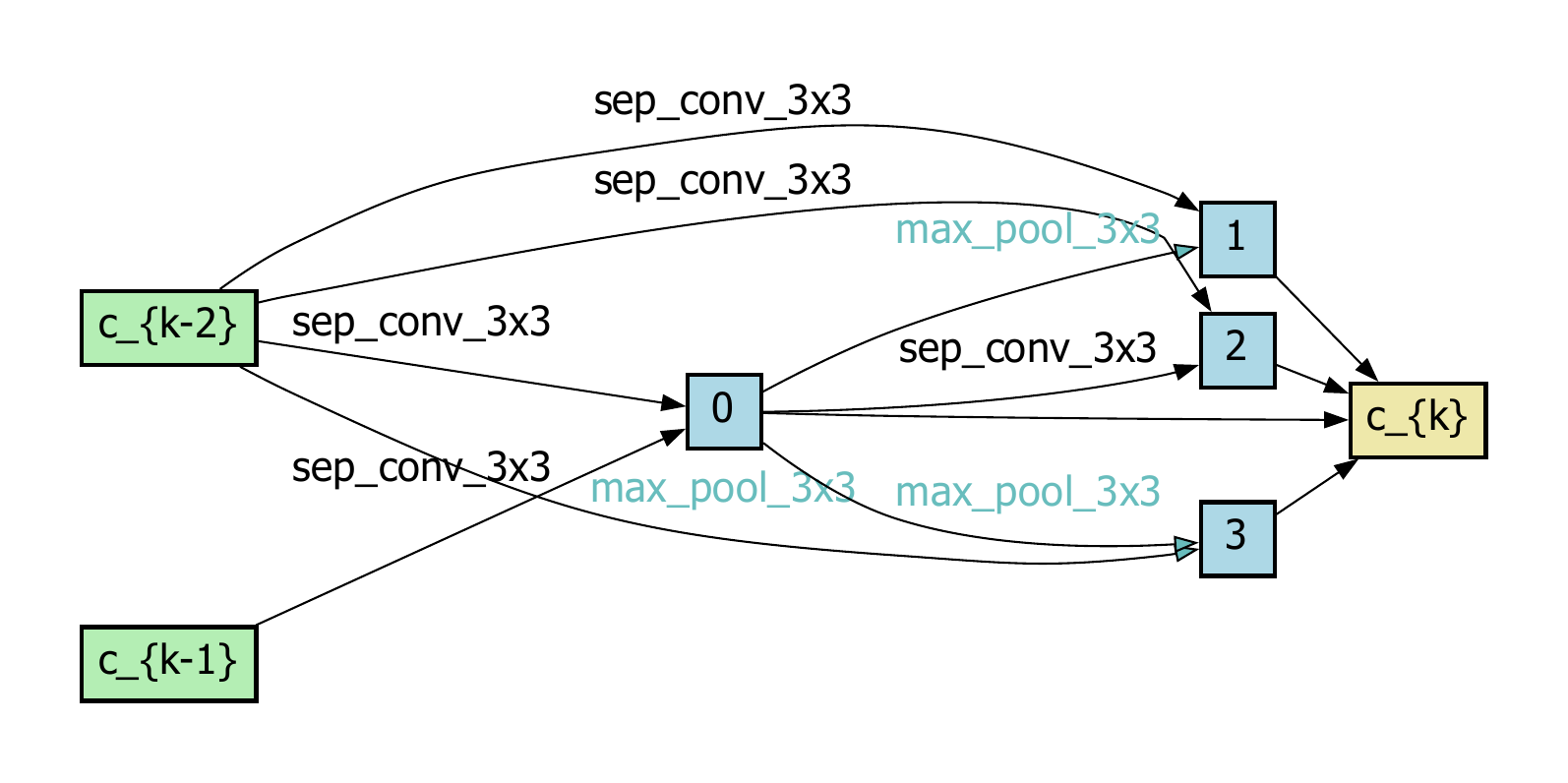} &\includegraphics[height=1.6cm,width=4cm,trim= 0 0 0 0,clip]{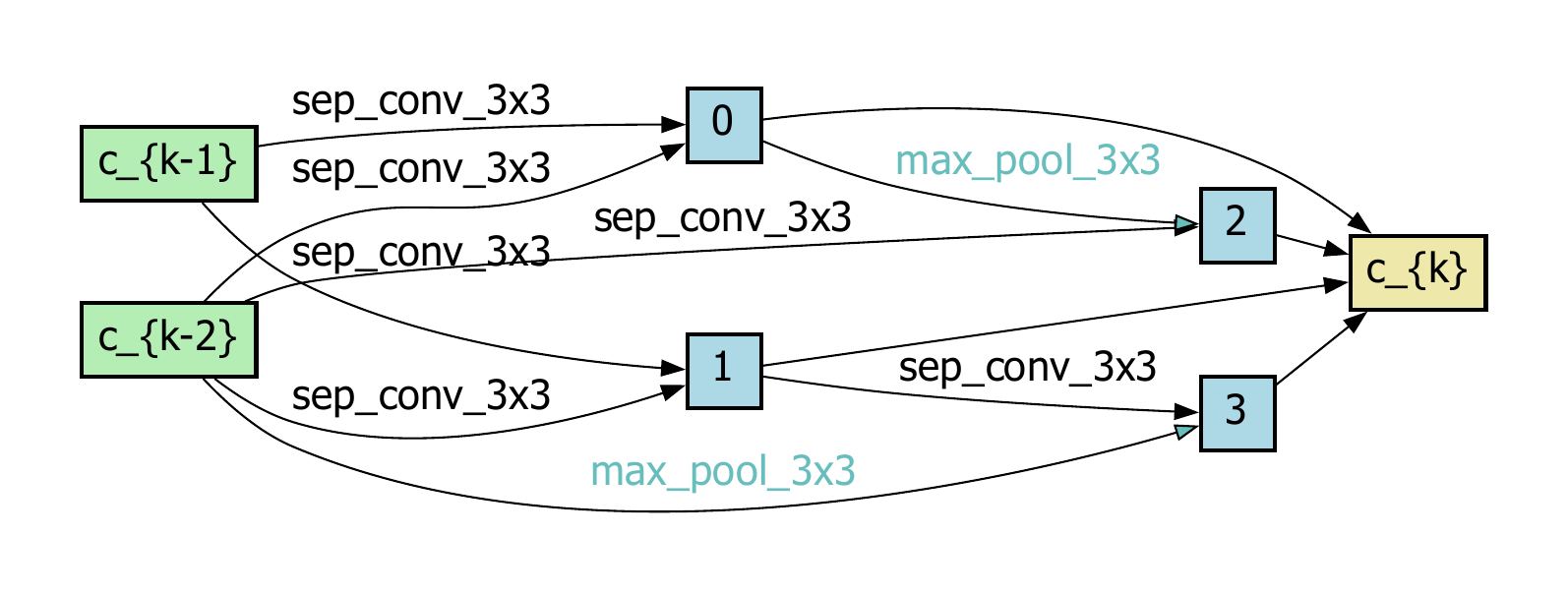}	\\
			&\footnotesize (a) DARTS &\footnotesize (b) RHG \\
			&\includegraphics[height=1.4cm,width=4cm,trim= 0 0 0 0,clip]{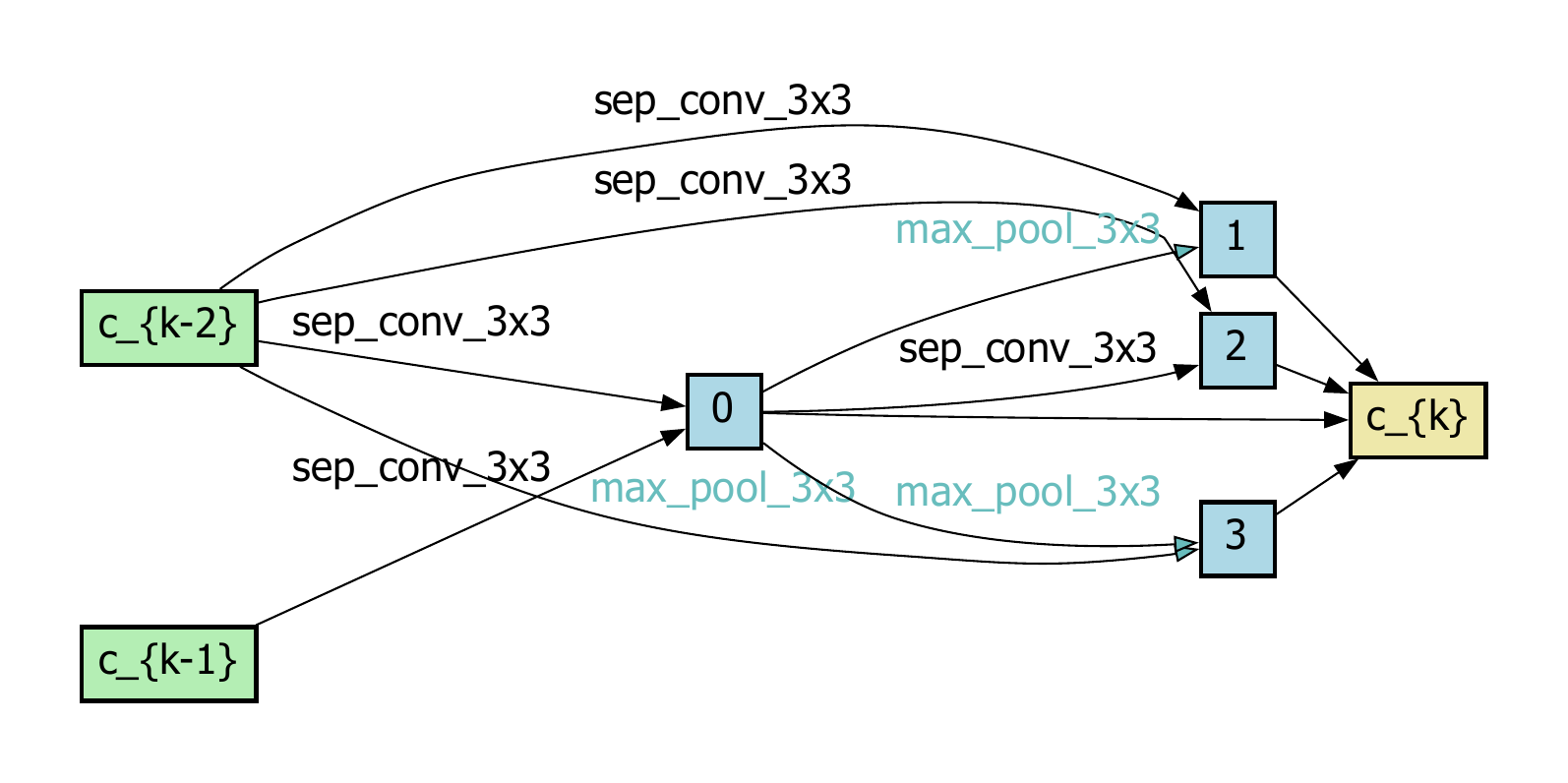}&	\includegraphics[height=1.4cm,width=4cm,trim= 0 0 0 0,clip]{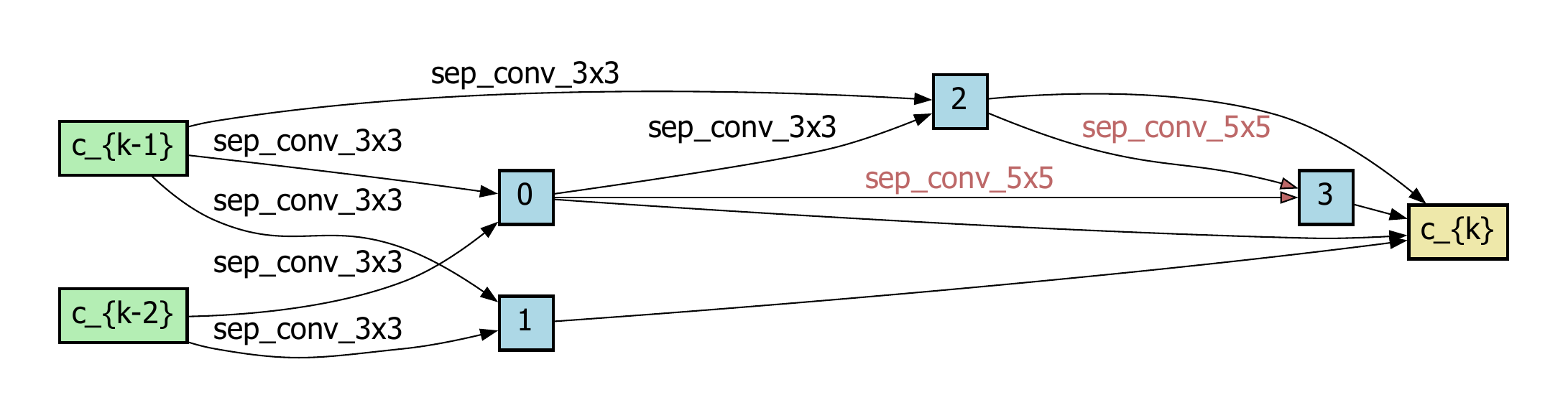}\\
			\multirow{2}{*}{\rotatebox{90}{Reduction Cell~~~~~~~~~~}}&\footnotesize (c) CG &\footnotesize (d) AIT \\
			\specialrule{0.01em}{3pt}{3pt}
			&\includegraphics[height=1.8cm,width=4cm,trim=0 0 0 0,clip]{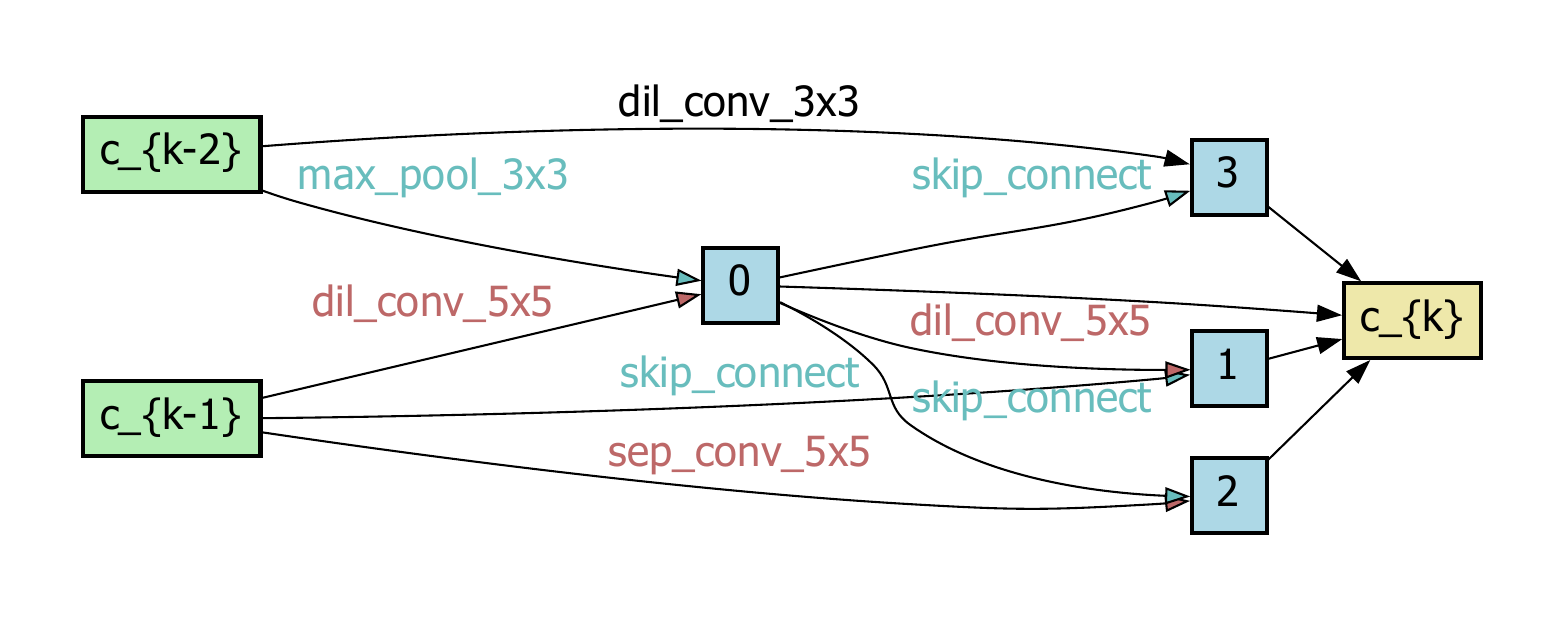}&			\includegraphics[height=1.8cm,width=4cm,trim=0 0 0 0,clip]{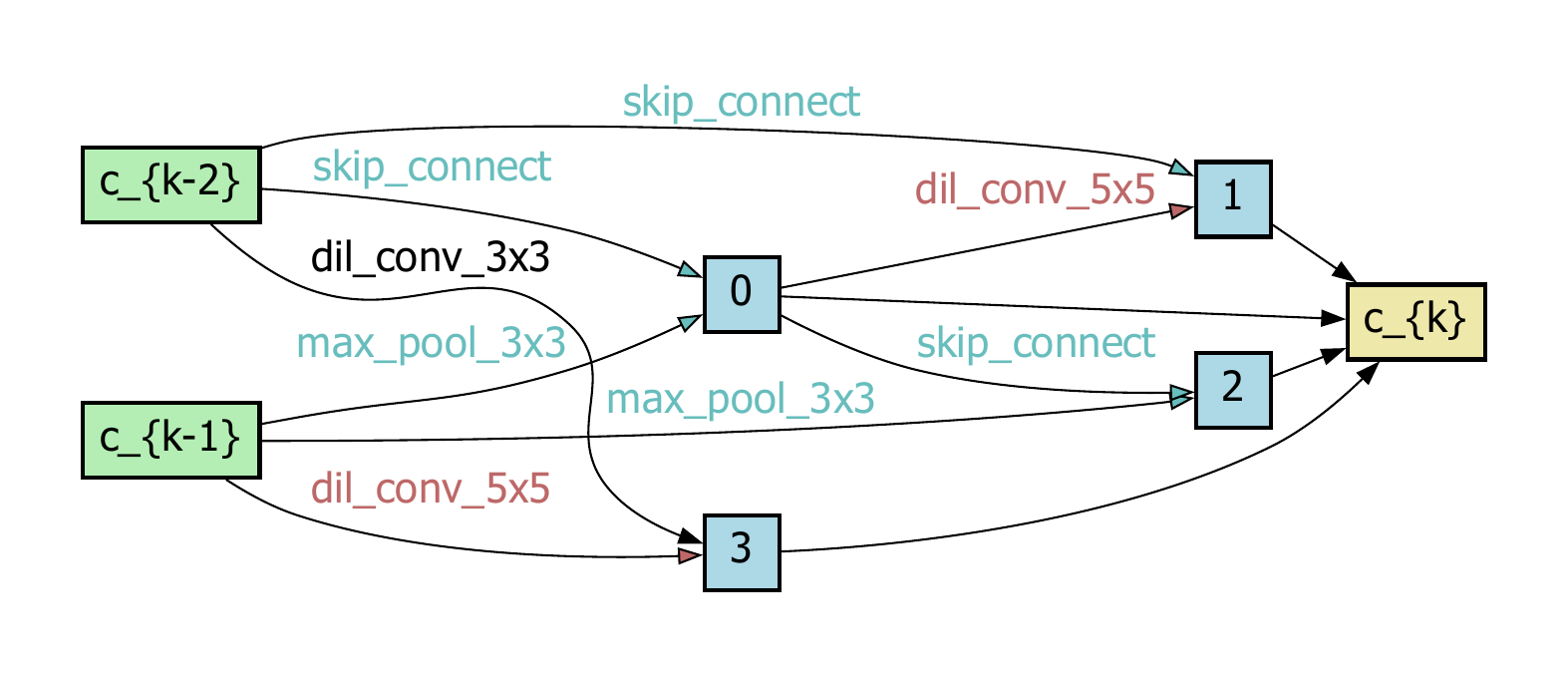}		\\	
			&\footnotesize (a) DARTS &\footnotesize (b) RHG\\
			&\includegraphics[height=1.8cm,width=4cm,trim=0 0 0 0,clip]{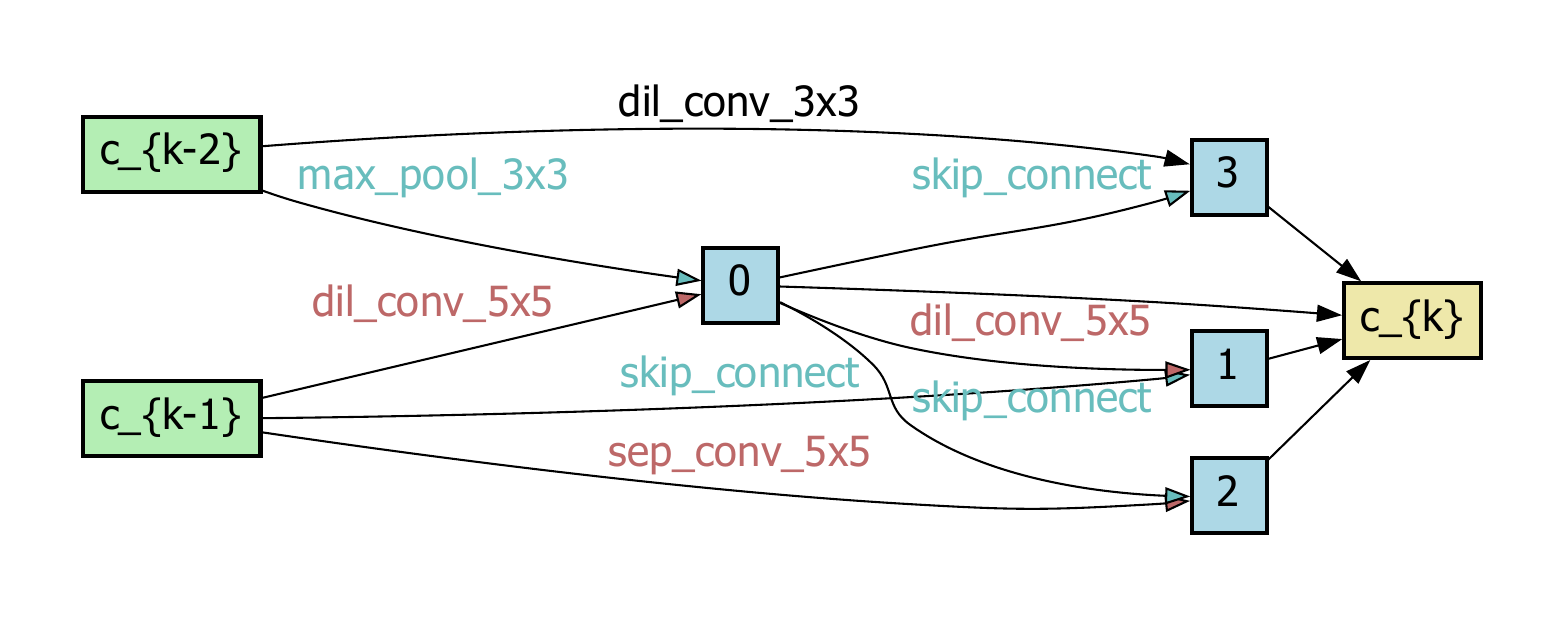}&\includegraphics[height=1.8cm,width=4cm,trim=0 0 0 0,clip]{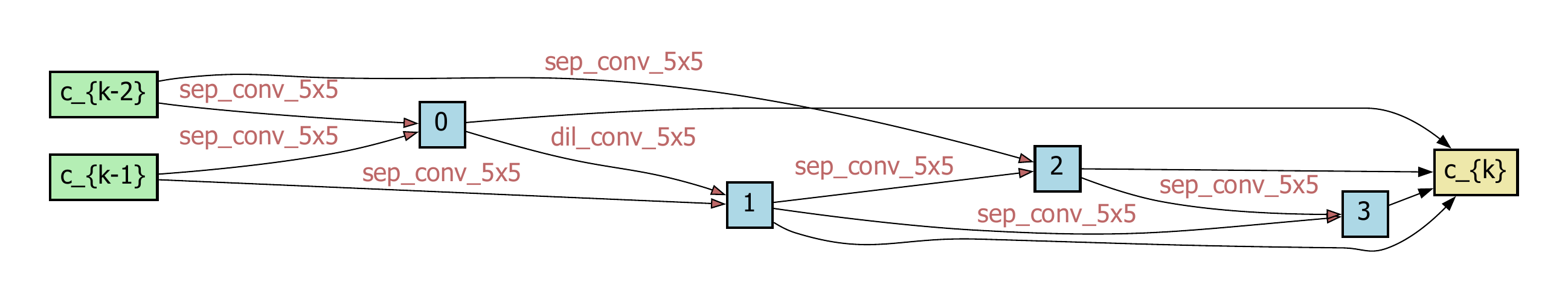}\\
			&\footnotesize (c) CG &\footnotesize (d) AIT\\
			\specialrule{0.01em}{3pt}{3pt}
		\end{tabular}
	\end{center}
	\caption{Visualization results of the searched normal cell and reduction cell for DARTS, RHG, CG and AIT. Note that the edges with blue fonts denote operations with less parameters such as skip connections and pooling operations, and edges with red fonts denote more complex operations such as $5\times5$ separable convolutions.}\label{nas_architect}
\end{figure}

\begin{table}[h!]
	\centering
	\caption{Reporting top 1 accuracy of searching stage, inference stage, final test stage for DARTS, RHG, CG and AIT. We also report the number of parameters of the searched structure (MB) for different methods.}\label{tab:nas_acc}
	\renewcommand\arraystretch{1.4}
	\setlength{\tabcolsep}{1.5mm}{
		\begin{tabular}{|c|c|c|c| c|c| c|}
			\hline
			\multirow{2}{*}{Method}& \multicolumn{2}{c|}{Searching}&\multicolumn{2}{c|}{Inference}&\multirow{2}{*}{Test}&\multirow{2}{*}{Params}  \\
			\cline{2-5}
			&Train &Valid&Train&Valid&  &\\
			\hline
			DARTS &$98.320$&$88.940$&$97.952$&$96.710$&$96.670$&$1.277$\\
			RHG&$98.448$&{$89.556$}&{\color{black}$98.494$}& {\color{black}$97.010$}&{\color{black}$96.920$}&$1.359$\\
			CG &{\color{black}$\mathbf{99.126}$}&$89.298$&$98.286$&$96.690$&$96.640$&$1.268$\\
			\hline
			AIT &$98.904$&{\color{black}$\mathbf{99.512}$}&{\color{black}$\mathbf{98.718}$}& {\color{black}$\mathbf{97.470}$}&{\color{black}$\mathbf{97.330}$}&$\mathbf{1.963}$ \\
			\hline
		\end{tabular}
	}
\end{table}

\begin{figure*}[h!]
	\begin{center}
		\begin{tabular}{c@{\extracolsep{0.5em}}c|@{\extracolsep{2.5em}}@{\extracolsep{1em}}c@{\extracolsep{1.5em}}c@{\extracolsep{1.5em}}c@{\extracolsep{1.5em}}c@{\extracolsep{1.5em}}}
			\rotatebox{90}{~~2D~Pentagram}&\includegraphics[height=2.3cm,width=3cm,trim=0 0 0 30,clip]{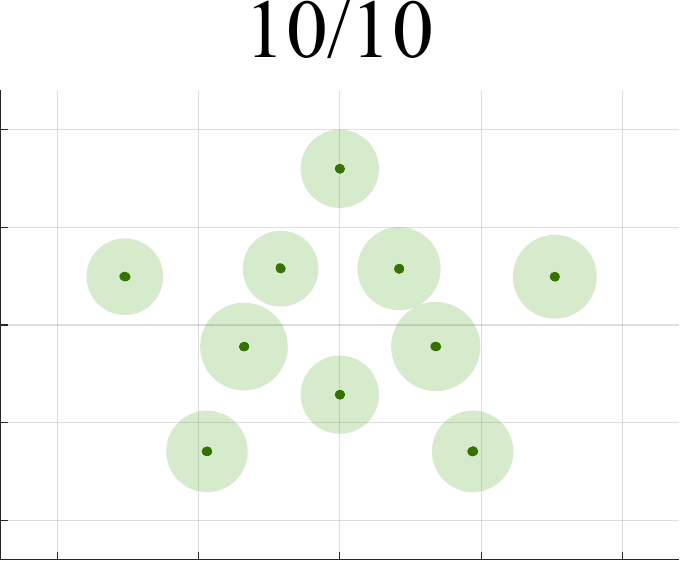} &\includegraphics[height=2.3cm,width=3cm,trim=0 0 0 30,clip]{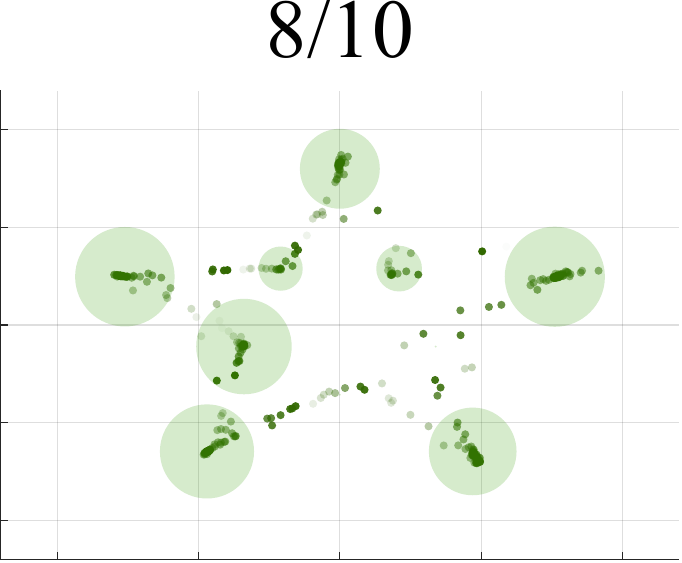}
			& \includegraphics[height=2.3cm,width=3cm,trim=0 0 0 30,clip]{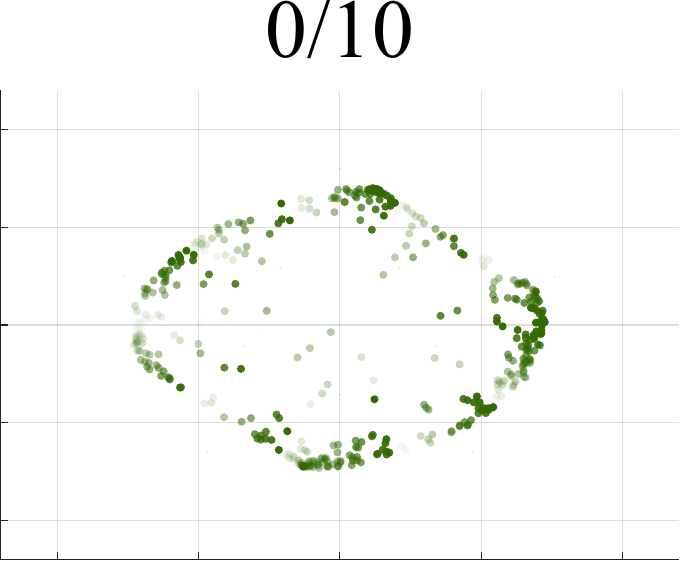} &\includegraphics[height=2.3cm,width=3cm,trim=0 0 0 30,clip]{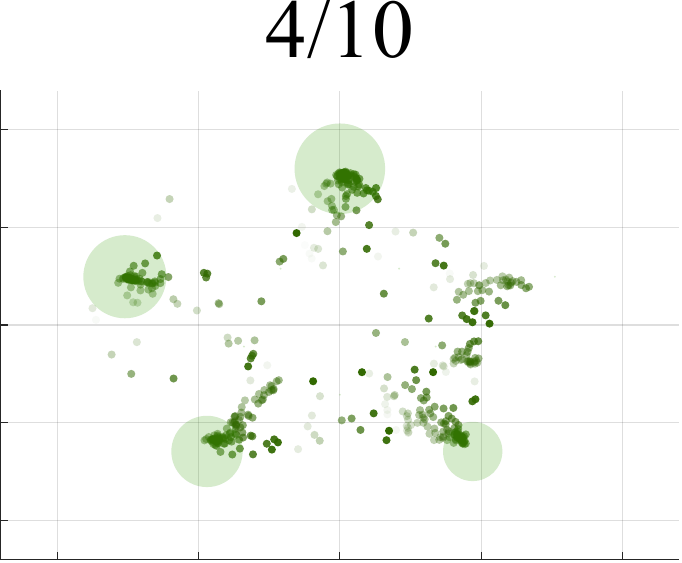}
			&\includegraphics[height=2.3cm,width=3cm,trim=0 0 0 30,clip]{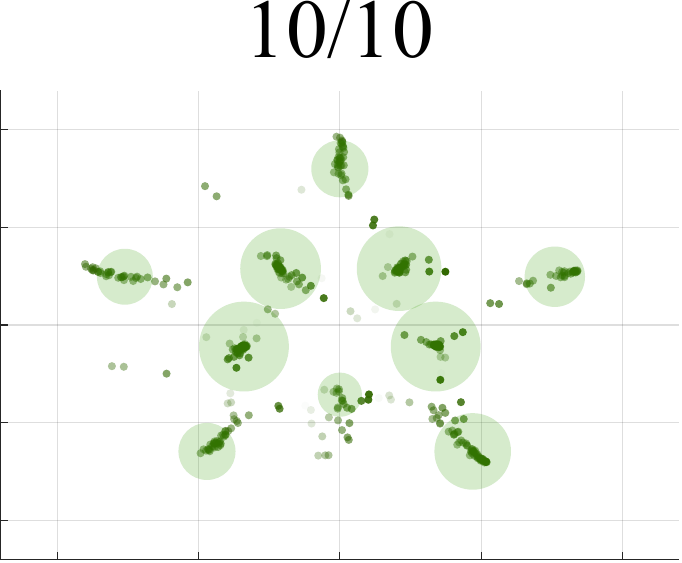}\\
			&Target &(a) VGAN & (b) WGAN	
			&(c) UGAN & (d) AIT\\
		\end{tabular}
	\end{center}
	\caption{We report the results based on the synthesized distribution of 2D Pentagram, and compare AIT with existing GAN methods including VGAN, WGAN and UGAN.}\label{GAN_plot}
\end{figure*}

In Fig.~\ref{nas_acc}, we report the validation loss and accuracy after each epoch to evaluate the convergence behavior of DARTS, RHG, CG and our AIT. It can be seen that RHG, CG and DARTS gain similar validation performance on this non-convex BLO application, while AIT significantly improves the convergence results on $\mathcal{D}_{\textrm{val}}$. Noticed that several works~\cite{chu2020fair, arber2020understanding} found that DARTS may find the global minimum with very small search space but it starts overfitting the architecture parameters on $\mathcal{D}_{\textrm{val}}$, which leads to poor generalization performance. To demonstrate that the searched architecture has consistent performance, in Tab.~\ref{tab:nas_acc}, we also report the accuracy on $\mathcal{D}_{\mathtt{tr}}$ and $\mathcal{D}_{\mathtt{val}}$ at different stages. It can be observed that only AIT maintains better performance on both $\mathcal{D}_{\textrm{tr}}$ and $\mathcal{D}_{\textrm{val}}$, and it also improves the test performance of finally constructed network structure by a large margin.

Besides, it have been investigated~\cite{arber2020understanding} that DARTS tends to choose more parameter-less operations such as skip connection instead of convolution since the skip connection often leads to rapid gradient descent especially on the proxy datasets which are small and easy to fit. As a result, the generated structure may contains less learnable parameters to some extent. In Fig.~\ref{nas_architect}, we illustrate the searched structure for normal cells and reduction cells to show the difference between these GBMs and AIT. As it is depicted, since DARTS derives approximative hyper-gradient w.r.t. $\x$ compared to RHG, both of them found similar structure while RHG uses complex edges with more parameters as reported in Tab.~\ref{tab:nas_acc}. Besides, though CG found significantly different structures from DARTS and RHG, it has no guarantee of convergence and generalization performance on the test set. In comparison, AIT—equipped with convergence guarantees—constructs both normal and reduction cells using more complex separable convolutions, thereby significantly enhancing the expressive power and generalization ability of the resulting architecture.
\subsubsection{Generative Adversarial Networks}

In practice, we conduct the experiments with synthesized 2 dimensional mixture of Gaussian distributions. Specifically, we generate 10 Gaussians to form the shape of pentagram with maximum radius $r=8$. In Fig.~\ref{GAN_plot}, we compare the visualization results generated by the VGAN, Unrolled GAN~\cite{metz2016unrolled} (UGAN), WGAN and AIT. It can be seen that AIT generates more modes than the above methods to fit the target distribution. In comparison, the original GAN methods only capture a small number of distribution, which also shows the effectiveness of AIT to solve GAN based applications. As for the quantitative results, in Tab.~\ref{tab:GAN_results}, we report three commonly used metrics including Frechet Inception Distance (FID)~\cite{goodfellow2014generative}, Jensen-Shannon divergence (JS)~\cite{heusel2017gans}, number of Modes (Mode) . In comparison, AIT not only always generates all the modes with different random seeds, but also gains significant performance improvement on both JS and FID.

\begin{table}[h!]
	%\vspace{-0.4cm} %
	\centering
	\caption{Reporting results of existing methods for the 2D Pentagram Gaussian distribution bu running the experiment with $3$ random seeds. We use different seeds to run the experiments and compare our AIT with VGAN, UGAN and WGAN on three metrics.}\label{tab:GAN_results}
	\renewcommand\arraystretch{1.4}
	\setlength{\tabcolsep}{1.5mm}{
		\begin{tabular}{|c|c|c|c|}
			\hline
			\multirow{2}[0]{*}{\footnotesize Method} & \multicolumn{3}{c|}{2D Pentagram (Max Mode=10)} \\
			\cline{2-4}		
			& \footnotesize  FID  & \footnotesize  JS  & \footnotesize Mode\\
			\hline
			\footnotesize VGAN  &  \footnotesize 1.214 $\pm$ 0.37  &  \footnotesize 0.191 $\pm$ 0.094 &  \footnotesize 8.00 $\pm$ 0.00\\
			\footnotesize WGAN  & \footnotesize  2.404 $\pm$ 3.80 & \footnotesize 0.746 $\pm$ 0.064 & \footnotesize 3.33 $\pm$ 2.52\\
			\footnotesize UGAN  & \footnotesize 2.355 $\pm$ 2.93 & \footnotesize 0.762 $\pm$ 0.083 & \footnotesize 3.33 $\pm$ 2.89\\
			\hline
			\footnotesize AIT  & \textbf{ \footnotesize 0.256 $\pm$ 0.13 } &\textbf{\footnotesize  0.187 $\pm$ 0.046 } & \textbf{ \footnotesize 10.00 $\pm$ 0.00 }\\
			\hline			
		\end{tabular}%
	}	 
\end{table}

\section{Conclusion}
To summarize, our approach involves a thorough analysis of the gradient-based BLO scheme, identifying two significant deficiencies of the existing iterative trajectory. To address these deficiencies, we propose two augmentation techniques, namely Initialization Auxiliary (IA) and Pessimistic Trajectory Truncation (PTT), and develop a series of variations based on these techniques, such as prior regularization and different forms of acceleration dynamics, to construct our Augmented Iterative Trajectory (AIT) for both convex and non-convex scenarios (i.e., $\textrm{AIT}_{C}$ and $\textrm{AIT}_{NC}$). Theoretically, we establish detailed convergence analysis of AIT for different types of BLOs. Finally, we conduct a series of experiments on various BLO settings using numerical examples, and demonstrate the effectiveness of our approach on typical learning and vision tasks as well as more challenging deep learning models such as NAS and GANs.

% if have a single appendix:
%\appendix[Proof of the Zonklar Equations]
% or
%\appendix  % for no appendix heading
% do not use \section anymore after \appendix, only \section*
% is possibly needed

% use appendices with more than one appendix
% then use \section to start each appendix
% you must declare a \section before using any
% \subsection or using \label (\appendices by itself
% starts a section numbered zero.)
%

%\appendices
%\section{Proof of the First Zonklar Equation}
%Appendix one text goes here.
%
%% you can choose not to have a title for an appendix
%% if you want by leaving the argument blank
%\section{}
%Appendix two text goes here.

% use section* for acknowledgment
\ifCLASSOPTIONcompsoc
% The Computer Society usually uses the plural form
\section*{Acknowledgments}
\else
% regular IEEE prefers the singular form
\section*{Acknowledgment}
\fi

This work is partially supported by the National Natural Science Foundation of China (Nos. 62450072, U22B2052, 12326605, 12222106, 62495085, 62027826, 12501429), the Distinguished Youth Funds of the Liaoning Natural Science Foundation (No.2025JH6/101100001), the Distinguished Young Scholars Funds of Dalian (No.2024RJ002),  the Guangdong Basic and Applied Basic Research Foundation (No. 2022B1515020082), the LiaoNing Revitalization Talents Program (No. 2022RG04) and the Fundamental Research Funds for the Central Universities.

% Can use something like this to put references on a page
% by themselves when using endfloat and the captionsoff option.
\ifCLASSOPTIONcaptionsoff
\newpage
\fi

% trigger a \newpage just before the given reference
% number - used to balance the columns on the last page
% adjust value as needed - may need to be readjusted if
% the document is modified later
%\IEEEtriggeratref{8}
% The "triggered" command can be changed if desired:
%\IEEEtriggercmd{\enlargethispage{-5in}}

% references section

% can use a bibliography generated by BibTeX as a .bbl file
% BibTeX documentation can be easily obtained at:
% http://mirror.ctan.org/biblio/bibtex/contrib/doc/
% The IEEEtran BibTeX style support page is at:
% http://www.michaelshell.org/tex/ieeetran/bibtex/
% argument is your BibTeX string definitions and bibliography database(s)
%\bibliography{IEEEabrv,../bib/paper}
%
% <OR> manually copy in the resultant .bbl file
% set second argument of \begin to the number of references
% (used to reserve space for the reference number labels box)
\bibliographystyle{IEEEtran}
\bibliography{reference}

@article{liu2021investigating,
	title={Investigating bi-level optimization for learning and vision from a unified perspective: A survey and beyond},
	author={Liu, Risheng and Gao, Jiaxin and Zhang, Jin and Meng, Deyu and Lin, Zhouchen},
	journal={IEEE Transactions on Pattern Analysis and Machine Intelligence},
	year={2021},
	publisher={IEEE}
}

@article{xiao2022alternating,
	title={Alternating Implicit Projected SGD and Its Efficient Variants for Equality-constrained Bilevel Optimization},
	author={Xiao, Quan and Shen, Han and Yin, Wotao and Chen, Tianyi},
	journal={arXiv preprint arXiv:2211.07096},
	year={2022}
}

@inproceedings{ye2021difference,
	title={Value function based difference-of-convex algorithm for bilevel hyperparameter selection problems},
	author={Gao, Lucy L and Ye, Jane and Yin, Haian and Zeng, Shangzhi and Zhang, Jin},
	booktitle={International Conference on Machine Learning},
	pages={7164--7182},
	year={2022}
}

@inproceedings{khanduri2021momentum,
	title={A momentum-assisted single-timescale stochastic approximation algorithm for bilevel optimization},
	author={Khanduri, Prashant and Zeng, Siliang and Hong, Mingyi and Wai, Hoi-To and Wang, Zhaoran and Yang, Zhuoran},
	booktitle={Advances in Neural Information Processing Systems},
	year={2021}
}

@article{hong2020two,
	title={A two-timescale framework for bilevel optimization: Complexity analysis and application to actor-critic},
	author={Hong, Mingyi and Wai, Hoi-To and Wang, Zhaoran and Yang, Zhuoran},
	journal={arXiv preprint arXiv:2007.05170},
	year={2020}
}

@inproceedings{ji2020provably,
	title={Provably faster algorithms for bilevel optimization and applications to meta-learning},
	author={Ji, Kaiyi and Yang, Junjie and Liang, Yingbin},
	booktitle={International Conference on Learning Representations},
	year={2021}
}

@inproceedings{chen2022single,
	title={A single-timescale method for stochastic bilevel optimization},
	author={Chen, Tianyi and Sun, Yuejiao and Xiao, Quan and Yin, Wotao},
	booktitle={International Conference on Artificial Intelligence and Statistics},
	pages={2466--2488},
	year={2022}
}

@inproceedings{dong2019searching,
	title={Searching for a robust neural architecture in four gpu hours},
	author={Dong, Xuanyi and Yang, Yi},
	booktitle={IEEE/CVF Conference on Computer Vision and Pattern Recognition},
	pages={1761--1770},
	year={2019}
}

@inproceedings{xu2019pc,
	title={Pc-darts: Partial channel connections for memory-efficient architecture search},
	author={Xu, Yuhui and Xie, Lingxi and Zhang, Xiaopeng and Chen, Xin and Qi, Guo-Jun and Tian, Qi and Xiong, Hongkai},
	booktitle={International Conference on Learning Representations},
	year={2020}
}

@inproceedings{chen2019progressive,
	title={Progressive differentiable architecture search: Bridging the depth gap between search and evaluation},
	author={Chen, Xin and Xie, Lingxi and Wu, Jun and Tian, Qi},
	booktitle={IEEE/CVF International Conference on Computer Vision},
	pages={1294--1303},
	year={2019}
}

@inproceedings{liu2021towards,
  title={Towards gradient-based bilevel optimization with non-convex followers and beyond},
  author={Liu, Risheng and Liu, Yaohua and Zeng, Shangzhi and Zhang, Jin},
  booktitle={Advances in Neural Information Processing Systems},
  volume={34},
  pages={8662--8675},
  year={2021}
}

@inproceedings{liu2020generic,
  title={A generic first-order algorithmic framework for bi-level programming beyond lower-level singleton},
  author={Liu, Risheng and Mu, Pan and Yuan, Xiaoming and Zeng, Shangzhi and Zhang, Jin},
  booktitle={International Conference on Machine Learning},
  pages={6305--6315},
  year={2020}
}

@article{liu2022general,
  title={A general descent aggregation framework for gradient-based bi-level optimization},
  author={Liu, Risheng and Mu, Pan and Yuan, Xiaoming and Zeng, Shangzhi and Zhang, Jin},
  journal={IEEE Transactions on Pattern Analysis and Machine Intelligence},
  year={2022},
  publisher={IEEE}
}

@inproceedings{LiuLYZZ21,
	title     = {A Value-Function-based Interior-point Method for Non-convex Bi-level Optimization},
	author    = {Liu, Risheng and Liu, Xuan and Yuan, Xiaoming and Zeng, Shangzhi and Zhang, Jin},
	booktitle = {International Conference on Machine Learning},
	volume    = {139},
	pages     = {6882--6892},
	year      = {2021}
}

@article{liu2021value,
	author={Liu, Risheng and Liu, Xuan and Zeng, Shangzhi and Zhang, Jin and Zhang, Yixuan},
	journal={IEEE Transactions on Pattern Analysis and Machine Intelligence}, 
	title={Value-function-based sequential minimization for bi-level optimization}, 
	year={2023},
	volume={45},
	number={12},
	pages={15930-15948}}

@inproceedings{heusel2017gans,
  title={{GANs} trained by a two time-scale update rule converge to a local nash equilibrium},
  author={Heusel, Martin and Ramsauer, Hubert and Unterthiner, Thomas and Nessler, Bernhard and Hochreiter, Sepp},
  booktitle={Advances in Neural Information Processing Systems},
  year={2017}
}

@inproceedings{ArjovskyWGAN,
  author = {Arjovsky, Martin and Chintala, Soumith and Bottou, Léon},
  title = {Wasserstein {GAN}},
  booktitle={International Conference on Machine Learning},
  year = {2017}
}

@inproceedings{metz2016unrolled,
  title={Unrolled generative adversarial networks},
  author={Metz, Luke and Poole, Ben and Pfau, David and Sohl-Dickstein, Jascha},
  booktitle={International Conference on Learning Representations},
  year={2017}
}

@article{ghadimi2016accelerated,
  title={Accelerated gradient methods for nonconvex nonlinear and stochastic programming},
  author={Ghadimi, Saeed and Lan, Guanghui},
  journal={Mathematical Programming},
  volume={156},
  number={1},
  pages={59--99},
  year={2016},
  publisher={Springer}
}

@inproceedings{wu2019fbnet,
	title={Fbnet: Hardware-aware efficient convnet design via differentiable neural architecture search},
	author={Wu, Bichen and Dai, Xiaoliang and Zhang, Peizhao and Wang, Yanghan and Sun, Fei and Wu, Yiming and Tian, Yuandong and Vajda, Peter and Jia, Yangqing and Keutzer, Kurt},
	booktitle={IEEE Conference on Computer Vision and Pattern Recognition},
	pages={10734--10742},
	year={2019}
}

@inproceedings{hu2020tf,
	title={Tf-nas: Rethinking three search freedoms of latency-constrained differentiable neural architecture search},
	author={Hu, Yibo and Wu, Xiang and He, Ran},
	booktitle={European Conference on Computer Vision},
	year={2020}
}

@inproceedings{wang2020global,
	title={On the Global Optimality of Model-Agnostic Meta-Learning: Reinforcement Learning and Supervised Learning},
	author={Wang, Lingxiao and Cai, Qi and Yang, Zhuoyan and Wang, Zhaoran},
	booktitle={International Conference on Machine Learning},
	year={2020}
}

@inproceedings{zhang2020bi,
	title={Bi-level actor-critic for multi-agent coordination},
	author={Zhang, Haifeng and Chen, Weizhe and Huang, Zeren and Li, Minne and Yang, Yaodong and Zhang, Weinan and Wang, Jun},
	booktitle={AAAI Conference on Artificial Intelligence},
	volume={34},
	number={05},
	pages={7325--7332},
	year={2020}
}

@inproceedings{goodfellow2014generative,
  title={Generative adversarial nets},
  author={Goodfellow, Ian and Pouget-Abadie, Jean and Mirza, Mehdi and Xu, Bing and Warde-Farley, David and Ozair, Sherjil and Courville, Aaron and Bengio, Yoshua},
  booktitle={Advances in Neural Information Processing Systems},
  volume={27},
  year={2014}
}

@article{beck2009fast,
	title={A Fast Iterative Shrinkage-Thresholding Algorithm for Linear Inverse Problems},
	author={Beck, Amir and Teboulle, Marc},
	journal={SIAM Journal on Imaging Sciences},
	volume={2},
	number={1},
	pages={183--202},
	year={2009},
	publisher={Society for Industrial and Applied Mathematics}
}

@book{beck2017first,
	title={First-Order Methods in Optimization},
	author={Beck, Amir},
	year={2017},
	publisher={Society for Industrial and Applied Mathematics},
	address={Philadelphia, PA, USA}
}

@inproceedings{grazzi2020iteration,
	title={On the iteration complexity of hypergradient computation},
	author={Grazzi, Riccardo and Franceschi, Luca and Pontil, Massimiliano and Salzo, Saverio},
	booktitle={International Conference on Machine Learning},
	year={2020},
}

@inproceedings{franceschi2018bilevel,
  title={Bilevel programming for hyperparameter optimization and meta-learning},
  author={Franceschi, Luca and Frasconi, Paolo and Salzo, Saverio and Grazzi, Riccardo and Pontil, Massimiliano},
  booktitle={International Conference on Machine Learning},
  year={2018}
}

@inproceedings{liu2018darts,
	author  = {Hanxiao Liu and Karen Simonyan and Yiming Yang},
	title = {{DARTS:} Differentiable Architecture Search},
	booktitle = {International Conference on Learning Representations},
	year = {2019}
}

@inproceedings{franceschi2017forward,
	title={Forward and reverse gradient-based hyperparameter optimization},
	author={Franceschi, Luca and Donini, Michele and Frasconi, Paolo and Pontil, Massimiliano},
	booktitle={International Conference on Machine Learning},
	year={2017}
}

@inproceedings{rajeswaran2019meta,
	title={Meta-learning with implicit gradients},
	author={Rajeswaran, Aravind and Finn, Chelsea and Kakade, Sham M and Levine, Sergey},
	booktitle={Advances in Neural Information Processing Systems}, 	
	year={2019}
}

@book{moore2010bilevel,
  title={Bilevel programming algorithms for machine learning model selection},
  author={Moore, Gregory M},
  year={2010},
  publisher={Rensselaer Polytechnic Institute}
}

@inproceedings{shaban2019truncated,
	title={Truncated back-propagation for bilevel optimization},
	author={Shaban, Amirreza and Cheng, Ching-An and Hatch, Nathan and Boots, Byron},
	booktitle={International Conference on Artificial Intelligence and Statistics},
	year={2019}
}

@inproceedings{pedregosa2016hyperparameter,
	title={Hyperparameter optimization with approximate gradient},
	author={Pedregosa, Fabian},
	booktitle={ International Conference on Machine Learning},
	year={2016}
}

@inproceedings{lorraine2020optimizing,
	title={Optimizing millions of hyperparameters by implicit differentiation},
	author={Lorraine, Jonathan and Vicol, Paul and Duvenaud, David},
	booktitle={International Conference on Artificial Intelligence and Statistics}, 	
	year={2020}
}

@article{lecun1998gradient,
  title={Gradient-based learning applied to document recognition},
  author={LeCun, Yann and Bottou, L{\'e}on and Bengio, Yoshua and Haffner, Patrick},
  journal={Proceedings of the IEEE},
  volume={86},
  number={11},
  pages={2278--2324},
  year={1998},
  publisher={Ieee}
}

@inproceedings{vinyals2016matching,
	title={Matching networks for one shot learning},
	author={Vinyals, Oriol and Blundell, Charles and Lillicrap, Timothy and Wierstra, Daan and others},
	booktitle={Advances in Neural Information Processing Systems},
	year={2016}
}

@inproceedings{maclaurin2015gradient,
	title={Gradient-based hyperparameter optimization through reversible learning},
	author={Maclaurin, Dougal and Duvenaud, David and Adams, Ryan},
	booktitle={International Conference on Machine Learning}, 	
	year={2015}
}

@inproceedings{finn2017model,
	title={Model-agnostic meta-learning for fast adaptation of deep networks},
	author={Finn, Chelsea and Abbeel, Pieter and Levine, Sergey},
	booktitle={International Conference on Machine Learning}, 	
	year={2017}
}

@article{nichol2018first,
	title={On first-order meta-learning algorithms},
	author={Nichol, Alex and Achiam, Joshua and Schulman, John},
	journal={ arXiv:1803.02999},
	year={2018}
}

@inproceedings{ren2018meta,
  title={Meta-learning for semi-supervised few-shot classification},
  author={Ren, Mengye and Triantafillou, Eleni and Ravi, Sachin and Snell, Jake and Swersky, Kevin and Tenenbaum, Joshua B and Larochelle, Hugo and Zemel, Richard S},
  booktitle={International Conference on Learning Representations},
  year={2018}
}

@article{liu2022bome,
	title={Bome! bilevel optimization made easy: A simple first-order approach},
	author={Liu, Bo and Ye, Mao and Wright, Stephen and Stone, Peter and Liu, Qiang},
	journal={Advances in Neural Information Processing Systems},
	volume={35},
	pages={17248--17262},
	year={2022}
}

@inproceedings{ji2021bilevel,
	title={Bilevel optimization: Convergence analysis and enhanced design},
	author={Ji, Kaiyi and Yang, Junjie and Liang, Yingbin},
	booktitle={International Conference on Machine Learning},
	pages={4882--4892},
	year={2021}
}

@article{zemkoho2020theoretical,
	title={Theoretical and numerical comparison of the Karush--Kuhn--Tucker and value function reformulations in bilevel optimization},
	author={Zemkoho, Alain B and Zhou, Shenglong},
	journal={Computational Optimization and Applications},
	pages={1--50},
	year={2020}
}

@inproceedings{chu2020fair,
  title={Fair {DARTS}: Eliminating unfair advantages in differentiable architecture search},
  author={Chu, Xiangxiang and Zhou, Tianbao and Zhang, Bo and Li, Jixiang},
  booktitle={European Conference on Computer Vision},
  pages={465--480},
  year={2020},
  organization={Springer}
}

@inproceedings{arber2020understanding,
  title={Understanding and robustifying differentiable architecture search},
  author={Arber Zela, Thomas Elsken and Saikia, Tonmoy and Marrakchi, Yassine and Brox, Thomas and Hutter, Frank},
  booktitle={International Conference on Learning Representations},
  year={2020}
}

@Book{rockafellar2009variational,
	title = "Variational analysis",
	author = "Rockafellar, R Tyrrell and Wets, Roger J-B",
	publisher = "Springer Science \& Business Media",
	year = "2009"
}

@inproceedings{shen2023penalty,
	title={On penalty-based bilevel gradient descent method},
	author={Shen, Han and Chen, Tianyi},
	booktitle={International Conference on Machine Learning},
	pages={30992--31015},
	year={2023}
}

@inproceedings{huang2024optimal,
	title={Optimal Hessian/Jacobian-Free Nonconvex-PL Bilevel Optimization},
	author={Huang, Feihu},
	booktitle={International Conference on Machine Learning},
	pages={19598--19621},
	year={2024}
}

\begin{IEEEbiography}[{\includegraphics[width=1in,height=1.25in,clip,keepaspectratio]{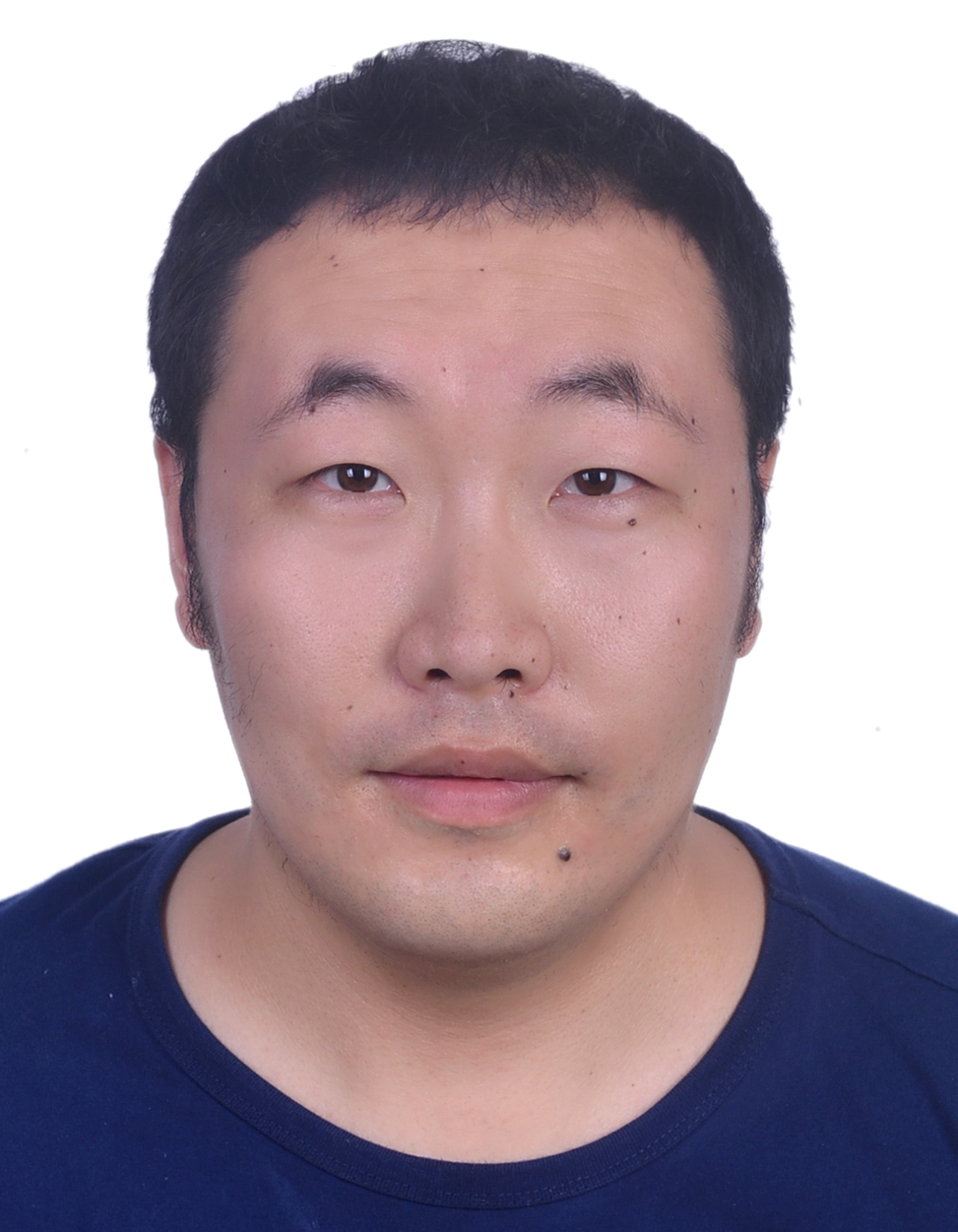}}]{Risheng Liu} received the B.S. and Ph.D. degrees both in mathematics from the Dalian University of Technology in 2007 and 2012, respectively. He was a visiting scholar in the Robotic Institute of Carnegie Mellon University from 2010 to 2012. He served as Hong Kong Scholar Research Fellow at the Hong Kong Polytechnic University from 2016 to 2017. He is currently a professor with DUT-RU International School of Information Science \& Engineering, Dalian University of Technology. His research interests include machine learning, optimization, computer vision and multimedia.
\end{IEEEbiography}

\begin{IEEEbiography}[{\includegraphics[width=1in,height=1.25in,clip,keepaspectratio]{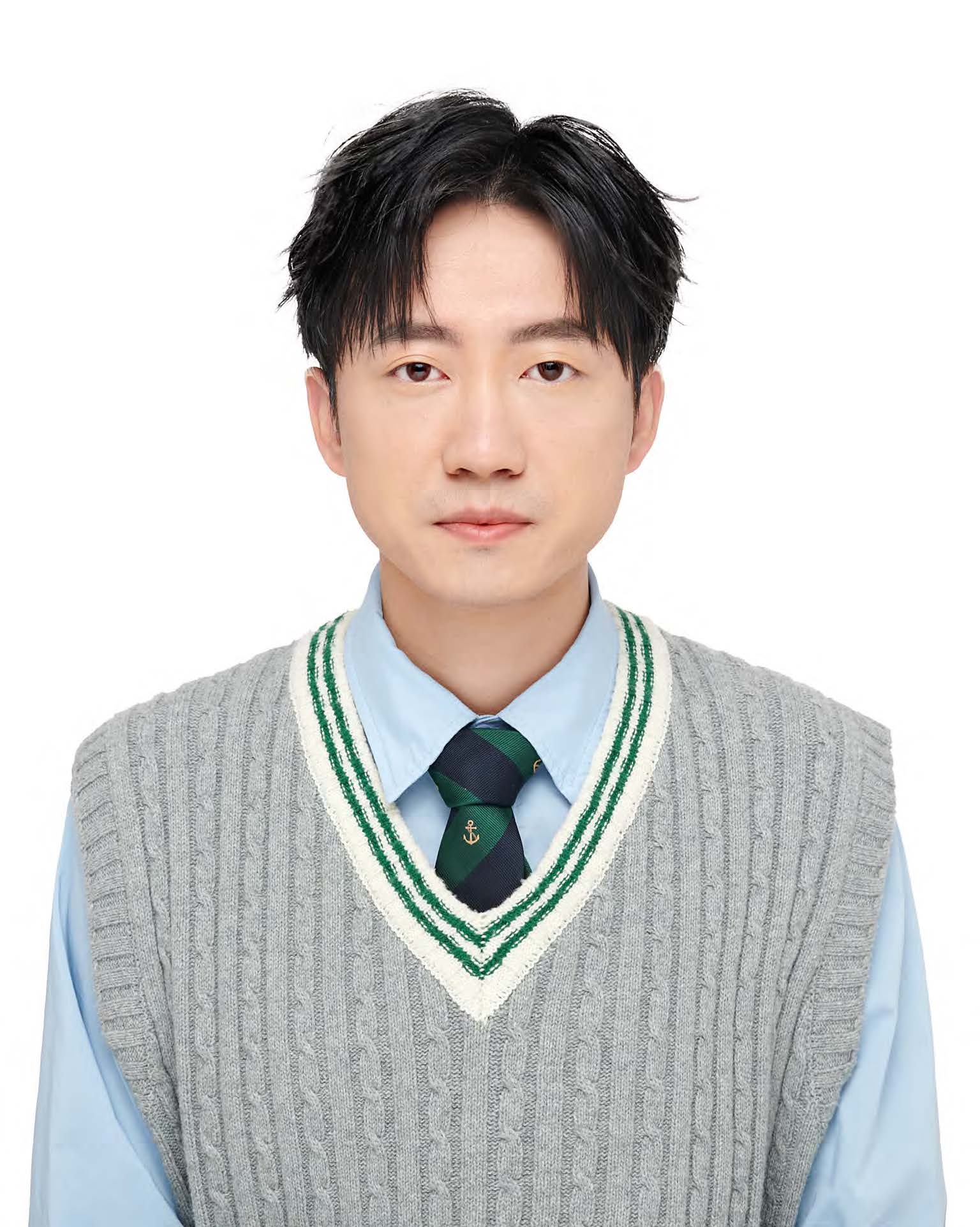}}]{Yaohua Liu}
	received the B.Eng. degree in Software Engineering from the School of Software, Dalian University of Technology, China, in 2019, and the M.Eng. and Ph.D. degrees in Software Engineering from the DUT-RU International School of Information Science and Engineering, Dalian University of Technology, China, in 2021 and 2025, respectively. He is currently a Postdoctoral Fellow with the School of Computing and Data Science, The University of Hong Kong. His research interests include bilevel optimization, machine learning, and trustworthy artificial intelligence.
\end{IEEEbiography}

\begin{IEEEbiography}[{\includegraphics[width=1in,height=1.25in,clip,keepaspectratio]{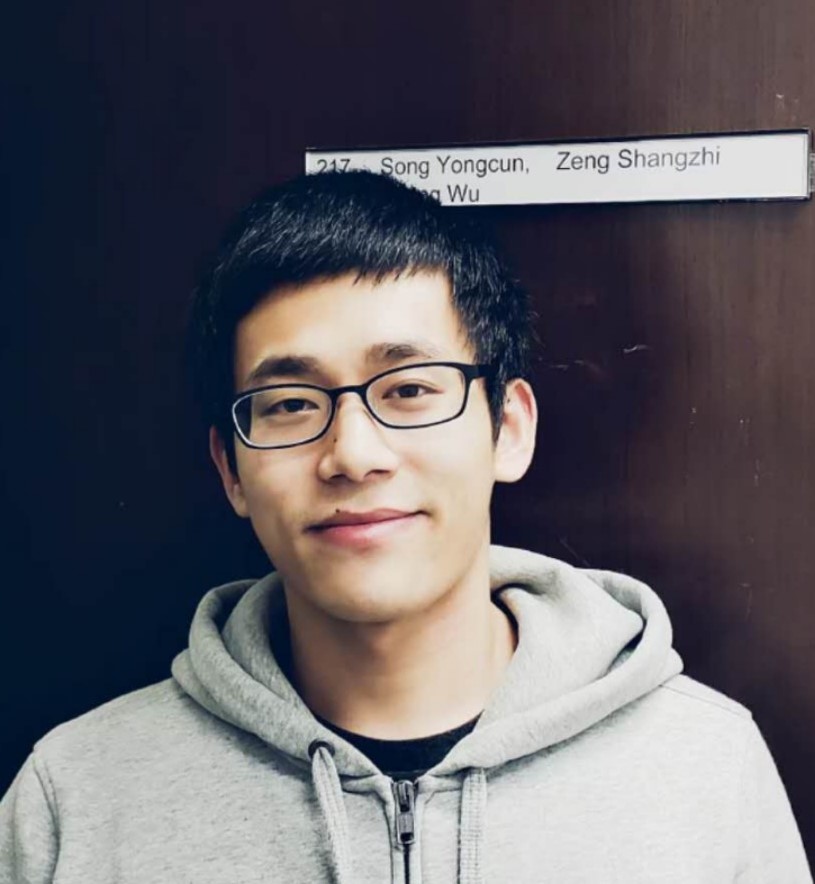}}]{Shangzhi Zeng} received the B.Sc. degree in Mathematics and Applied Mathematics from Wuhan University in 2015, the M.Phil. degree in Mathematics from Hong Kong Baptist University in 2018, and the Ph.D. degree in Mathematics from The University of Hong Kong in 2021. From 2021 to 2024, he was a Postdoctoral Fellow with the Department of Mathematics and Statistics, University of Victoria, Canada. He is currently an Associate Professor with the National Center for Applied Mathematics and the Department of Mathematics, Southern University of Science and Technology, Shenzhen, China. His research interests include optimization theory and methods, bilevel programming, and machine learning optimization algorithms.	

\end{IEEEbiography}

\begin{IEEEbiography}[{\includegraphics[width=1in,height=1.32in,clip,keepaspectratio]{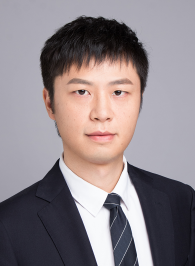}}]{Jin Zhang} received the B.A. degree in Journalism from the Dalian University of Technology in 2007. He received the M.S. degree in mathematics from the Dalian University of Technology, China, in 2010, and the PhD degree in Applied Mathematics from University of Victoria, Canada, in 2015. He worked as a Research Assistant Professor at Hong Kong Baptist University from 2015 to 2019. He is currently a Professor with the Department of Mathematics and the National Center for Applied Mathematics, Southern University of Science and Technology, Shenzhen, China. His research interests include optimization, variational analysis and their applications in economics, engineering and data science.
\end{IEEEbiography}

% simpler here.)
%\begin{IEEEbiography}[{\includegraphics[width=1in,height=1.25in,clip,keepaspectratio]{mshell}}]{Michael Shell}
% or if you just want to reserve a space for a photo:

%\begin{IEEEbiography}{Michael Shell}
%	Biography text here.
%\end{IEEEbiography}
%
%% if you will not have a photo at all:
%\begin{IEEEbiographynophoto}{John Doe}
%	Biography text here.
%\end{IEEEbiographynophoto}
%
%% insert where needed to balance the two columns on the last page with
%% biographies
%%\newpage
%
%\begin{IEEEbiographynophoto}{Jane Doe}
%	Biography text here.
%\end{IEEEbiographynophoto}
		
		% You can push biographies down or up by placing
		% a \vfill before or after them. The appropriate
		% use of \vfill depends on what kind of text is
		% on the last page and whether or not the columns
		% are being equalized.
		
		%\vfill
		
		% Can be used to pull up biographies so that the bottom of the last one
		% is flush with the other column.
		%\enlargethispage{-5in}

		% that's all folks
	\end{document}